\documentclass[11pt,a4paper]{article}
\usepackage{amsfonts,latexsym,amsmath, amssymb, amscd, amsthm}
\usepackage{mathrsfs} 
\usepackage{subfigure} 
\usepackage{bm}
\usepackage{array,float,adjustbox}    
\usepackage{appendix}
\usepackage{color}
\usepackage[top=1.2in, bottom=1.2in, left=1.25in, right=1.25in]{geometry}
\usepackage{cite}
\usepackage{cases}
\usepackage[utf8]{inputenc}
\usepackage{graphicx,float,enumerate}
\usepackage[T1]{fontenc}
\usepackage{authblk}
\numberwithin{equation}{section}
\usepackage[colorlinks=true,citecolor=red,linkcolor=black,urlcolor=RubineRed,pdfpagetransition=Blinds,pdftoolbar=false,pdfmenubar=false]{hyperref}

\def\varep{\varepsilon}

\def\N{\mathbb{N}}

\def\R{\mathbb{R}}

\newcommand{\be}{\begin{equation}}
\newcommand{\ee}{\end{equation}}
\newcommand{\baa}{\begin{array}}
\newcommand{\eaa}{\end{array}}

\newtheorem{theorem}{Theorem}[section]
\newtheorem{lemma}[theorem]{Lemma}
\newtheorem{proposition}[theorem]{Proposition}

\allowdisplaybreaks	

\title{\bf{The influence of advection on the propagation phenomena of reaction-diffusion equations with KPP-bistable nonlinearity}}
\author{Xing Liang\thanks{School of Mathematical Sciences and Wu Wen-Tsun Key Laboratory of Mathematics, University of Science and Technology of China, Hefei, Anhui 230026,  China  (\texttt{xliang@ustc.edu.cn}). X. Liang is supported by the National Natural Science Foundation of China (12331006).}, \  Lei Zhang\thanks{School of Mathematics and Statistics, Shaanxi Normal University, Xi’an, Shaanxi 710119, China (\texttt{zhanglei890512@gmail.com}). L. Zhang is supported by the National Natural Science Foundation of China (12471168, 12171119) and the Fundamental Research Funds for the Central Universities (GK202304029, GK202306003, GK202402004). } \ and \
Mingmin Zhang\thanks{CNRS, UMR 5219, Institut de Math\'ematiques de Toulouse, Universit\'e Paul Sabatier,  Toulouse 31062, France  (\texttt{mingmin.zhang.math@gmail.com}). M. Zhang is supported by TIRIS  ``Junior Fellowship Program'' through the project ReaDi‐LS, and by the French Agence Nationale de la Recherche (ANR) through the project ReaCh (ANR-23-CE40-0023-01).} \\
}

\date{}
\begin{document}

\maketitle
\begin{abstract}   
	This paper is devoted to  propagation phenomena  for  reaction-diffusion-advection equations in  one-dimensional heterogeneous environments, where heterogeneity is reflected by the nonlinearity term - being KPP type on $(-\infty, -L]$ and being bistable type on $[L,+\infty)$ for some $L>0$.  A comprehensive analysis is presented on the influence of advection and heterogeneous reactions, based on various values of the advection rate $c$. Denote by $c_m$ the minimal wave speed of KPP equation and by $c_b$ the unique wave speed of bistable equation, respectively.  When $c>-c_m$, it is shown that propagation can always occur with leftward spreading speed $c_m+c$ and rightward spreading speed $\min\big(\max(c_b-c,0),c_m-c\big)$. Moreover, a logarithmic delay of the level sets in the left direction is discovered. When $c\le -c_m$, propagation phenomena are determined by the initial data and by the sign of $c_b$. In particular, when $c_b>0$, the leftward propagation speed is $c_b+c$ if the initial population is ``large enough''; whereas extinction occurs if the initial value is located in the bistable region and is ``relatively small''. In addition, the attractiveness of the bistable traveling wave is obtained when the leftward spreading speed is $c_b+c$ and/or when the rightward spreading speed is $c_b-c$.
	

	\vskip 2mm
	\noindent{\small{\it  
			Mathematics Subject Classification (2020)}:   35B40; 35K57; 35B35; 92D25.}
	\vskip 2mm
	\noindent{\small{\it Key words}:    Reaction-diffusion-advection equations; Propagation phenomena; Shifting environment; KPP-bistable.}

\end{abstract}

\tableofcontents

\section{Introduction and main results}
\hspace{1.em}
We address in this paper the  propagation phenomena for  reaction-diffusion-advection  equations in a heterogeneous framework. The heterogeneous character arises  in the equation through continuously varying (in space) reaction term between KPP (for Kolmogorov,
Petrovsky and Piskunov) type and bistable type.  The type of equations we consider here is:
\begin{equation}
	\label{1.1'}
	u_t-u_{xx}-cu_x=f(x,u)~~~t>0,~x\in\R,
\end{equation}
in which $u(t,x)$ denotes the species  density at time
$t$ and location $x$,  the constant $c\in\R$ is  advection rate, $f(x, u)$ describes population dynamics. From the biological point of view, advection can arise either from the behavior of individuals or from physical transport processes, such as winds, currents in rivers, for which the population may have a tendency to move along or against the gradient of $u$. Our paper is intended to understand the influence of the advection on propagation phenomena in the face of a heterogeneous environment of KPP-transition-bistable type.

Throughout this work, we assume that the function $f(y,s): \R\times \R_+\to\R$ is of class $C^2(\R\times\R_+)$, and $s\mapsto f(y,s)$ is locally Lipschitz continuous uniformly for $y\in\R$. Moreover, $f$ satisfies
\be\label{hypf}
		\left\{\baa{l}
		\begin{aligned} \displaystyle
	&\forall~ y\in\R,~~f(y,0)=f(y, 1)=0,~\partial_sf(y,1)<0 \hbox{ and }\ f(y, s)\le0\hbox{ for }s\ge 1,\\
	&\forall s\in\R_+, ~\exists~ L>0,~~f(y,s)=f_m(s)~\text{for}~y\in(-\infty,-L],~~f(y,s)=f_b(s)~\text{for}~y\in[L,+\infty),\\
	&\forall s\in\R_+, ~f(y,s)~\text{is decreasing in $y\in[-L,L]$}.
\end{aligned}
	\eaa\right.
\ee
Here, $f_m$ denotes the KPP nonlinearity:
\be\label{f_m}
f_m(0)=f_m(1)=0,~ 0<f_m(s)\le f'_m(0)s~\hbox{in }(0, 1),~f'_m(1)<0,~ f_m<0\hbox{ in }(1,+\infty),
\ee
while $f_b$ represents bistable reaction:
\begin{equation}
	\label{f_b}
	\begin{aligned}
			f_b(0)=f_b(\theta)=f_b(1)=0~\text{for some}~\theta\in(0,1),\vspace{3pt}~~~~~~~~~~~~~~~~\\
		f_b'(0)\!<\!0,~f_b'(\theta)\!>\!0,~f_b'(1)\!<\!0,~f_b\!>\!0~\text{in}~
		(\theta,1),~f_b\!<\!0~\text{in}~(0,\theta)\!\cup\!(1,+\infty).
	\end{aligned}
\end{equation}
We then refer to $(-\infty,-L]$ as the \textit{KPP region}, while  $[L,+\infty)$ would be regarded as the \textit{bistable region}. 

The propagation phenomena for  \eqref{1.1'} with nontrivial nonnegative continuous and compactly supported initial data shall be  investigated in the following sense: 
\begin{itemize}
	\item \textit{extinction}: $u(t,x)\to 0$ as $t\to+\infty$ uniformly in $x\in\mathbb{R}$;
	\item \textit{blocking $($say, in the right direction$)$}: $u(t,x)\to0$ as $x\to+\infty$ uniformly in $t\ge0$;
	\item \textit{propagation}: there exist $c_r$ and $c_l$ with $c_r+c_l>0$ such that  for $\varep>0$ small enough,
	$$\lim_{t \rightarrow +\infty} \inf_{-(c_l-\varep) t \le x \le  (c_r-\varep) t} u(t,x)>0,\qquad
	\lim_{t \rightarrow +\infty} \sup_{x \le -(c_l+\varep) t} u(t,x)=0=\lim_{t \rightarrow +\infty} \sup_{x \ge (c_r+\varep) t} u(t,x).$$ 
	Here $c_r$ and $c_l$ are called the rightward and leftward asymptotic spreading speeds, respectively.
\end{itemize}



Before stating our main result, let us review the fundamental results on the classical homogeneous reaction-diffusion equation 
\begin{equation}
	\label{1.2}
	u_t=u_{xx}+f(u),~~t>0,~x\in\mathbb{R},
\end{equation}
where $f$ is a $C^1(\R)$ function satisfying $f(0)=f(1)=0$. This equation has been extensively studied in the mathematical, physical and biological literature since the pioneering works of Fisher \cite{F1937} and Kolmogorov, Petrovskii and Piskunov \cite{KPP1937}.

With the KPP reaction $f=f_m$ in \eqref{1.2}, it is shown in \cite{KPP1937} that \eqref{1.2} admits traveling front solutions $u(t,x)=\varphi_\nu(x\cdot e-\nu t)$ where $e=\pm 1$ denotes the direction of propagation and $\nu$ is the wave speed, with $\varphi_\nu:\R\to(0,1)$ and $\varphi_\nu(-\infty)=1$, $\varphi_\nu(+\infty)=0$, if and only if $\nu\ge c_m:=2\sqrt{f'_m(0)}$. For each $\nu\ge c_m$, the wave profile $\varphi_\nu$ satisfies 
\be\label{TW}
\varphi_\nu''+\nu\varphi_\nu'+f(\varphi_\nu)=0~\text{in}~\mathbb{R},~~\varphi_\nu'<0~\text{in}~\mathbb{R},~~\varphi_\nu(-\infty)=1,~~\varphi_\nu(+\infty)=0,
\ee
and has translation invariance. Moreover, $\varphi_\nu$ has the following asymptotics:
\begin{align}\label{phiexp}
	\varphi_\nu(s)\mathop{\sim}_{s\to+\infty}\begin{cases}
		Ae^{-\lambda_\nu s} & \hbox{if }\nu>c_m,\\
		A^*s e^{-\lambda_\nu s} & \hbox{if }\nu=c_m,\\
	\end{cases}
\end{align} 
where $A,\,A^*$ are positive constants and the decay rate $\lambda_\nu>0$ is obtained from the linearized equation $u_t=u_{xx}+f'_m(0)u$ and is given by $\lambda_\nu=\big(\nu-\sqrt{\nu^2-4f_m'(0)}\big)/2$. It was proved in~\cite{HNRR2013,L1985,U1978} that the front with minimal speed $c_m$ attracts, in some sense, the solutions of the Cauchy problem~\eqref{1.2} associated with nonnegative bounded nontrivial compactly supported initial data $u_0$ in $\R$. Furthermore, Aronson and Weinberger~\cite{AW2} proved the  {\it spreading property}, stating that the solution $u$ to the Cauchy problem~\eqref{1.2} with a  nontrivial nonnegative compactly supported initial datum $u_0$ in $\R$ admits an asymptotic spreading speed $c_m=2\sqrt{f_m'(0)}$, in the following sense:
\begin{equation}
	\label{spreading properties}
	\begin{aligned}
	\begin{cases}
			\sup_{|x|\ge \omega t}u(t,x)\to 0,~~~&\text{if}~\omega>c_m\\
			\inf_{|x|\le  \omega t} u(t,x)\to 1~~~~&\text{if}~0\le \omega<c_m,
		\end{cases}
		~~~	\text{as}~~t\to+\infty.
	\end{aligned}
\end{equation}
The minimal traveling wave speed $c_m$ for \eqref{1.2} can therefore also be considered as the asymptotic spreading speed for the Cauchy problem associated with \eqref{1.2}.

In contrast, when the reaction is of bistable type, i.e. $f=f_b$,
equation~\eqref{1.2} has a unique (up to translation) traveling front solution $u(t,x)=\phi(x\cdot e-c_b t)$, where $\phi:\R\to(0,1)$ and satisfies~\eqref{TW}, $e=\pm 1$ is the direction of propagation, and $c_b\in\mathbb{R}$ is the wave speed  (has the sign of $\int_0^1 f_b(s) \mathrm{d}s$), depending only on $f_b$ \cite{AW2,FM1977}. It is known \cite{FM1977} that
\begin{equation}
	\label{2.6}
	\left\{\begin{aligned}
		a_0 e^{-\alpha s}\le \phi(s)\le a_1 e^{-\alpha s},~~s\ge 0,\\
		b_0 e^{\beta s}\le 1-\phi(s)\le b_1 e^{\beta s},~~s\le0,
	\end{aligned}\right.
\end{equation}
where $a_0$, $a_1$, $b_0$ and $b_1$ are some positive constants, $\alpha$ and $\beta$ are given by $\alpha=(c_b+\sqrt{c_b^2-4f_b'(0)})/2>0$ and $\beta=(-c_b+\sqrt{c_b^2-4f_b'(1)})/2>0$. 

\textit{Throughout this paper, we denote by $\varphi_\nu(x-\nu t)$  the KPP traveling front satisfying \eqref{TW} with $f=f_m$ and with wave speed $\nu\ge c_m$, whereas by $\phi(x-c_bt)$ the unique bistable traveling front satisfying \eqref{TW} with $f=f_b$, with wave speed $c_b$ and with the normalization $\phi(0)=\theta$. Let us stress that $c_m>c_b$, which is shown in Lemma \ref{Lemma_cm>cb}.} 

\vskip 0.2cm

\paragraph{Main results.}

Before stating our main result, let us provide some intuition regarding the propagation dynamics of \eqref{1.1'}. We begin with the simplest case where the advection speed $c$ is zero. In this case,  the sign of $c_b$, the speed of bistable traveling front, plays a decisive role in the global dynamics, where \cite{HLZ2} gives a strong clue about the propagation phenomena. However, the dynamics becomes unclear when the advection speed $c$ is non-zero. Especially, it is natural to ask about the role that $c$ plays. For instance, we wonder how $c_b$ vs. $c$ affect the dynamics, and what should be further considered to achieve a complete characterization of the front propagation.

Indeed, $c>0$ indicates that the species will be transported in the positive direction due to advection, while $c<0$ suggests that the species will be transported in the negative direction. Thus, it is not hard to imagine that the dynamics of \eqref{1.1'} can be governed by either the KPP region when $c \gg 1$ or bistable region when $c \ll -1$. However, the dynamics of the system will be characterized by the combined driving forces of the KPP region and the bistable region for $c$ not large enough in either direction. From the above preliminary analysis, it is obvious that  the advection speed $c$ does play a key role in determining the dynamical behavior of \eqref{1.1'}. Which factors will lead to the change in dynamics? Can we give a full picture of the propagation phenomena for \eqref{1.1'}? 

Our goal is therefore to answer these above questions and present  our main result regarding the propagation phenomena for problem \eqref{1.1'} associated with compactly supported initial data, where a particular attention will be devoted to the role played by advection.

\begin{theorem}
	\label{thm-main}
	The solution $u$ of~\eqref{1.1'} with a nonnegative continuous and compactly supported initial datum $u_0\not\equiv 0$ has the following propagation phenomena:
	\begin{itemize}
		\item
		
		When $c_b>0$, there holds
		\begin{table}[H]
			\centering
			\makebox[\textwidth][c]{
				\begin{tabular}{|c|c|c|c|c|}\hline
					Range of advection $c$& $(-\infty,-c_m]$& $(-c_m,c_b]$ &   $(c_b,c_m)$ & $[c_m,+\infty)$ \\\hline
					Leftward speed & $c_b+c$--$\bigstar$; Extinction--$\clubsuit$ & $c_m+c$&  $c_m+c$&$c_m+c$ \\\hline
					& Theorem \ref{thm_c le -cm,cb>0}; Theorem \ref{thm_extinction} & Theorem \ref{thm:-cm_cm:left}& Theorem \ref{thm:-cm_cm:left} &Theorem \ref{thm_c>cm} \\\hline
					Rightward speed & $c_b-c$--$\bigstar$; Extinction--$\clubsuit$ & $c_b-c$& 0 (Blocking) &$c_m-c$ \\\hline
					& Theorem \ref{thm_c le -cm,cb>0}; Theorem \ref{thm_extinction} & Theorem \ref{thm_propagation-2}& Theorem \ref{thm_blocking_c>c_b} &Theorem \ref{thm_c>cm} \\\hline
				\end{tabular}
			}
		\end{table}
		Here,  the notation $\bigstar$ represents that the localized initial condition is further assumed to be ``large enough'', while $\clubsuit$ denotes that \textnormal{spt}$(u_0)$ is included in the bistable region and $\Vert u_0\Vert_{L^\infty(\mathbb{R})} < \theta$.
		\item 
		\noindent
		When $-c_m<c_b<0$, there holds
		\begin{table}[H]
			\centering
			\makebox[\textwidth][c]{
				\begin{tabular}{|c|c|c|c|c|}\hline
					Range of  advection $c$&$(-\infty,-c_m]$& $(-c_m,c_b]$  & $(c_b,c_m)$ & $[c_m,+\infty)$ \\\hline 
					Leftward speed & Extinction & $c_m+c$&  $c_m+c$&$c_m+c$ \\\hline
					& Theorem \ref{thm_extinction}  & Theorem \ref{thm:-cm_cm:left}& Theorem \ref{thm:-cm_cm:left} &Theorem \ref{thm_c>cm} \\\hline
					Rightward speed & Extinction & $c_b-c$& 0 (Blocking) &$c_m-c$ \\\hline
					& Theorem \ref{thm_extinction}  & Theorem \ref{thm_propagation-2}& Theorem \ref{thm_blocking_c>c_b} &Theorem \ref{thm_c>cm} \\\hline
				\end{tabular}
			}
		\end{table}
		
		\item 
		\noindent
		When $c_b\le -c_m$, there holds
		\begin{table}[H]
			\centering
			\makebox[\textwidth][c]{
				\begin{tabular}{|c|c|c|c|c|}\hline
					Range of  advection $c$&$(-\infty,c_b]$ & $(c_b,-c_m]$ & $(-c_m,c_m)$ &  $[c_m,+\infty)$ \\\hline 
					Leftward speed & Extinction &  Extinction& $c_m+c$&$c_m+c$ \\\hline
					& Theorem \ref{thm_extinction}  & Theorem \ref{thm_extinction} & Theorem \ref{thm:-cm_cm:left} &Theorem \ref{thm_c>cm} \\\hline
					Rightward speed & Extinction & Extinction&  0 (Blocking) &$c_m-c$ \\\hline
					& Theorem \ref{thm_extinction}  & Theorem \ref{thm_extinction} & Theorem \ref{thm_blocking_c>c_b} &Theorem \ref{thm_c>cm} \\\hline
				\end{tabular}
			}
		\end{table}		
	\end{itemize}
 Moreover, when $u$ propagates to the left with speed $c_m+c$, the level sets always exhibit a logarithmic time delay, see Theorem \ref{thm_log delay}; when $u$ propagates to the left with speed $c_b+c$ or propagates to the right with speed $c_b-c$,  $u$ will eventually converge to the unique bistable traveling wave with a constant shift, see Theorems \ref{thm_c le -cm,cb>0} and  \ref{thm_propagation-2}. 
\end{theorem}

\noindent
\textbf{Remark}.
	 For the particular case $c_b=0$, we point out that when $c>-c_m$, the  propagation phenomena  coincide exactly with the classification for $c_b>0$, however the case of $c\le -c_m$ is rather delicate, and we leave it as an open question.
	
	\vspace{3mm}

A few comments on our main result are in order. First, it is straightforward to observe that there is greater possibility for species to survive and also propagate when $c \gg 1$. Since large advection will help the species to stay in the KPP region, which is the (more) favorable habitat, this (at least) allows species to propagate in the left direction. On the contrary, the bistable region will dominate the whole habitat when $c \ll -1$, which may lead to 
propagation or extinction, depending on how favorable the bistable region is and also on the initial population. The most complicated and interesting situation is when $c$ is not large enough in either direction, for which the long time dynamics of the solution in the right direction will rest on both the KPP and  bistable regions.  Among these intermediate cases,  two situations in our main result  may be somehow confusing at first glance, for which we give tentative and detailed explanations below:
\begin{itemize}
	\item In the case of $c_b>0$, when the advection rate $c$ changes from the interval $(-c_m,c_b]$ across $-c_m$ to $(-\infty,-c_m]$: we notice that there is a {\it jump} of leftward speed of propagation from $c_m+c$ to $c_b+c$ (for the latter, of course we consider large enough localized initial condition). For the  former case, the advection is not sufficiently negative so that the KPP region is not totally excluded from the habitat of the species. Consequently, this allows for the leftward propagation with speed $c_m+c>0$. However,  for the sufficiently negative advection  $c\in(-\infty,-c_m]$,  eventually it is only the bistable region that plays the role of habitat for the species. The species will propagate to the left with speed $c_b+c<0$ (which means the leftward front of the solution actually moves to the right with speed $-(c_b+c)>0$), when the initial population is ``large enough''; whereas the species will extinct provided that the initial population is located in the bistable region and is ``relatively small''.
	
	\item The rightward blocking phenomena for $c\in(-c_m,c_m)$ and $c>c_b$: the advection is not as strong as the KPP reaction drive, resulting in the KPP region persisting as part of the species' habitat. In view of $c>c_b$, the advection effect forces the species to  leave the bistable region and migrate into the KPP region, where they subsequently get resources, grow and invade in the right direction again. This eventually leads to a saturation state -- blocking in the right direction in the sense that the solution actually converges to the positive stationary solution $U$ such that $U(-\infty)=1$ and $U(+\infty)=0$. 
\end{itemize}

 Compared with existing literature for analyzing the effect of heterogeneous advection on nonlinear spreading and propagation phenomena for reaction-diffusion equations with a single dynamical mechanism (either KPP or bistable or combustion type) of the form
 \begin{equation}
 	\label{adv}
 	u_t-\nabla\cdot(A(x)\nabla u) +q(x)\cdot\nabla u=f(x,u),
 \end{equation}
 with diffusion matrix  $A(x)$ and drift $q(x)$ \cite{Freidlin1984,BLL1990,BLR1992,PX1991,R1992,MR1995,Hamel1997-1,Hamel1997,R1997,ABP2000,CKR2001,KR2001,B2002,BHN-1-2005}, here we consider the simple constant advection mode for the reaction-diffusion-advection equation \eqref{1.1'} in a heterogeneous environment of  KPP-transition-bistable type, and investigate all the possibilities of propagation phenomena by exhausting different values of the advection. To the best of our knowledge, our work is the first to address the effects of the advection on the propagation phenomena and the speed of propagation in a heterogeneous dynamical setting reflected by a KPP-transition-bistable reaction.  More generally, our problem can be generalized to a high dimensional situation:
 \begin{equation*}
 	u_t-\Delta u+ \alpha u_x=f(x,u),~~~~\text{in}~\Omega,
 \end{equation*}
where $\Omega\subset\R^\N$ is a straight cylinder and  $q(x)=\alpha e_1$ represents a shear flow, which can be studied in a similar way but one needs to pay attention to the effect of the geometry of the domains.

\paragraph{Discussion -- perspective of shifting environment.}

We close this section by introducing our main result from the shifting environment viewpoint and by making a comparison between our result and the existing ones.

By setting $u(t,x)=v(t,x+ct)$, equation \eqref{1.1'} is recast as:
\begin{equation}
	\label{1.1}
	v_t=v_{yy}+f(y-ct,v),~~~t>0,~y\in\R,
\end{equation}
in which the parameter $c$ is understood as the shifting speed, and the shifting growth $f(s,v)$ is assumed decreasing in $s\in[-L,L]$, being of KPP type on the left semi-infinite region, while being of bistable type on the right semi-infinite region.  With fixed diffusion, we aim to investigate the effect of shifting reaction. Of particular interest here is  to give a full classification of the propagation phenomena for such a model subject to KPP-transition-bistable nonlinearities with any value of shifting speeds.   The novelty is that here the strong Allee effect is taken into account  with the KPP reaction as the complement so that the whole habitat can be either favorable plus favorable, or favorable plus unfavorable.
In this sense, Theorem \ref{thm-main} can be transferred to the following result for \eqref{1.1}.
\begin{theorem}
	\label{thm-main'}
	The solution $v$ of~\eqref{1.1} with a nonnegative continuous and compactly supported initial datum $u_0\not\equiv 0$ has the following propagation phenomena:
	\begin{itemize}
		\item
		
		When $c_b>0$, there holds
		\begin{table}[H]
			\centering
			\makebox[\textwidth][c]{
				\begin{tabular}{|c|c|c|c|c|}\hline
					Range of $c$& $(-\infty,-c_m]$& $(-c_m,c_b]$ &   $(c_b,c_m)$ & $[c_m,+\infty)$ \\\hline
					Leftward speed & $c_b$--$\bigstar$; Extinction--$\clubsuit$ & $c_m$&  $c_m$&$c_m$ \\\hline
					& Theorem \ref{thm_c le -cm,cb>0}; Theorem \ref{thm_extinction} & Theorem \ref{thm:-cm_cm:left}& Theorem \ref{thm:-cm_cm:left} &Theorem \ref{thm_c>cm} \\\hline
					Rightward speed & $c_b$--$\bigstar$; Extinction--$\clubsuit$ & $c_b$& $c$ &$c_m$ \\\hline
					& Theorem \ref{thm_c le -cm,cb>0}; Theorem \ref{thm_extinction} & Theorem \ref{thm_propagation-2}& Theorem \ref{thm_blocking_c>c_b} &Theorem \ref{thm_c>cm} \\\hline
				\end{tabular}
			}
		\end{table}
		Here,  the notation $\bigstar$ represents that the localized initial condition is further assumed to be ``large enough'', while $\clubsuit$ denotes that \textnormal{spt}$(u_0)$ is included in the bistable region and $\Vert u_0\Vert_{L^\infty(\mathbb{R})} < \theta$.
		\item 
		\noindent
		When $-c_m<c_b<0$, there holds
		\begin{table}[H]
			\centering
			\makebox[\textwidth][c]{
				\begin{tabular}{|c|c|c|c|c|}\hline
					Range of $c$&$(-\infty,-c_m]$& $(-c_m,c_b]$  & $(c_b,c_m)$ & $[c_m,+\infty)$ \\\hline 
					Leftward speed & Extinction & $c_m$&  $c_m$&$c_m$ \\\hline
					& Theorem \ref{thm_extinction} & Theorem \ref{thm:-cm_cm:left}& Theorem \ref{thm:-cm_cm:left} &Theorem \ref{thm_c>cm} \\\hline
					Rightward speed & Extinction & $c_b$& $c$ &$c_m$ \\\hline
					& Theorem \ref{thm_extinction} & Theorem \ref{thm_propagation-2}& Theorem \ref{thm_blocking_c>c_b} &Theorem \ref{thm_c>cm} \\\hline
				\end{tabular}
			}
		\end{table}
		
		\item 
		\noindent
		When $c_b\le -c_m$, there holds
		\begin{table}[H]
			\centering
			\makebox[\textwidth][c]{
				\begin{tabular}{|c|c|c|c|c|}\hline
					Range of $c$&$(-\infty,c_b]$ & $(c_b,-c_m]$ & $(-c_m,c_m)$ &  $[c_m,+\infty)$ \\\hline 
					Leftward speed & Extinction &  Extinction& $c_m$&$c_m$ \\\hline
					& Theorem \ref{thm_extinction}  & Theorem \ref{thm_extinction} & Theorem \ref{thm:-cm_cm:left} &Theorem \ref{thm_c>cm} \\\hline
					Rightward speed & Extinction & Extinction&  $c$ &$c_m$ \\\hline
					& Theorem \ref{thm_extinction}  & Theorem \ref{thm_extinction} & Theorem \ref{thm_blocking_c>c_b} &Theorem \ref{thm_c>cm} \\\hline
				\end{tabular}
			}
		\end{table}
		
	\end{itemize}
 Moreover, when $u$ propagates to the left with speed $c_m$, the level sets always admit a logarithmic time delay, see Theorem \ref{thm_log delay}; when $u$ propagates to the left with speed $c_b$ or propagates to the right with speed $c_b$,  $u$ will eventually converge to the unique bistable traveling wave with a constant shift, see Theorems \ref{thm_c le -cm,cb>0} and  \ref{thm_propagation-2}. 
\end{theorem}

Given different values of the shifting speed $c$, the species might reside within the KPP and/or the bistable regions, resulting in significantly varying propagation dynamics after a long time.
Theorem \ref{thm-main'}
shows that, on the one hand,  very similar to the non-shifting case,  the population localization and total population
size can increase the possibility of persistence and  propagation under a moving climate, and, on the other hand, mobility can both reduce and enhance the ability of population to track climate change in order for persistence and even propagation.

In 2009, Berestycki et al. \cite{BDNZ2009} studied \eqref{1.1} by assuming that $f(y,v)$ is of KPP type for $0<y<L$, and $f(y,v)<0$ for $y<0$ and for $y> L$ which embodies the assumption that the population grows logistically in a favorable region of length $L$ but declines exponentially outside of that region. The authors \cite{BDNZ2009} proved that the global dynamics (extinction, persistence as well as convergence to the unique positive solution of $U''+cU'+f(x,U)=0$) is determined by the sign of the generalized principal eigenvalue of the linearized problem around zero state. Berestycki and Rossi  extended the large time dynamics result to high dimension \cite{BR2008}  and to infinite cylindrical type domains \cite{BR2009}, which was then investigated by Vo \cite{Vo2015} for more general unfavorable media at infinity.
In a particular situation $f(y-ct,v)=v(r(y-ct)-v)$ in \eqref{1.1} with $r:\R\to\R$ continuous, bounded and  nonincreasing\footnote{For convenience, here we  restate the conclusions of some references by assuming that $r$ is nonincreasing instead of nondecreasing.}, the KPP-decay case, i.e. $r(-\infty)>0>r(+\infty)$, was studied by Li et al. \cite{LBSF2014} and Fang et al. \cite{FLW2016},  showing that the species will extinct if $c \leq -c_-$ and persist if $c> -c_-$, and when propagation occurs, the leftward speed is $c_-$ and rightward speed is $\min(c,c_-)$. Fang et al. \cite{FPZ}  further considered the  spreading properties and forced waves in a time-periodic setting. Besides,  the KPP-KPP case, i.e. $r(-\infty)>r(+\infty)>0$, was investigated by Hu et al. \cite{HSL}, and by Lam and Yu \cite{LamYu2022} in which the spreading properties were studied, and moreover the frame of  reaction-diffusion equations and integro-differential equations with a distributed time-delay were addressed in \cite{LamYu2022}.

Yet, other types of growth functions (especially Allee effect) are also interesting  in the investigation of population dynamics in moving habitats, which is known for instance in \cite{Allee1938, RRBK2008}. 
Bouhous and Giletti \cite{BG2019} studied the propagation phenomena for  \eqref{1.1} in cylindrical type domains with general monostable reaction $f$ (including Allee effect). Bouhours and Nadin \cite{BN2015} considered \eqref{1.1} when the size of the favorable zone is bounded with general $f$ (including monostable and bistable cases) in the favorable zone, and proved  the
existence of two speeds $0<\underline c \le \overline c<+\infty$ such that the population persists for large enough initial data
when $0<c<\underline c$ and goes extinction when $c >\overline c$.
Recently, Li and Otto \cite{LO2023BMB} investigated forced waves for \eqref{1.1} where the favorable region, characterized by strong Allee effect, is a bounded interval surrounded by the unfavorable ones.
Beyond this, there have been extensive investigations on questions concerning the existence and further properties of forced waves for \eqref{1.1} of the form $v(t,y)=V(y-ct)$ with prescribed speed $c$, see e.g. \cite{Hamel1997-1,Hamel1997, BR2009, BF2018, FLW2016, HZ2017}. Spreading speeds in shifting environments have also been concerned, for instance, with nonlocal feature \cite{ZK2013,ABR2017,LWZ2018, C2021,LamYu2022}, with shifting diffusion \cite{FGH}, in free boundary problems \cite{LeiDu2017,DWZ2018,DHL2021},
in systems \cite{CTW2017, DSFL2021,HGW2023}, in abstract setting \cite{YZ2020}, etc. 
We refer the readers to \cite{Cosner2014,dcdsb} for  comprehensive description of recent development and open challenging questions on reaction-diffusion problems in shifting environments.

We conclude with a brief discussion.  The propagation dynamics resemble those in the KPP-decay case\cite{FLW2016,LBSF2014,FPZ} when $c_b \le -c_m$, where the bistable environment approximates the decay one. When $c_b \in (-c_m,c_m)$, the spreading properties are dominated by the KPP region for a positive and sufficiently large advection, i.e. $c \ge c_m$; the propagation dynamics are determined by the sign of $c_b$ and by the initial value for a negative and sufficiently large  advection, i.e. $c \le -c_m$; the spreading phenomena are characterized by the sign of $c_b-c$  when $c$ is not large enough in either direction, i.e. $c \in (-c_m,c_m)$. It is also noteworthy that our results align with those in \cite{HLZ2} in the specific case $c=0$, which indeed provided the foundational insight for our work.

\section{Propagation phenomena}\label{Sec 2}
In this section, we shall restate our main result in a separate way regarding different propagation phenomena, all of which eventually forms Theorem \ref{thm-main}.

\subsection{Complete propagation when $c\in [c_m,+\infty)$}
We begin by showing the complete propagation result when $c \ge c_m$ (independent of $c_b$). An intuitive explanation  is that  the KPP region  dominates the whole environment and plays a favorable role for the survival and spread of species, due to the large advection. Therefore, given any compactly supported initial population, there will be a hair-trigger effect for species. Moreover, the solution of \eqref{1.1'} will spread with leftward and rightward asymptotic spreading speeds $c_m+c$ and $c_m-c$ respectively.
\begin{theorem}[Complete propagation]
	\label{thm_c>cm}
	Assume that $c\ge c_m$. Let $u$ be the solution of~\eqref{1.1'} with a nonnegative continuous and compactly supported initial datum $u_0\not\equiv 0$. Then,  $u$ propagates to the left with speed $c_m+c$ and to the right with speed $c_m-c$, in the sense that 
	\begin{equation}
		\label{c ge c_m}
		\left\{\baa{l}
		\displaystyle\forall\,\varep>0,~~\lim_{t\to+\infty}\Big(\sup_{x\in(-\infty,-(c_m+c+\varep)t]\cup[(c_m-c+\varep)t,+\infty)}u(t,x)\Big)=0,\vspace{3pt}\\
		\displaystyle\forall\,\varep\in(0,c_m),~~\lim_{t\to+\infty}\Big(\sup_{-(c_m+c-\varep)t\le x\le (c_m-c-\varep)t}|u(t,x)-1|\Big)=0.\eaa\right.
	\end{equation}
\end{theorem}

\subsection{Complete propagation vs. rightward blocking when $c\in(-c_m,c_m)$}

When $c\in(-c_m,c_m)$, it reflects that the KPP region necessarily plays the role of habitat for species, but it is not sure for the bistable region. We shall prove that the species can always spread to the left, whereas  it may be blocked or propagate to the right in the bistable region depending on the sign of $c_b-c$. At this stage, it is worth to notice that the initial condition $u_0$ does not play any role in the propagation phenomena when $c\in(-c_m,c_m)$.

\subsubsection*{Leftward propagation when $c\in(-c_m,c_m)$}

We first show that the species is able to  propagate  to the left for $c\in(-c_m,c_m)$. This is because the KPP region contributes to the survival and leftward spread of the species. Here is our statement.
\begin{theorem}[Leftward propagation]
	\label{thm:-cm_cm:left} 	
	Assume that $c\in(-c_m,c_m)$. Let $u$ be the solution of~\eqref{1.1'} with a nonnegative continuous and compactly supported initial datum $u_0\not\equiv 0$. Then,  $u$ propagates to the left with speed $c_m+c>0$ in the sense that 
	\begin{equation}\label{equ:-cm_cm}
		\left\{\baa{l}
		\displaystyle\forall\,\varep>0,~~\lim_{t\to+\infty}\Big(\sup_{x\le -(c_m+c+\varep)t}u(t,x)\Big)=0,\vspace{3pt}\\
		\displaystyle\forall\,\varep\in(0,c_m+c),\ \forall\,\delta>0,\ \exists\,x_1< -L,~~\limsup_{t\to+\infty}\Big(\sup_{-(c_m+c-\varep)t\le x\le x_1}|u(t,x)-1|\Big)<\delta.\eaa\right.
	\end{equation}
	In particular, $\sup_{-c_2t\le x\le -c_1t}|u(t,x)-1|\to0$ as $t\to+\infty$ for every $0<c_1\le c_2<c_m+c$.
\end{theorem}


\subsubsection*{Rightward propagation vs. blocking when $c\in(-c_m,c_m)$}
In contrast to leftward propagation Theorem \ref{thm:-cm_cm:left}, rightward propagation phenomena in the case of $c\in(-c_m,c_m)$ is uncertain. We will distinguish our analysis for $c\in(-c_m,c_m)$ into two cases: either $c>c_b$ or $c\le c_b$, where the former means $\max(c_b,-c_m)<c<c_m$, and
 for the latter it suffices to consider the situation when $-c_m<c_b$ and $c\in(-c_m,c_b]$. The case where $c \leq c_b$ and $c_b \leq -c_m$ is outside the range $(-c_m, c_m)$, and it will be discussed later.

\begin{theorem}[Blocking]
	\label{thm_blocking_c>c_b}
	Assume that $c\in(-c_m,c_m)$ and $c>c_b$. Then the solution $u$ of the Cauchy problem~\eqref{1.1'} with a nonnegative continuous and compactly supported initial datum $u_0\not\equiv 0$ is blocked in the right direction, that is, 	\begin{equation}\label{blocking}
		u(t,x)\to 0~~\text{as}~x\to+\infty, ~\text{uniformly in}~t\ge 0.
	\end{equation}
	Furthermore, $u$ satisfies
	\begin{equation}
		\label{blocking-convergence}
		u(t,\cdot)- U\to 0 ~~\text{as}~t\to+\infty, ~\text{locally uniformly in}~x\in\R,
	\end{equation} 
	where $U$ is the unique positive bounded stationary solution of \eqref{1.1'} such that $U(-\infty)=1$ and $U(+\infty)=0$, given in Proposition \ref{prop_U connecting 1 and 0}.
\end{theorem}


\begin{theorem}[Rightward propagation]
	\label{thm_propagation-2}
	Assume that $-c_m<c_b$ and $c\in(-c_m,c_b]$. Then the solution $u$ of~\eqref{1.1'} with a nonnegative continuous and compactly supported initial datum $u_0\not\equiv 0$ propagates completely, namely,
	\begin{equation}
		\label{2.7}
		u(t,x)\to 1~~\text{as}~t\to+\infty,~\text{locally uniformly in}~x\in\mathbb{R}.
	\end{equation}
	Furthermore, 
	\begin{enumerate}[(i)]
		\item if $c\in(-c_m,c_b)$, then $u$ propagates to the right with speed $c_b-c>0$. Moreover, there exist $X>L$ and $\xi\in\R$ such that
		\begin{equation*}
\lim_{t\to+\infty}\Big(\sup_{x\ge X}|u(t,x)-\phi(x-(c_b-c)t+\xi)\Big)=0,
		\end{equation*}
		where $\phi$ is the unique bistable traveling wave profile  solving \eqref{TW} with $f=f_b$ propagating with speed $\nu=c_b$ and with normalization $\phi(0)=\theta$;
		\item if $c=c_b$, then $u$ propagates to the right with speed zero, in the sense that~\eqref{2.7} holds and $\sup_{x\ge \nu t}u(t,x)\to0$ as $t\to+\infty$ for every $\nu >0$.
	\end{enumerate}	
\end{theorem}

\noindent
{\bf Remark}. Theorems \ref{thm_blocking_c>c_b}--\ref{thm_propagation-2}  demonstrate that the rightward spreading property is totally determined by the sign of $c_b -c$, provided that $c \in (-c_m,c_m)$.
\begin{itemize}
	\item  Theorem \ref{thm_propagation-2}: In the case where $-c_m<c_b$ and $c \in (-c_m,c_b]$, since the bistable speed is no less than the advection, that is, $c_b-c\ge 0$, the population will take both the KPP and bistable regions as habitats, thus will lead to a rightward propagation with speed $c_b-c\ge 0$. In particular, when $c=c_b$, we also observe the ``virtual blocking'' phenomenon as in \cite{HLZ2} (among other things, it corresponds to the situation of $c=0$ and $c_b=0$),  that is, the level sets do expand to the right, but with speed 0.

	\item Theorem \ref{thm_blocking_c>c_b}: In the case where $c \in (c_b,c_m)\cap (-c_m,c_m)$, since the advection is greater than the bistable speed, that is, $c_b-c < 0$, the population ``intuitively'' may lead to rightward propagation with speed $c_b-c<0$. However, this is not the case, and our tentative explanation is the following:  On the one hand, since $c_b-c<0$, it pushes the species to leave the bistable region and migrate into the KPP region. On the other hand, since the advection $c \in (-c_m,c_m)$ is not too large, the species will get the source to grow and expand, in particular the species can invade in the right direction, which eventually leads to a saturation state -- blocking in the right direction -- in the sense that the solution converges to the positive steady state $U$ asymptotically. 
\end{itemize}


\subsection{Conditional complete propagation vs. extinction when $c\in(-\infty,-c_m]$}

Eventually, let us deal with  the case that $c\le -c_m$. The sufficient negative advection will force the bistable region to play a central role for the survival of the species, which will result in diverse propagation dynamics of the species depending on the sign of $c_b$ and/or on the size and also the location of the initial condition.

\subsubsection*{Conditional complete propagation}
In the case where  $c\le-c_m$ and $c_b>0$,  the bistable region could be either favorable or unfavorable depending on the initial data. For the survival of
the species, the size and the position of the initial population have to be taken into account. Our following result says that  for initially ``large enough'' localized population set in bistable region, the species can persist and spread, which is indeed in the same spirit of Fife and McLeod \cite{FM1977}. 
\begin{theorem}[Conditional complete propagation]
	\label{thm_c le -cm,cb>0}
	Assume that $c\le -c_m$ and $c_b>0$. Then for any $\eta>0$, there is $L^*>0$ such
	that the following holds: for any nonnegative continuous and compactly supported initial datum\footnote{Here we do not restrict the location of the initial datum $u_0$. In fact, this theorem is still true, when such $u_0$ is fully set in the bistable region.}
	satisfying $u_0\ge \theta+\eta$ on an interval of size $L^*$, the solution $u$ of \eqref{1.1'} with
	initial datum $u_0$ propagates to the right with speed $c_b-c$ and  to the left with speed $c_b+c$.  Moreover, there are $X>L$ and $z_i\in\R$ $($$i=1,2$$)$ such that
	\begin{equation}\label{equ:c_le_cm:r}
	\lim_{t\to+\infty}\Big(	\sup_{x\ge X-ct} |u(t,x)-\phi(x-(c_b-c)t+z_1)|\Big)=0.
	\end{equation}
	and 
 	\begin{equation}\label{equ:c_le_cm:l}
		\lim_{t\to+\infty}\Big(\sup_{ x\le X-ct} |u(t,x)-\phi(-x-(c_b+c)t+z_2)|\Big)=0,
	\end{equation}
	where $\phi$ is the bistable traveling wave profile of \eqref{1.2} with $f=f_b$ propagating with speed $c_b$ and with normalization $\phi(0)=\theta$.
\end{theorem}

\subsubsection*{Extinction}
In the end, we consider extinction phenomenon under the assumption that $c\le -c_m$. On the one hand, if $c_b<0$, the bistable region has no possibility to be favorable for species to persist; on the other hand, if we assume further  that the localized initial condition is set on the bistable region and is ``relatively small'', the solution will go to zero, whatever the sign of $c_b$ is. Therefore, we have the following extinction result.
\begin{theorem}[Extinction]
	\label{thm_extinction}
	Assume that $c\le -c_m$. Let  $u$ be the solution of \eqref{1.1'} with a nonnegative continuous and compactly supported initial datum $u_0\not\equiv 0$. Then, $u$ will extinct, namely, $u(t,x)\to 0$ as $t\to+\infty$ uniformly in $x\in\R$, provided one of the following conditions is satisfied:
	\begin{itemize}
		\item[(i)] $c_b<0$;
		\item[(ii)]   \textnormal{spt}$(u_0)$ is included in the bistable region and $\Vert u_0\Vert_{L^\infty(\mathbb{R})} < \theta$.
	\end{itemize}
\end{theorem}
 Let us comment on Theorem \ref{thm_extinction} (ii). This sufficient condition for extinction indeed relates to not only the size but also the location of the initial condition. This condition, we believe, can be relaxed but somehow it is not easy to find an optimal one. A heuristic explanation is the following: suppose that \textnormal{spt}$(u_0)$ is large enough and fully located in the KPP region, i.e., far to the left, and $\Vert u_0\Vert_{L^\infty(\mathbb{R})} < \theta$, the solution may have a chance to persist with $u(T,\cdot)\ge \theta+\eta$ on an interval of size $L^*$ for some large $T>0$, which will lead to propagation provided that $c_b>0$, due to Theorem \ref{thm_c le -cm,cb>0}.

\subsection{Sharp estimate of the level sets in the left direction when $c\in(-c_m,+\infty)$}

Given $\varrho\in(0,1)$, define the level set $E_\varrho^-(t)=\inf\{x\in\R|u(t,x)=\varrho\}$ for any $t>0$.

\begin{theorem}\label{thm_log delay}
	Assume that $c>-c_m$.	Let  $u$ be the solution of \eqref{1.1'} with a nonnegative continuous and compactly supported initial datum $u_0\not\equiv 0$.
	Then $u$ propagates to the left with speed $c_m+c$, thanks to Theorems \ref{thm_c>cm}--\ref{thm:-cm_cm:left}, and for every $\varrho\in(0,1)$,
	\begin{equation*}
		E_\varrho^-(t)=-(c_m+c)t+\frac{3}{2\lambda^*}\ln t+O_{t\to+\infty}(1)
	\end{equation*}
	with $\lambda^*=c_m/2$.
\end{theorem}

This theorem demonstrates that whenever the KPP region plays a role as (part of) the habitat, the leftward propagation will be very similar to the homogeneous KPP case in the sense that the same logarithmic correction appears regarding the asymptotic position of the level sets in the left direction, which means that the effect of the bistable region is rather weak. Then, one can ask whether such logarithmic delay of the level sets remains true, when $u$ propagates to the right with speed $c_m-c$? As a matter of fact,  given any $\varrho\in(0,1)$, by defining $E_\varrho^+(t)=\sup\{x\in\R|u(t,x)=\varrho\}$ for any $t>0$, one readily observes that  $E_\varrho^+(t)\le (c_m-c)t-3/(2\lambda^*)\ln t+C$ for some $C\in\R$, since equation \eqref{1.1} with $f_m(v)$ instead of $f(y-ct,v)$ produces such an upper bound. This means that there exists {\it at least} a logarithmic delay for the level sets in the rightward propagation, which answers part of the question.   However, how to catch a more precise behavior of $E_\varrho^+(t)$ is far from clear, which also strongly rests on the effect of the bistable region. This is beyond the scope of this work and is left as an open question.

\paragraph{Organization of the paper.} The main body of the paper is devoted to the proofs of theorems in this section.









	\section{Preliminary results}
	
	\begin{lemma}\label{Lemma_cm>cb}
	There holds $c_m>c_b$.
	\end{lemma}
\begin{proof}
	 Let $w$ and $z$ be respectively the solution to the Cauchy problem \eqref{1.2} in $\R_+\times\R$ with $f$ replaced by $f_m$ and $f_b$, associated with the same initial data. First of all, one readily infers from the comparison principle that  $w(t,x)>z(t,x)$ for $t>0$ and $x\in\R$, which implies $c_m\ge c_b$. 
	 
	 Assume towards contradiction that $c_m=c_b$. Let us focus on the region $x>0$. It follows from  \cite[Theorem 3.2]{FM1977},  there exists $\xi\in\R$ such that 
	$ \sup_{x>0}\big|z(t,x)-\phi(x-c_b t+\xi)\big|\to 0$ as $t\to+\infty$.
	 This implies that the level sets of $z$ is asymptotically $c_bt+O_{t\to+\infty}(1)$. On the other hand, it is known from \cite{HNRR2013} that the level sets of $w$ behave asymptotically as $c_mt -3/(2\lambda^*)\ln t+O_{t\to+\infty}(1)$, which move apprently  slower than the front propagation of $z$. This is a contradiction. We then conclude that $c_m>c_b$.
\end{proof}

	\begin{lemma}\label{lem:u:basic_propert}
		Let $u$ be the solution to~\eqref{1.1'} with a nonnegative continuous and compactly supported initial datum $u_0\not\equiv 0$. Then
			\begin{equation*}
				\sup_{x \in \mathbb{R}}u(t,x)\le 1~~~~~~\text{as}~~t\to+\infty.
			\end{equation*}
			Namely,		for any $\varep>0$, there exists $T>0$ such that $\sup_{x \in \mathbb{R}}u(t,x)\le 1+\varep$ for all $t\ge T$.
	\end{lemma}
	\noindent
		\textbf{Remark}. Based upon Lemma \ref{lem:u:basic_propert}, we observe that the solution of \eqref{1.1'} will eventually be bounded from above by 1 for large times albeit with a small nuance, no matter how large the $L^\infty$ norm of $u_0$ is. Therefore, we assume with no loss of generality that the initial condition $u_0$ is nontrivial such that, for any $\varep>0$, there holds $0\le u_0(x)< 1+\varep$ for $x\in\R$  throughout this paper.
	\begin{proof}
%
		Let $\xi$ be the solution to the ODE $\xi'(t)=f_m(\xi(t))$ for~$t\ge 0$ with initial condition $\xi(0)=\max\big( 1,\Vert u_0\Vert_{L^\infty(\mathbb{R})})$, it follows that $\xi(t)\searrow 1$ as $t\to+\infty$. By the comparison principle, one has~$0<u(t,x)\le \xi(t)$ for all $(t,x)\in(0,+\infty)\times \mathbb{R}$, which implies that $ u(t, x)\le 1$ as $t\to+\infty$ uniformly for $x\in\R$. This completes the proof.
	\end{proof}

The following lemma provides the Gaussian upper bounds for the solution $u$ to \eqref{1.1'}.

\begin{lemma}
	\label{lemma1.3}
	Let $L_1>0$, $L_2>0$, and let $u$ be the solution to the Cauchy problem~$	u_t=u_{xx}+cu_x+f(u)$, with a nonnegative continuous and compactly supported initial datum $u_0$ satisfying {\rm{spt}}$(u_0)\subset[-L_1,L_2]$,   where the $C^1(\R_+)$ function $f$ is assumed to  satisfy $f(s)\le Ks$ for all $s\ge 0$ with some constant $K>0$. Then there holds, for all $t>0$,
	$$u(t,x)\le M e^{K t}  e^{-\frac{(x+ct+L_1)^2}{4t}}\hbox{ for all }x\le-ct-L_1,\ \hbox{ and }\ u(t,x)\le M e^{K t} e^{-\frac{(x+ct-L_2)^2}{4t}}\hbox{ for all }x\ge L_2-ct,$$
	with $M:=\max(1,\|u_0\|_{L^\infty(\R)})$.
\end{lemma}
\begin{proof}
	Set $z=x+ct$ and $v(t,z)=u(t,z-ct)$, then $v(t,z)$ satisfies 
	\begin{equation*}
		v_t=v_{zz}+f(v)~~~t>0,~z\in\R,
	\end{equation*}
	with initial condition $u_0$ satisfying spt$(u_0)\subset[-L_1,L_2]$. By the comparison principle, for any $z \in \mathbb{R}$ and $t>0$, 
	$$
	v(t,z)\leq 
	\frac{e^{Kt}}{\sqrt{4\pi t}} \int_{\mathbb{R}} e^{- \frac{(z-y)^2}{4t}} u_0(y) d y=\frac{e^{Kt}}{\sqrt{4\pi t}} \int_{-L_1}^{L_2} e^{- \frac{(z-y)^2}{4t}} u_0(y) d y.
	$$
	This  gives that for  $t>0$,
	\begin{equation*}
		v(t,z)\le M e^{K t}  e^{-\frac{(z+L_1)^2}{4t}}\hbox{ for all }z\le-L_1,\ \hbox{ and }\ v(t,z)\le M e^{K t} e^{-\frac{(z-L_2)^2}{4t}}\hbox{ for all }z\ge L_2,
	\end{equation*}
	Turning back to the function $u$, the conclusion immediately follows.
\end{proof}

	\begin{lemma}\label{lem:Semi-persistence}
		Assume that $c\in(-c_m,c_m)$. Let $u$ be the solution of~\eqref{1.1'} with a nonnegative continuous and compactly supported initial datum $u_0\not\equiv 0$. Then $u(t,x)$ is semi-persistent, that is, for every $\overline x\in\mathbb{R}$,
		$$\inf_{x\le\bar x}\Big(\liminf_{t\to+\infty} u(t,x)\Big)>0.$$
		Moreover, \eqref{1.1'} admits positive stationary solutions, and any  positive stationary solution $p$ of \eqref{1.1'}  satisfies $p(-\infty)=1$.
	\end{lemma}
	\begin{proof}
		For any $0<\varep<\frac{1}{4}(c_m^2-c^2)$, one can choose $R>0$ large enough  such that
		\begin{equation}
			\label{4.1}
			\frac{\pi}{2R}<\sqrt{f'_m(0)-\varep-\frac{c^2}{4}}.
		\end{equation}
		Then, define $\Psi:\R\to\R$ as
		\begin{equation}
			\label{4.2}
			\begin{aligned}
				\Psi(x)=\begin{cases}
					\ \displaystyle e^{-\frac{c}{2}x}\cos\Big(\frac{\pi}{2R}x\Big)&\text{in} ~[-R,R],\cr
					\ 0&\text{elsewhere}.
				\end{cases}
			\end{aligned}
		\end{equation}
		Then there exists $\eta_0>0$ such that $-(\eta\Psi)''-c(\eta\Psi)'\le f_m(\eta\Psi)$ in $(-R,R)$ for all $0<\eta\le \eta_0$. Fix now any~$x_0\le -L-R$ and pick $\eta\in(0,\eta_0]$ such that $\eta\Psi(\cdot-x_0)<u(1,\cdot)$ in $\mathbb{R}$. 
		
		Let $z$ be the solution to~\eqref{1.1'} with initial datum $z(0,\cdot)=\eta\Psi(\cdot-x_0)$ in $\mathbb{R}$. The strong maximum principle applied in $(0,+\infty)\times[-L-2R,-L]$ yields that $z(t,x)>z(0,x)=\eta\Psi(x-x_0)$ for $t>0$ and $x\in[-L-2R,-L]$. Therefore, $z(t+h,\cdot)>z(t,\cdot)$ in $\R$ for every $h>0$ and $t\ge 0$. Namely, $z$ is increasing with respect to~$t$.  Moreover, the comparison principle implies that~$0<z(t,x)< u(t+1,x)$ for all $t>0$ and $x\in\mathbb{R}$. It then follows from parabolic estimates that $z(t,\cdot)$ converges as $t\to+\infty$, locally uniformly in $\mathbb{R}$,  to a  positive bounded stationary solution $p$  of~\eqref{1.1'}. Clearly, $p(x) \geq z(t,x)>0$ for $t>0$ and $x\in \mathbb{R}$. Furthermore, together with Lemma \ref{lem:u:basic_propert}, it follows that
		\begin{equation}\label{4.3}
			0<p(x)\le \liminf_{t\to+\infty} u(t,x)\le \limsup_{t\to+\infty} u(t,x)\le 1,~~\text{locally uniformly for}~x\in\mathbb{R}.
		\end{equation} 
		Since $p$ is continuous and positive in $\mathbb{R}$, one gets from~\eqref{4.3} that, for any given~$\overline x>x_0$,
		\begin{equation}
			\label{equ:inf:x0_xbar}
			\liminf_{t\to+\infty}u(t,x)\ge \min_{[x_0,\overline x]}p>0~~\text{for all}~ x\in[x_0, \overline x].
		\end{equation}
		In view of \eqref{equ:inf:x0_xbar}, to prove the semi-persistence result,
		it suffices to show 
		\begin{equation}\label{equ:inf:x0}
			\inf_{x \leq x_0}\liminf_{t\to+\infty}u(t,x)>0.
		\end{equation}
		
		Notice also that $p(x)>  z(0,x)$ for $x\in\mathbb{R}$. By continuity of $p$ and $z(0,\cdot)$, there is $\hat\kappa>1$ such that $p>\eta\Psi(\cdot-\kappa x_0)$ in $[\kappa x_0-R,\kappa x_0+R]$ for all $\kappa\in[1,\hat \kappa]$. Define
		$$\kappa^*:=\sup\big\{\kappa\ge 1: p>\eta\Psi(\cdot-\widetilde\kappa x_0)~\text{in}~[\widetilde\kappa x_0-R,\widetilde\kappa x_0+R]~\text{for all}~\widetilde\kappa\in [1,\kappa]\big\}.$$
		It follows that $\kappa^*\ge\hat\kappa>1$. We only need to prove $\kappa^*=+\infty$. Assuming by contradiction that~$\kappa^*<+\infty$, we see from the  definition of $\kappa^*$ that $p\ge \eta\Psi(\cdot-\kappa^*x_0)$ in $[\kappa^*x_0-R,\kappa^*x_0+R]$ and there is $x^*\in[\kappa^*x_0-R,\kappa^*x_0+R]$ such that $p(x^*)= \eta\Psi(x^*-\kappa^*x_0)$. Since $p>0$ in $\mathbb{R}$ and $\Psi(\cdot-\kappa^*x_0)=0$ at $\kappa^*x_0\pm R$, one has $x^*\in(\kappa^*x_0-R,\kappa^*x_0+R)$. Then the strong elliptic maximum principle implies that $p\equiv \eta\Psi(\cdot-\kappa^*x_0)$ in $(\kappa^*x_0-R,\kappa^*x_0+R)$ and then in $[\kappa^*x_0-R,\kappa^*x_0+R]$ by continuity, which is impossible. Thus, $\kappa^*=+\infty$ and $p>\eta\Psi(\cdot-\kappa x_0)$ in $[\kappa x_0-R,\kappa x_0+R]$ for all $\kappa\ge 1$. This implies particularly that $p(x)>\eta\Psi(0)=\eta$ for  $x\le x_0$. It follows from \eqref{4.3} that
		$
		\liminf_{t\to+\infty}u(t,x)>\eta~~\text{for all}~x\le x_0
		$, and hence, \eqref{equ:inf:x0} holds. We then obtain the semi-persistence result, as well as the existence of positive stationary solutions to \eqref{1.1'}.
		
		Assume that $p$ is a positive stationary solution  of \eqref{1.1'} in $\R$, let us show that  $p(-\infty)=1$. Fix $x_1<-L-R$. Since the function $\Psi$ given in \eqref{4.2} is compactly supported in $\R$, we can  choose $\eta\in(0,\eta_0]$ such that $\eta\Psi(\cdot-x_1)<p$ in $\mathbb{R}$. By repeating the arguments to obtain the semi-persistence result, we have $\eta\Psi(\cdot-\tilde{x})<p$ in $\mathbb{R}$ for any $\tilde{x}<-L-R$, which further implies that $\liminf_{x\rightarrow -\infty} p(x)>0$.  We now claim that
		\begin{equation*}
			\liminf_{x\rightarrow -\infty} p(x)>0~~~\Longrightarrow~~~p(-\infty)=1.
		\end{equation*}
		Indeed, let us consider an arbitrary sequence $(x_n)_{n\in\mathbb{N}}$ in $\mathbb{R}$ diverging to  $-\infty$ as $n\to+\infty$ and define $p_n:=p(\cdot+x_n)$ in $\R$ for each $n\in\mathbb{N}$. By standard elliptic estimates, the sequence $(p_n)_{n\in\mathbb{N}}$ converges as $n\to+\infty$, up to extraction of some subsequence, in $C^2_{loc}(\mathbb{R})$ to a bounded function $p_\infty$ which solves $p_\infty''+cp_\infty'+f_m(p_\infty)=0$ in $\mathbb{R}$. In view of $\liminf_{x\rightarrow -\infty} p(x)>0$, we have $\inf_{\R}p_\infty>0$. Notice that $f_m>0$ in $(0,1)$ and $f_m<0$ in $(1,+\infty)$, it then follows that the equilibrium $1$ attracts all solutions of $\xi'(t)=f_m(\xi(t))$ with any positive initial data $\xi(0)$. Therefore, a comparison argument implies ~$p_\infty\equiv 1$ in $\R$. That is,~$p_n\to 1$ as~$n\to+\infty$ in $C^2_{loc}(\R)$. Since the  sequence~$(x_n)_{n\in\mathbb{N}}$ was arbitrarily chosen, we conclude that  $p(x)\to 1$ and $p'(x)\to0$ as $x\to-\infty$. Our claim is achieved. This completes the proof of Lemma \ref{lem:Semi-persistence}.
	\end{proof}

\begin{lemma}
	\label{lemma_c<-c_m}
	Assume that $c\le -c_m$. Let $u$ be the solution of~\eqref{1.1'} with a nonnegative continuous and compactly supported initial datum $u_0\not\equiv 0$. Then, for any $x_0\in\R$,
	\begin{equation}
		\label{6.1-2}
		\lim_{t\to+\infty} u(t,x)=0,~~~~\text{uniformly in}~x\le -(c_m+c)t+x_0.
	\end{equation}
	In particular, $\lim_{t\to+\infty} u(t,x)=0$ uniformly for $x\le x_0$.
\end{lemma}
\begin{proof}
	Let $w$ be the solution of $w_t=w_{xx}+cw_x+f'_m(0)w$ for $t>0$ and $x\in\R$, associated with initial condition $w(0,\cdot)=u_0$ in $\R$. Then,
	\begin{align*}
		w(t,x)&=\frac{e^{f_m'(0)t}}{\sqrt{4\pi t}}\int_\R e^{-\frac{[x+(c_m+c)t-c_mt-y]^2}{4t}}u_0(y)dy\\
		&=\frac{e^{\frac{c_m^2t}{4}}}{\sqrt{4\pi t}}\int_\R e^{-\frac{[x+(c_m+c)t-y]^2-2[x+(c_m+c)t-y](c_mt)+(c_m t)^2}{4t}}u_0(y)dy\\
		&\leq\frac{1}{\sqrt{4\pi t}}\int_\R e^{\frac{2[x+(c_m+c)t-y](c_mt)}{4t}}u_0(y)dy\\
		&=\frac{1}{\sqrt{4\pi t}} e^{\frac{[x+(c_m+c)t]c_m}{2}}\int_\R e^{\frac{-c_my}{2}}u_0(y)dy
	\end{align*} 
	from which we deduce that $\sup_{x\le -(c_m+c)t+x_0} w(t,x)\to 0$ as $t\to+\infty$, for any $x_0\in\R$. This, together with the comparison principle, implies  \eqref{6.1-2}. 
\end{proof}

\section[Proofs of Theorems \ref{thm_c>cm}--\ref{thm:-cm_cm:left}]{Propagation in the KPP region when $c\ge c_m$ and when $c\in(-c_m,c_m)$: Proofs of Theorems \ref{thm_c>cm}--\ref{thm:-cm_cm:left}}
\label{Sec 3}
	
		We now give the proofs of Theorem \ref{thm_c>cm} -- describing the complete propagation when $c\ge c_m$ and of Theorem \ref{thm:-cm_cm:left} -- concerning the spreading property in the left direction when $c\in(-c_m,c_m)$.
	
	\begin{proof}[Proofs of Theorems \ref{thm_c>cm}--\ref{thm:-cm_cm:left}]
		Assume that $c> -c_m$.  The conclusion can be easily reached by comparison arguments.

		Let $g(x,s): \R\times \R_+\mapsto\R$ be a $C^1$ function. Moreover, we also require that $g$ satisfies
		$g(x,0)=g(x, 1)=0$ for each $x\in\R$, and $g(x,s)\le f(x,s)$ for $(x,s)\in\R\times\R_+$,
		$g(x,s)=f_m(s)~\text{for}~x\in(-\infty,-L],~~g(x,s)=-rs~\text{for}~y\in[L,+\infty)$, 
		$g(x,s)$ is nonincreasing in $x\in\R$ for each $s\in\R_+$. Let $z$ be the solution to problem $z_t=z_{xx}+cz_x+g(x,z)$ starting with initial value $z_0=u_0$ in $\R$. Notice that $z$ behaves as a lower barrier and it is known from, for instance, Theorem 1.3 (ii) and (iii) of \cite{FPZ} that $z$ satisfies \eqref{c ge c_m}--\eqref{equ:-cm_cm} when $c \geq c_m$ and when $c \in (-c_m,c_m)$, respectively.  
		
		On the other hand, define by $w$ the solution of $w_t=w_{xx}+cw_x+f_m(w)$ in $\R_+\times\R$ with initial datum $w_0=u_0$ in $\R$, it follows from \cite{AW1} that $w$ satisfies property \eqref{c ge c_m}. 
		By the comparison priniciple,  we have
		$z(t,x)\leq u(t,x) \leq w(t,x)$ for $(t,x)\in\R_+\times \R$. As a consequence, taking into account the properties satisfied by $z$ and $w$, the conclusion in Theorems \ref{thm_c>cm}--\ref{thm:-cm_cm:left}  immediately follows.
	\end{proof}

\section[Proof of Theorem \ref{thm_blocking_c>c_b}]{Blocking in bistable region when  $c\in(-c_m,c_m)$ and $c>c_b$: Proof of Theorem \ref{thm_blocking_c>c_b}}

 As preliminaries, we first show the existence and uniqueness of positive bounded stationary solutions  for the above two cases respectively, which will play crucial roles in the course of our investigation. 
\begin{proposition}
	\label{prop_U connecting 1 and 0}
	Assume  that $c\in(-c_m,c_m)$ and $c>c_b$. Then  problem~\eqref{1.1'} admits a unique positive bounded stationary solution $U$ such that $U(-\infty)=1$ and $U(+\infty)=0$. Such $U$ satisfies $0<U< 1$ and $U'<0$ in $\R$.
\end{proposition}

Based on Proposition \ref{prop_U connecting 1 and 0}, we point out that if $u_0$ satisfies~$0\le u_0\le U$ in $\mathbb{R}$, then the comparison principle immediately implies that the associated solution $u$ of the Cauchy problem~\eqref{1.1'} will satisfy $0\le u(t,x)\le U(x)$ for all $(t,x)\in\R_+\times\R$, hence it is blocked in the bistable region. 

\begin{proof} We divide our proof into four steps.
	
	\noindent
	\textit{Step 1. Existence.}	For $\varep\in(0,(1-\theta)/2)$, let $f_{b,\varep}$ be a $C^1(\mathbb{R})$ function such that 
	\begin{equation}
		\label{above f_b}
		\begin{aligned}
			\left\{\baa{l}
			f_{b,\varep}(0)=f_{b,\varep}(\theta)=f_{b,\varep}(1+\varep)=0,~~f_{b,\varep}'(0)<0,~~f_{b,\varep}'(1+\varep)<0,\vspace{3pt}\\
			f_{b,\varep}>0~\text{in}~(-\infty,0)\cup(\theta,1+\varep),~~f_{b,\varep}<0~\text{in}~(0,\theta)\cup(1+\varep,+\infty).\eaa\right.
		\end{aligned}
	\end{equation}
	Moreover, we assume that $f_{b,\varep}\ge f_b$ in $\R$, $f_{b,\varep}=f_b$ in $(0,1-\varep)$, and $f_{b,\varep}$ is decreasing in $[1-\varep,1+\varep]$. For each $\varep\in(0,(1-\theta)/2)$, let  $\overline\phi_\varep$ be the unique traveling front profile of $u_t=  u_{xx}+f_{b,\varep}(u)$ such that 
	\begin{equation}
		\label{TW-above f_b}
		\overline\phi_\varep''+c_{b,\varep}\overline\phi_\varep'+f_{b,\varep}(\overline\phi_\varep)=0\hbox{ in }\R,~~\overline\phi_\varep'<0~\text{in}~\mathbb{R},~~\overline\phi_\varep(L)=1,~~\overline\phi_\varep(-\infty)=1+\varep,~~\overline\phi_\varep(+\infty)=0,
	\end{equation}
	with speed $c_{b,\varep}$. It is standard to see that $\overline\phi_\varep\to\phi$ in $C^2_{loc}(\R)$ and $c_{b,\varep}\to c_b$ as $\varep\to0$. Since $c_b<c$, we can fix $\varep\in(0,(1-\theta)/2)$ small enough such that $c_{b,\varep}$ is close enough to $c_b$ and $c_{b,\varep}<c$.
	
	Let $\overline{u}$ be the solution of \eqref{1.1'} in $\R_+\times\R$ with initial condition
	\begin{align}
		\label{prop2.3-super u}
		\overline{u}_0(x)=\min\big(\overline{\phi}_\varep(x),1 \big)~~~~\text{for}~~x\in\R.
	\end{align}
	We observe that $\overline u_0\equiv1$ in $(-\infty,L]$ and $\overline u_0=\overline\phi_\varep$ in $(L,+\infty)$. As long as $\overline u_0=\overline \phi_\varep<1$, since $\overline{\phi}_\varep''+c\overline\phi_\varep'+f_b(\overline{\phi}_\varep)= \overline{\phi}_\varep''+c_{b,\varep}\overline\phi_\varep'+f_{b,\varep}(\overline{\phi}_\varep)+(c-c_{b,\varep})\overline\phi_\varep'+f_b(\overline{\phi}_\varep)-f_{b,\varep}(\overline{\phi}_\varep)<0$ in $\R$, it follows that $\overline{u}_0$ is a stationary supersolution of~\eqref{1.1'} in $\R$.

	Let $\underline u$  be the solution of~\eqref{1.1'} with initial condition $\underline u_0=\eta \Psi(\cdot-x_0)$ given in \eqref{4.2} for $x_0\le-L -R$ with $R>0$ given in~\eqref{4.1}. We can choose $\eta>0$ sufficiently small such that $\underline u_0<\overline u_0$ in $\mathbb{R}$.	The comparison principle implies  that 
	\begin{equation}
		\label{comparison-u}
		0<\underline u(t,x)<\overline u(t,x)< \overline u_0(x)~~~\text{for all}~~t>0~\text{and}~x\in\R.
	\end{equation} 
	Moreover, $\underline u$~is increasing with respect to $t$ and $\overline u$ is decreasing with respect to $t$ in $\R_+\times\R$. 
	From standard parabolic estimates, $\underline u(t,\cdot)$ and $\overline u(t,\cdot)$ converge as $t\to+\infty$, locally uniformly in $\mathbb{R}$, to classical positive stationary solutions $p$ and $q$  of~\eqref{1.1'}, respectively. The comparison principle implies that
	\begin{equation}
		\label{prop2.3-1}
		0<\underline u(t,x)<p(x)\le  q(x)<\overline u_0,~~~\text{for}~t>0,~x\in\R.
	\end{equation} 
Moreover, since $\max(-c_m,c_b)<c<c_m$, we infer from Lemma \ref{lem:Semi-persistence} that $	p(-\infty)=q(-\infty)= 1.$ On the other hand, since
$\lim_{x\to+\infty}\overline u_0(x)=0$, we deduce from  \eqref{prop2.3-1} that $p(+\infty)=q(+\infty)=0$.
	Therefore, we have shown the existence of positive bounded stationary solution $U$ of \eqref{1.1'} such that $U(-\infty)=1$ and $U(+\infty)=0$. 
	
	\noindent
	\textit{Step 2. $0<U< 1$ in $\R$.}
	Let now $U$ be a positive stationary solution of \eqref{1.1} such that $U(-\infty)=1$ and $U(+\infty)=0$.  Indeed, \eqref{prop2.3-1} implies that $U>0$ in $\R$. It suffices to show that $U<1$ in $\R$. To do so, we first suppose that  there is $\hat x\in\R$ such that $U(\hat x)=\max_\R U>1$. Hence, $U'(\hat x)=0$ and $U''(\hat x)\le 0$. We then observe that $0=U''(\hat x)+cU'(\hat x)+f(\hat x,U(\hat x))<0$, which  is impossible. Therefore, $U\le 1$ in $\R$. If there is $\bar x\in\R$ such that $U(\bar x)=1$, then it is necessarily a local maximum of $U$, whence  $U'(\bar x)=0$. The Cauchy-Lipschitz theorem implies that $U\equiv 1$ in $[\bar x, +\infty)$. This contradicts $U(+\infty)=0$. Consequently, $0<U<1$ in $\R$.	
	
	\noindent
	\textit{Step 3. Uniqueness.}
	Suppose that $U$ and $W$ are two distinct positive bounded stationary solutions  of \eqref{1.1'} connecting 1 and 0. We notice that $U$ and $W$ have the same decay rate at $\pm\infty$, that is,
	\begin{equation}\label{decay-UW}
		1-U(x)=O(e^{\eta x})=1-W(x) ~\text{as}~x\to-\infty,~~~U(x)=O(e^{-\zeta x})=W(x)~~\text{as}~x\to+\infty
	\end{equation}
	with $\eta=(-c+\sqrt{c^2-4f_m'(1)})/2>0$ and $\zeta=(c+\sqrt{c^2-4f_b'(0)})/2>0$. Therefore, one can choose $\kappa_0>0$ sufficiently large such that $W>U(\cdot+\kappa_0)$ in $\R$, thanks to \eqref{decay-UW}. Define 
	\begin{equation*}
		\kappa^*=\inf\{\kappa\in\R: W>U(\cdot+\kappa)~\text{in}~\R\}.
	\end{equation*}
	We observe that $-\infty<\kappa^*\le \kappa_0$. Set $w:=W-U(\cdot+\kappa^*)$. By continuity, it follows that the function $w$ is nonnegative in $\R$ and vanishes at some point $x_0\in\R$.  Moreover, $w$ satisfies
	\begin{equation*}
		-w''(x)-cw'(x)=f(x,W(x))-f(x+\kappa^*,U(x+\kappa^*))\ge f(x,W(x))-f(x,U(x+\kappa^*)) =\gamma(x)w(x)~~~x\in\R,
	\end{equation*} 
	for some bounded function $\gamma$ in $\R$.
	Since $w(x_0)=0$, the strong maximum principle implies that $w\equiv0$ in $\R$. 
	That is, $W(x)=U(x+\kappa^*)$ for $x\in\R$. Suppose that $\kappa^*>0$, then we derive that both $U$ and $U(\cdot+\kappa^*)$ are stationary solutions of \eqref{1.1'}, which is impossible since 
	\begin{equation*}
		0\!=\!U''(x+\kappa^*)+cU'(x+\kappa^*)+f(x, U(x+\kappa^*))\gneqq U''(x+\kappa^*)+c U'(x+\kappa^*)+f(x+\kappa^*, U(x+\kappa^*))\!=\!0~~\text{for}~x\in\R,
	\end{equation*}
contradiction, thanks to the hypothesis \eqref{hypf} on $f$. Similarly, we also exclude the case that $\kappa^*<0$. Therefore, $\kappa^*=0$, which simply means that $W\equiv U$ in $\R$.	Consequently, such stationary solution of \eqref{1.1'} connecting 1 and 0 is unique.

		\noindent
	\textit{Step 4. $U'<0$ in $\R$.} Assume by contradiction that the unique positive stationary solution of \eqref{1.1} such that $U(-\infty)=1$ and $U(+\infty)=0$ is not strictly decreasing in $\R$, there then exists $a>0$ such that $U\ge U(\cdot+a)$ in $\R$ and $U(x^*)=U(x^*+a)$ for some $x^*\in\R$. In particular, the functions $U$ and $U(\cdot+a)$ are strictly separated and ordered for all $|x|$ large enough. Define $z:=U-U(\cdot+a)$ in $\R$, it follows that $z\ge\neq 0$ in $\R$ and $z(x^*)=0$. Moreover, $z$ satisfies
	\begin{equation*}
		-z''(x)-cz'(x)=f(x,U(x))-f(x+a,U(x+a))\ge f(x,U(x))-f(x,U(x+a)) =\mu(x)z(x),~~~x\in\R,
	\end{equation*} 
    	for some bounded function $\mu$ in $\R$.
    Since $z(x^*)=0$, the strong maximum principle implies that $z\equiv0$ in $\R$. With the same reasoning as in Step 3, we arrive at a  contradiction. Consequently, $U$ is strictly decreasing in $\R$.
	The proof of Proposition \ref{prop_U connecting 1 and 0} is complete. 
\end{proof}

We now  carry out the proof of Theorem \ref{thm_blocking_c>c_b}. The proof, among other things, is based  on a comparison with   a traveling front with  zero speed as a barrier.

\begin{proof}[Proof of Theorem \ref{thm_blocking_c>c_b}]
	 The strategy of the proof consists in constructing a supersolution which blocks the solution $u(t,x)$ for all times $t\ge0$ as $x\to +\infty$. 
	
	\noindent
	\textit{Step 1. Blocking.} Set $M=\max(\Vert u_0\Vert_{L^\infty(\mathbb{R})},1)+1$.  For any $\varep\in(0,(1-\theta)/2)$,  let $f_{b,\varep}\in C^1(\R)$ be defined as in the beginning of the proof of Proposition \ref{prop_U connecting 1 and 0}, with $M$ this time instead of $1+\varep$. There is then a decreasing front profile $\overline\phi_\varep$ solving \eqref{TW-above f_b}. It is seen that $\overline \phi_\varep\to\phi$ in $C^2_{loc}(\R)$ and $c_{b,\varep}\to c_b$ as $\varep\to 0$. We then fix $\varep\in(0,(1-\theta)/2)$ small enough such that  $c_{b,\varep}$ is close enough to $c_b$ and $c_{b,\varep}<c$. One can choose $A>0$ sufficiently large such that  $\max(u_0(x),1)\le \overline \phi_\varep(-A)$ for all $x\le L$ and $u_0(x)\le \overline\phi_\varep(x-A-L)$ for all $x\ge L$. 
	
	Define $\overline u$ by\footnote{We remark here that the construction of $\overline u$ relies only on the assumption that $c>c_b$.}
	\begin{equation}
		\label{super_thm2.5}
		\overline{u}(x)=\left\{\baa{ll} \overline{\phi}_\varep(-A) & \hbox{if }x<L,\vspace{3pt}\\ 
		\overline{\phi}_\varep(x-A-L) & \hbox{if }x\ge L.\eaa\right.
	\end{equation} 
	For $x\le L$, due to $\overline u(x)=\overline\phi_\varep(-A)\ge 1$, one has that $f(x,\overline\phi_\varep(-A))\le 0$. Moreover, since $\overline u(x)=\overline\phi_\varep(x-A-L)$ for $x\ge L$, it is observed that
	\begin{align}
		\label{stationary supersol_bistable}
		\overline u_t-\overline u_{xx}-c\overline u_x-f(x,\overline u)=-\overline\phi_\varep''-c\overline\phi_\varep'-f_b(\overline\phi_\varep)
		\ge -\overline\phi_\varep''-c_{b,\varep}\overline\phi_\varep'-\overline f_{b,\varep}(\overline\phi_\varep)=0,~~~x\ge L,
	\end{align}
	by noticing that $\overline \phi_\varep'<0$, $c>c_{b,\varep}$ and $f_b\le f_{b,\varep}$ in $\R$.
	Therefore, $\overline u$ is a stationary supersolution of~\eqref{1.1'} in $\R$, with $u_0\le\overline{u}$ in $\R$.  The comparison principle implies that $u(t,x)\le \overline u(x)$ for $t\ge 0$ and $x\in\R$. Consequently, $u$ will be blocked. That is,
	\begin{equation*}
		u(t,x)\to 0~~\text{as}~x\to+\infty, ~\text{uniformly in}~t\ge 0.
	\end{equation*}

	\noindent
	\textit{Step 2. Convergence to the stationary solution $U$.}
	We denote by $u_1$ the solution to the Cauchy problem \eqref{1.1'} with initial datum $u_1(0,\cdot)=\overline u$ in $\R$ with $\overline u$ given in \eqref{super_thm2.5} such that $u_1(0,\cdot)>u_0$ in $\R$, and by $u_2$ the solution to the Cauchy problem~\eqref{1.1'}  with initial condition $u_2(0,\cdot)=\eta\Psi(\cdot-x_0)< u(1,\cdot)$ in $\mathbb{R}$ for $\eta>0$ small enough and for some $x_0\le-L -R$, where $R>0$ and $\Psi$ are given as in~\eqref{4.1}--\eqref{4.2}. 
	The comparison principle implies  that 
	\begin{equation}
		0<u_2(t,x)<u(t+1,x)<u_1(t+1,x)< u_1(0,x)=\overline u(x)~~~\text{for all}~~t>0~\text{and}~x\in\R.
	\end{equation}  
	Moreover, $u_2$~is increasing with respect to $t$ and $u_1$ is decreasing with respect to $t$ in $(0,+\infty)\times\R$. 
	From standard parabolic estimates, $u_2(t,\cdot)$ and $u_1(t,\cdot)$ converge as $t\to+\infty$, locally uniformly in $\mathbb{R}$, to classical stationary solutions $p$ and $q$  of~\eqref{1.1'}, respectively. Moreover, there holds
	\begin{equation}
		\label{comparison-u-upper lower}
		0< p(x)\le\liminf_{t\to+\infty}u(t,x)\le\limsup_{t\to+\infty} u(t,x) \le  q(x)\le  \overline u(x),
	\end{equation} 
	locally uniformly for $x\in\mathbb{R}$.  
	
	In view of  $\max(-c_m,c_b)<c<c_m$, by Lemma \ref{lem:Semi-persistence}, we  eventually arrive at $p(-\infty)=q(-\infty)=1$. On the other hand, since $\underline u$ is nonnegative and $\overline\phi_\varep(+\infty)=0$, one also derives that $p(+\infty)=q(+\infty)=0.$
	Furthermore, it follows from Proposition \ref{prop_U connecting 1 and 0}  that $p=q=U$ in $\R$.
	Therefore, the local uniform convergence of $u(t,\cdot)$ to $U$ as $t\to+\infty$ follows immediately from \eqref{comparison-u-upper lower}.
	The proof of Theorem~\ref{thm_blocking_c>c_b} is therefore complete.
\end{proof}

\section[Proof of Theorem \ref{thm_propagation-2}]{Propagation in bistable region when $-c_m<c_b$  and $c\in(-c_m,c_b]$: Proof of Theorem \ref{thm_propagation-2}} 
\label{Sec4_propagtion in bistable region}

In this section, we study the rightward spreading speed and the attractiveness of the bistable  traveling wave when $-c_m<c_b$, motivated by \cite{FM1977,HLZ2}.

\begin{proposition}
	\label{prop_V=1 when c le c_b}
	Assume that $-c_m<c_b$  and  $c\in(-c_m,c_b]$, then $V\equiv 1$ is the unique positive stationary solution of \eqref{1.1'}.
\end{proposition}

\begin{proof}
	Suppose that $V$ is a positive bounded stationary solution of \eqref{1.1'}. Since $c_b<c_m$, we have $c\in(-c_m,c_m)$. It then follows from Lemma \ref{lem:Semi-persistence} that $V(-\infty)=1$.  We prove now that $V\le 1$ in $\R$. In fact, let $\xi(t)$ be the solution of the ODE $\xi'(t)=f_m(\xi(t))$ for $t>0$ associated with $\xi(0)=\Vert V\Vert_{L^\infty(\mathbb{R})}+1$. We notice that $\xi(t)\searrow 1$ as $t\to+\infty$. By comparison principle, there holds $V(x)<\xi(t)$ for $t\ge 0$ and $x\in\R$, whence $V(x)\le 1$ for $x\in\R$.  To complete the proof, it remains to prove that $V\ge 1$ in $\R$.  In the sequel, we deal with the case that $c<c_b$ and $c=c_b$, respectively. 
	
	
	\noindent	
	\textbf{Step 1.} \textit{The case that $c<c_b$.} 
	To show that $V\ge 1$ in $\R$, we build a subsolution as a lower barrier which moves to the right with speed $\delta>0$. To be specific, for $\varep\in(0,\min(\theta, 1-\theta)/2)$, let $ f_{b,\varep}$ be a $C^1(\mathbb{R})$ function such that 
	\begin{equation*}
		\label{below f_b}
		\begin{aligned}
			\left\{\baa{l}
			 f_{b,\varep}(-\varep)= f_{b,\varep}(\theta)= f_{b,\varep}(1-\varep)=0,~~ f_{b,\varep}'(-\varep)<0,~~ f_{b,\varep}'(1-\varep)<0,\vspace{3pt}\\
			 f_{b,\varep}=f_b~\text{in}~(\varep,1-2\varep),~~ f_{b,\varep}>0~\text{in}~(-\infty,-\varep)\cup(\theta,1-\varep),~~ f_{b,\varep}<0~\text{in}~(-\varep,\theta)\cup(1-\varep,+\infty).\eaa\right.
		\end{aligned}
	\end{equation*}
	We further assume that $ f_{b,\varep}\le f_b$ in $\R$ and that $ f_{b,\varep}$ is decreasing in $[-\varep,\varep]$ and in $[1-2\varep,1]$. For each $\varep\in(0,\min(\theta, 1-\theta)/2)$, let $\widetilde\phi_\varep$ be the unique traveling front profile of $u_t=  u_{xx}+ f_{b,\varep}(u)$ such that 
	\begin{equation}
		\label{TW-below f_b}
		\widetilde\phi_\varep''+ c_{b,\varep}\widetilde\phi_\varep'+ f_{b,\varep}(\widetilde\phi_\varep)=0\hbox{ in }\R,~~\widetilde\phi_\varep'<0~\text{in}~\mathbb{R},~~\widetilde\phi_\varep(0)=\theta,~~\widetilde\phi_\varep(-\infty)=1-\varep,~~\widetilde\phi_\varep(+\infty)=-\varep,
	\end{equation}
	with speed $ c_{b,\varep}$. It is standard to see that $\widetilde\phi_\varep\to\phi$ in $C^2_{loc}(\R)$ and $ c_{b,\varep}\to c_b$ as $\varep\to0$. We then fix $\varep\in(0,\min(\theta, 1-\theta)/2)$ small enough such that $ c_{b,\varep}>c$.

	Since $V(-\infty)=1$ and $V>0$ in $\R$, one can choose $x_0\in\R$ such that $\widetilde\phi_\varep(\cdot-x_0)<V$ in $\R$. 
	Fix now $0<\delta< c_{b,\varep}-c$, and define 
	\begin{align}
		\label{sub u}
		\underline{u}(t,x)=\max\big(\widetilde\phi_\varep(x-\delta t-x_0), 0\big),~~~~t\ge 0,~x\in\R.
	\end{align}
	Let us now check that $\underline u$ is a subsolution to \eqref{1.1'} in $\R_+\times\R$. To do so, it suffices to check the situation of $\underline u(t,x)=\widetilde\phi_\varep(x-\delta t-x_0)>0$, where we have 
	\begin{align*}
		\underline u_t-\underline u_{xx}-c\underline u_x-f(x,\underline u)
		\le -\widetilde\phi_\varep''-(c+\delta)\widetilde\phi_\varep'-f_b(\widetilde\phi_\varep)< -\widetilde\phi_\varep''- c_{b,\varep}\widetilde\phi_\varep'- f_{b,\varep}(\widetilde\phi_\varep)=0,
	\end{align*}
	due to the fact that $f(x,s)$ is nonincreasing with respect to $x$ for each $s> 0$, and that $c+\delta< c_{b,\varep}$ and $ f_{b,\varep}\le f_b$. Therefore, $\underline u$ is a subsolution to \eqref{1.1'} for $t\ge 0$ and $x\in\R$. By comparison principle, it follows that
	\begin{equation}
		\label{lemma2.3 comparison}
		\underline u(t,x)< V(x),~~~t\ge 0,~x\in\R.
	\end{equation}
	However, we observe that such a subsolution $\underline u$ moves to the right with speed $\delta>0$, which implies that $\sup_{x< \delta t}\underline u\to 1-\varep$ as $t\to+\infty$. Since $\varep>0$ is arbitrarily chosen and small enough, together with \eqref{lemma2.3 comparison}, we exclude the possibility that $V<1$ somewhere in $\R$. Thus, $V\ge 1$ in $\R$.  Consequently, we have achieved $V\equiv 1$ in the case of $c\in(-c_m,c_b)$.

	\noindent
	\textbf{Step 2.} \textit{The case that $c=c_b$}.  Assume by contradiction that $V(x_0)<1$ for some point $x_0\in\R$. Let $\phi$ be the bistable traveling wave profile such that
	\begin{equation*}
		\phi''+c_b\phi'+f_b(\phi)=0~\text{in}~\R,~\phi(0)=\theta,~\phi(-\infty)=1,~\phi(+\infty)=0.
	\end{equation*}
	By observing that $\phi(x)$ satisfies
	\begin{equation*}
		-\phi''-c_b\phi'-f(x,\phi)\le -	\phi''- c_b\phi'-f_b(\phi)=0~~\text{in}~\R,
	\end{equation*}
	it follows that $\phi$ is a stationary subsolution of \eqref{1.1'} in $\R$. Moreover, we observe that 
	$V(x)\ge O(e^{-\alpha x})=\phi(x)$ as $x\to+\infty$, with $\alpha$ given in \eqref{2.6}, whereas as $x\to-\infty$ one has $1-\phi(x)=O(e^{\beta x})$ and $1-V(x)=O(e^{\gamma x})$ with $\beta$ given in \eqref{2.6} and $\gamma=(-c_b+\sqrt{c_b^2-4f_m'(1)})/2>\beta$, due to  $f_m'(1)\le f_b'(1)<0$. It follows that $1-V(x)$ decays faster than $1-\phi(x)$ as $x\to-\infty$.
	There then exists $\eta_0\in\R$ such that $V>\phi(\cdot+\eta_0)$ in $\R$. Now define 
	\begin{equation*}
		\eta^*=\inf\{\eta\in\R: V>\phi(\cdot+\eta_0)~\text{in}~\R\}. 
	\end{equation*}
	One has that $-\infty<\eta^*\le \eta_0$ by noticing that $V(x_0)<\phi(-\infty)=1$ in $\R$, as assumed. By continuity, the function $w:=V-\phi(\cdot+\eta^*)$ is nonnegative, nontrivial and vanishes somewhere. Namely, there is $x_0\in\R$ such that $w(x_0)=0$ and $w$ satisfies
	\begin{align*}
		-w''(x)-c_bw'(x)=f(x,V)-f_b(\phi(x+\eta^*))
		\ge f_b(V)-f_b(\phi(x+\eta^*))=r(x)w(x),~~~x\in\R,
	\end{align*}
	for some bounded function $r$ in $\R$.
	The strong maximum principle implies that $w\equiv0$ in $\R$, i.e. $V\equiv\phi(\cdot+\eta^*)$ in $\R$. This is impossible, contradicting the asymptotic behaviors of $V$ and $\phi$ as $x\to-\infty$. We have proven that $V\ge 1$ in $\R$, whence $V\equiv 1$ also holds true in the case of $-c_m<c=c_b$. The proof of Proposition~\ref{prop_V=1 when c le c_b} is complete.
\end{proof}


\begin{proof}[Proof of Theorem~$\ref{thm_propagation-2}$]
	Proposition \ref{prop_V=1 when c le c_b} gives that   $V\equiv 1$ is the unique positive bounded stationary solution of \eqref{1.1'}. The comparison principle then implies that $0<u(t,x)\le M:=\max(1,\Vert u_0 \Vert_{L^\infty(\mathbb{R})})$ for all $t>0$ and $x\in\mathbb{R}$.
	
	Let $z$  be the solution to the Cauchy problem~\eqref{1.1'} with initial condition $z(0,\cdot)=\eta\Psi(\cdot-x_0)<u(1,\cdot)$ in $\mathbb{R}$ for $\eta>0$ small enough and for any arbitrary $x_0\le-L -R$, where $R>0$ and $\Psi$ are given as in~\eqref{4.1}--\eqref{4.2}, thanks to the condition that $c\in(-c_m,c_b]\subset(-c_m,c_m)$; while let $w$~denote the solution to~\eqref{1.1'} with initial condition $w(0,\cdot)=M$ in $\mathbb{R}$. The comparison principle implies that $0<z(t,x)<u(t+1,x)<w(t+1,x)\le M$ for all $t>0$ and $x\in\mathbb{R}$. Moreover, $z$~is increasing with respect to $t$ and $w$ is nonincreasing with respect to $t$ in $\R_+\times\R$. From parabolic estimates, $z(t,\cdot)$ and $w(t,\cdot)$ converge as $t\to+\infty$, locally uniformly in $\mathbb{R}$, to positive classical stationary solutions $p$ and $q$  of~\eqref{1.1'}, respectively. Moreover, there holds
	\begin{equation}
		\label{4.30}
		0<p\le \liminf_{t\to+\infty} u(t,\cdot)\le \limsup_{t\to+\infty} u(t,\cdot)\le q\le M,
	\end{equation} 
	locally uniformly in $\mathbb{R}$. Based on  Proposition \ref{prop_V=1 when c le c_b}, it is immediate to see that $p\equiv 1=q$ in $\R$. Consequently, it follows that $u(t,x)\to 1$ as $t\to+\infty$, locally uniformly for $x\in\R$. This proves \eqref{2.7}.

	Next, we shall prove Theorem~\ref{thm_propagation-2} (i), which  relies on three preliminary lemmas.
	

	\begin{lemma}
	\label{lemma 3.1}
	  Assume that $c<c_b$. Let $u$ be the solution of~\eqref{1.1'} with a nonnegative continuous and compactly supported initial datum $u_0\not\equiv 0$. The following holds true for some $\mu>0$ and $\delta>0$:
	 \begin{itemize}
	 	\item[(i)]  there exist $X_1>L$, $T_1>0$, $z_1\in\mathbb{R}$ such that
	 		\begin{equation}
	 		\label{4.57}
	 		u(t,x)\le\phi(x-(c_b-c)(t-T_1)+z_1)+\delta e^{-\delta(t-T_1)}+\delta e^{-\mu(x-X_1)}~~\text{for all}~ t\ge T_1~\text{and} ~x\ge X_1,
	 	\end{equation}
	 	
	 	\item[(ii)] assume further that $-c_m<c_b$ and $c\in(-c_m,c_b)$, there exist  $X_2>L$,  $T_2>0$,  $z_2\in\mathbb{R}$ such that
	 		\begin{equation}
	 		\label{4.58}
	 		u(t,x)\ge\phi(x-(c_b-c)(t-T_2)+z_2)-\delta e^{-\delta(t-T_2)}-\delta e^{-\mu(x-X_2)}~~\text{for all}~ t\ge T_2 ~\text{and} ~x\ge X_2,
	 	\end{equation}
	 \end{itemize}
	where $\phi$ is the unique bistable traveling front profile solving~~\eqref{TW}  with $f_b$ instead of $f$, with speed $\nu=c_b$ and with $\phi(0)=\theta$.
\end{lemma}

	\begin{proof}
	  We first introduce some parameters. Choose $\mu>0$ such that
		\begin{equation}
			\label{4.59}
			0<\mu<\min\left(\frac{\max_{[-2\delta,1+2\delta]}|f_b'|}{c_m},\frac{1}{2}\Big(c+\sqrt{c^2+2\min\big(|f_b'(0)|, |f_b'(1)|\big)}\Big) \right).
		\end{equation}
		Then we take $\delta>0$ such that 
		\begin{equation}	
			\label{4.60}\left\{\baa{l}
			\displaystyle 0<\delta< \min\Big(\mu (c_b-c), \frac{1}{2},\frac{|f_b'(0)|}{2}, \frac{|f_b'(1)|}{2}\Big),\vspace{3pt}\\
			\displaystyle f_b'\le \frac{f_b'(0)}{2}~\text{in}~[-2\delta,3\delta],~~f_b'\le \frac{f_b'(1)}{2}~\text{in}~[1-3\delta,1+2\delta].\eaa\right.
		\end{equation}
		Let $C>0$ be such that
		\begin{equation}
			\label{phiC}
			\phi\ge 1-\delta/2~\text{in}~(-\infty,-C]\ \hbox{ and }\ \phi\le \delta~\text{in}~[C,+\infty).
		\end{equation}
		Since $\phi'$ is negative and continuous in $\mathbb{R}$, there is $\kappa>0$ such that 
		\begin{equation}
			\label{4.61}
			\phi'\le-\kappa<0~\text{in}~[-C, C].
		\end{equation}
		Finally, pick $\omega>0$ so large that
		\begin{equation}
			\label{4.62}
			\kappa\omega\ge 2\delta+\max_{[-2\delta,1+2\delta]}|f_b'|,
		\end{equation} 
		and $B>\omega$ such that
		\begin{equation}
			\label{4.63}
			\Big(\max_{[-2\delta,1+2\delta]}|f_b'|+ \mu^2-c\mu\Big)e^{-\mu B}<\Big(\max_{[-2\delta,1+2\delta]}|f_b'|+ \mu^2-c\mu\Big)e^{-\mu(B-\omega)}\le \delta.
		\end{equation} 
		
		\vskip 2mm
		\noindent\textit{Step 1: proof of~\eqref{4.57}}.  First of all,   it follows from Lemma \ref{lem:u:basic_propert} that there are $X_1>L$ and $T_1>0$ such that 
		\be\label{4.64}
		u(t,x)\le 1+\frac{\delta}{2}\ \hbox{ for all }t\ge T_1\hbox{ and }x\ge X_1. 
		\ee
		Moreover, since $u(t,x)$ has a Gaussian upper bound  at each fixed $t>0$ for all $|x|$ large enough by Lemma \ref{lemma1.3}, whereas $\phi(s)$ decays exponentially to $0$ as $s\to+\infty$ by~\eqref{2.6}, there is $A\ge B$ such that
		\begin{equation} 
			\label{4.65}
			u(T_1,x)\le \phi (x-X_1-A-C)+\delta~~\text{for all}~x\ge X_1.
		\end{equation}
		
		For $t\ge T_1$ and $x\ge X_1$, let us define
		$$\overline u(t,x)=\phi(\overline\xi(t,x))+\delta e^{-\delta(t- T_1)}+\delta e^{-\mu(x-X_1)},$$
		where 
		$$\overline\xi(t,x)=x-X_1-(c_b-c)(t- T_1)+\omega e^{-\delta(t- T_1)}-\omega-A-C.$$
		Let us check that $\overline u(t,x)$ is a supersolution to $u_t=  u_{xx}+cu_x+f_b(u)$ for $t\ge T_1$ and $x\ge X_1$. At time $T_1$, one has $\overline u(T_1,x)\ge \phi(x-X_1-A-C)+\delta\ge u(T_1,x)$ for all $x\ge X_1$, by~\eqref{4.65}. Moreover, since~$\overline \xi(t,X_1)\le -A-C<-C$ for $t\ge T_1$, one gets that $\overline u(t,X_1)\ge 1-\delta/2+\delta e^{-\delta(t- T_1)}+\delta\ge 1+\delta/2 \ge u(t,X_1)$ by~\eqref{phiC} and~\eqref{4.64}. Thus, it remains to check that $\mathcal{N}\overline u(t,x)\!:=\!\overline u_t(t,x)\!-\! \overline u_{xx}(t,x)\!-c\overline u_x(t,x)\!-\!f_b(\overline u(t,x))\!\ge\!0$ for all $t\ge T_1$ and $x\ge X_1$. After a straightforward computation, one derives
		\begin{align*}
			\mathcal{N}\overline u(t,x)=f_b(\phi(\overline \xi(t,x)))-f_b(\overline u(t,x))-\phi'(\overline \xi(t,x)))\omega\delta e^{-\delta(t- T_1)}-\delta^2 e^{-\delta(t- T_1)}-( \mu^2-c\mu) \delta e^{-\mu(x-X_1)}.
		\end{align*}
		We distinguish three cases: 
		\begin{itemize}
			\item if $\overline \xi(t,x)\le -C$, one has $1-\delta/2\le \phi(\overline\xi(t,x))<1$ by~\eqref{phiC}, hence $1+2\delta>\overline u(t,x)\ge 1-\delta/2$; it follows from~\eqref{4.60} that $f_b(\phi(\overline \xi(t,x)))-f_b(\overline u(t,x))\ge -(f_b'(1)/2)\big(\delta e^{-\delta(t- T_1)}+\delta e^{-\mu(x-X_1)}\big)$ and it then can be deduced from~\eqref{4.59}--\eqref{4.60} as well as the negativity of $\phi'$ and $f_b'(1)$ that
			\begin{align*}
				\mathcal{N} \overline u(t,x)&\ge -\frac{f_b'(1)}{2}\Big(\delta e^{-\delta(t- T_1)}+\delta e^{-\mu(x-X_1)}\Big)-\delta^2 e^{-\delta(t- T_1)}-( \mu^2-c\mu)\delta e^{-\mu(x-X_1)}\cr
				&= \Big(-\frac{f_b'(1)}{2}-\delta\Big)\delta e^{-\delta(t- T_1)}+\Big(-\frac{f_b'(1)}{2}- \mu^2+c\mu\Big)\delta  e^{-\mu(x-X_1)}>0;
			\end{align*}
			\item if $\overline\xi(t,x)\ge C$, one derives $0<\phi(\overline \xi(t,x))\le \delta$ by~\eqref{phiC}, and then $0<\overline u(t,x)\le 3\delta$; it follows from~\eqref{4.60} that $f_b(\phi(\overline \xi(t,x)))-f_b(\overline u(t,x))\ge -(f_b'(0)/2)\big(\delta e^{-\delta(t- T_1)}+\delta e^{-\mu(x-X_1)}\big)$; by virtue of~\eqref{4.59}--\eqref{4.60} and the negativity of $\phi'$ and $f_b'(0)$, there holds
			\begin{align*}
				\mathcal{N} \overline u(t,x)&\ge -\frac{f_b'(0)}{2}\Big(\delta e^{-\delta(t- T_1)}+\delta e^{-\mu(x-X_1)}\Big)-\delta^2 e^{-\delta(t- T_1)}-( \mu^2-c\mu)\delta e^{-\mu(x-X_1)}\cr
				&=  \Big(-\frac{f_b'(0)}{2}-\delta\Big)\delta e^{-\delta(t- T_1)}+\Big(-\frac{f_b'(0)}{2}- \mu^2+c\mu\Big)\delta e^{-\mu(x-X_1)}>0;
			\end{align*}
			\item if $-C\le\overline\xi(t,x)\le C$, it turns out that $x-X_1\ge (c_b-c)(t- T_1)-\omega e^{-\delta(t- T_1)}+\omega+A\ge (c_b-c)(t- T_1)+B$, whence $e^{-\mu(x-X_1)}\le e^{-\mu((c_b-c)(t- T_1)+B)}$. By~\eqref{4.60} and~\eqref{4.61}--\eqref{4.63}, one infers that
			\begin{align*}
				\mathcal{N}\overline u(t,x)&\ge -\!\max_{[-2\delta,1+2\delta]}\!|f_b'|\Big(\delta e^{-\delta(t- T_1)}\!+\!\delta e^{-\mu(x-X_1)}\Big)\!+\!\kappa\omega\delta e^{-\delta(t- T_1)}\!-\!\delta^2 e^{-\delta(t- T_1)}\!-\! ( \mu^2-c\mu) \delta e^{-\mu(x-X_1)}\cr
				&\ge \Big(\kappa\omega-\delta-\max_{[-2\delta,1+2\delta]}|f_b'|\Big)\delta e^{-\delta(t- T_1)}-\Big(\max_{[-2\delta,1+2\delta]}|f_b'|+ \mu^2-c\mu \Big)\delta e^{-\mu((c_b-c)(t- T_1)+B)}\cr
				&\ge \Big(\kappa\omega-2\delta-\max_{[-2\delta,1+2\delta]}|f_b'|\Big)\delta e^{-\delta(t- T_1)}\ge 0.
			\end{align*} 
		\end{itemize}
		\vskip -2mm
		
		As a consequence, we have proven that $\mathcal{N}\overline u(t,x):=\overline u_t(t,x)- \overline u_{xx}(t,x)-c\overline u_x(t,x)-f_b(\overline u(t,x))\ge 0$ for all $t\ge T_1$ and $x\ge X_1$. The maximum principle implies that 
		\begin{align*}
			u(t,x)\le \overline u(t,x)=\phi\big(x-X_1-(c_b-c)(t- T_1)+\omega e^{-\delta(t- T_1)}-\omega-A-C\big)+\delta e^{-\delta(t- T_1)}+\delta e^{-\mu(x-X_1)}
		\end{align*}
		for all $t\ge T_1$ and $x\ge X_1$, whence~\eqref{4.57} is achieved by taking $T_1=X_1$ and $z_1=-X_1-\omega-A-C$, since $\phi$ is decreasing.
		
		\vskip 2mm
		\noindent\textit{Step 2: proof of~\eqref{4.58}}.  Assume that $-c_m<c_b$ and $c\in(-c_m,c_b)$. Since $u(t,\cdot)\to 1$ as $t\to+\infty$ locally uniformly in $\R$ by \eqref{2.7}, one can choose $T_2>0$  large enough and $X_2>L$ such that 
		\begin{equation}\label{4.66}
			u(t,x)\ge 1-\delta~~\text{for all}~t\ge T_2\text{ and for all}~ x\in [X_2, X_2+ B+ 2C].
		\end{equation}
		
		For $t\ge T_2$ and $x\ge X_2$, we set
		$$\underline u(t,x)=\phi(\underline\xi(t,x))-\delta e^{-\delta(t-T_2)}-\delta e^{-\mu(x-X_2)},$$
		in which
		$$\underline\xi(t,x)=x-X_2-(c_b-c)(t-T_2)-\omega e^{-\delta(t-T_2)}+\omega-B-C.$$
		We shall check that $\underline u(t,x)$ is a subsolution to $u_t=  u_{xx}+cu_x+f_b(u)$ for all  $t\ge T_2$ and $x\ge X_2$. At time $t=T_2$, one has $ \underline u(T_2,x)\le 1-\delta-\delta e^{-\mu(x-X_2)}\le 1-\delta\le u(T_2,x)$ for $X_2\le x \le X_2+B+2C$ due to~\eqref{4.66}. For $x\ge X_2+B+2C$, one infers from $\underline\xi(T_2,x)\ge X_2+B+2C-X_2-B-C=C$ and \eqref{phiC} that $\phi(\underline\xi(T_2,x))\le \delta$, hence $\underline u(T_2,x)\le\delta-\delta-\delta e^{-\mu(x-X_2)}<0<u(T_2,x)$. In conclusion, $\underline u(T_2,x)\le u(T_2,x)$ for all $x\ge X_2$. At $x=X_2$, we have $\underline u(t,X_2)\le 1-\delta e^{-\delta(t-T_2)}-\delta< u(t,X_2)$ for all $t\ge T_2$, owing to~\eqref{4.66}. It thus suffices to check that $\mathcal{N}\underline u(t,x):=\underline u_t(t,x)-  \underline u_{xx}(t,x)-c\underline u_x(t,x)-f_b(\underline u(t,x))\le 0$ for all  $t\ge T_2$ and $x\ge X_2$. By a straightforward computation, one has
		$$\mathcal{N}\underline u(t,x)=f_b(\phi(\underline \xi(t,x)))-f_b(\underline u(t,x))+\phi'(\underline\xi(t,x))\omega\delta e^{-\delta(t-T_2)}+\delta^2 e^{-\delta(t-T_2)}+  ( \mu^2-c\mu)\delta e^{-\mu(x-X_2)}$$
		By analogy to Step 1, we consider three cases:
		\begin{itemize}
			\item	if $\underline \xi(t,x)\le -C$, then $1-\delta/2\le \phi(\underline \xi(t,x))<1$ by~\eqref{phiC} and thus $1>\underline u(t,x)\ge 1-3\delta$; thanks to~\eqref{4.60}, one has $f_b(\phi(\underline \xi(t,x)))-f_b(\underline u(t,x))\le (f_b'(1)/2)(\delta e^{-\delta(t-T_2)}+\delta e^{-\mu(x-X_2)})$; therefore, by using~\eqref{4.59}--\eqref{4.60} as well as the negativity of $\phi'$ and $f_b'(1)$, it comes that
			\begin{align*}
				\mathcal{N}\underline u(t,x)&< \frac{f_b'(1)}{2}\Big(\delta e^{-\delta(t-T_2)}+\delta e^{-\mu(x-X_2)}\Big)+\delta^2 e^{-\delta(t-T_2)}+  ( \mu^2-c\mu) \delta e^{-\mu(x-X_2)}\cr
				&=\Big(\frac{f_b'(1)}{2}+\delta\Big)\delta e^{-\delta(t-T_2)}+\Big(\frac{f_b'(1)}{2}+ \mu^2-c\mu\Big)\delta e^{-\mu(x-X_2)}<0;
			\end{align*}
			\item if $\underline \xi(t,x)\ge C$, then $0<\phi(\underline\xi(t,x))\le \delta$ by~\eqref{phiC} and thus $-2\delta<\underline u(t,x)\le \delta$; it follows from~\eqref{4.60} that $f_b(\phi(\underline \xi(t,x)))-f_b(\underline u(t,x))\le (f_b'(0)/2)(\delta e^{-\delta(t-T_2)}+\delta e^{-\mu(x-X_2)})$; therefore, owing to~\eqref{4.59}--\eqref{4.60} as well as the negativity of $\phi'$  and $f_b'(0)$, one infers that
			\begin{align*}
				\mathcal{N}\underline u(t,x)&< \frac{f_b'(0)}{2}\Big(\delta e^{-\delta(t-T_2)}+\delta e^{-\mu(x-X_2)}\Big)+\delta^2 e^{-\delta(t-T_2)}+   ( \mu^2-c\mu) \delta e^{-\mu(x-X_2)}\cr
				&=\Big(\frac{f_b'(0)}{2}+\delta\Big)\delta e^{-\delta(t-T_2)}+\Big(\frac{f_b'(0)}{2}+ \mu^2-c\mu\Big)\delta e^{-\mu(x-X_2)}<0;
			\end{align*}
			\item if $-C\le\underline\xi(t,x)\le C$, one has  $x-X_2\ge (c_b-c)(t-T_2)+\omega e^{-\delta(t-T_2)}-\omega+B\ge (c_b-c)(t-T_2)-\omega+B$, whence $e^{-\mu(x-X_2)}\le e^{-\mu((c_b-c)(t-T_2)+B-\omega)}$; by~\eqref{4.60} and~\eqref{4.61}--\eqref{4.63}, one deduces that
			\begin{align*}
				\mathcal{N}\underline u(t,x)&\le \max_{[-2\delta,1+2\delta]}|f_b'|\Big(\delta e^{-\delta(t-T_2)}+\delta e^{-\mu(x-X_2)}\Big)-\kappa\omega\delta e^{-\delta(t-T_2)}+\delta^2 e^{-\delta(t-T_2)}+   ( \mu^2-c\mu) \delta e^{-\mu(x-X_2)}\cr
				&\le \Big(\max_{[-2\delta,1+2\delta]}|f_b'|-\kappa\omega+\delta\Big)\delta e^{-\delta(t-T_2)}+\Big(\max_{[-2\delta,1+2\delta]}|f_b'|+ \mu^2-c\mu\Big)\delta e^{-\mu((c_b-c)(t-T_2)+B-\omega)}\cr
				&\le \Big(\max_{[-2\delta,1+2\delta]}|f_b'|-\kappa\omega+2\delta\Big)\delta e^{-\delta(t-T_2)}\le0.
			\end{align*}
		\end{itemize}
		\vskip -2mm
		
		Consequently, one has $\mathcal{N}\underline u(t,x):=\underline u_t(t,x)- \underline u_{xx}(t,x)-c\underline u_x(t,x)-f_b(\underline u(t,x))\le 0$ for all $t\ge T_2$ and $x\ge X_2$. The maximum principle implies that 
		\begin{align*}
			u(t,x)\ge \underline u(t,x)=\phi \big(x-X_2-(c_b-c)(t-T_2)-\omega e^{-\delta(t-T_2)}+\omega-B-C\big)-\delta e^{-\delta(t-T_2)}-\delta e^{-\mu(x-X_2)}
		\end{align*}
		for all $t\ge T_2$ and $x\ge X_2$. Therefore,~\eqref{4.58} is proved by taking $z_2=-X_2+\omega-B-C$, since $\phi$ is decreasing. The proof of Lemma~\ref{lemma 3.1} is thereby complete.
	\end{proof}

%
%
	
	More general than Lemma \ref{lemma 3.1}, we have:
	
	\begin{lemma}
		\label{lemma 3.2}
		Assume that $c<c_b$. Let $u$ be the solution of~\eqref{1.1'} with a nonnegative continuous and compactly supported initial datum $u_0\not\equiv 0$. Then for any $\varep>0$, 
		
		\begin{itemize}
			\item[(i)]
			there exist $X_{1,\varep}>L$, $T_{1,\varep}>0$ and $z_{1,\varep}\in\mathbb{R}$  such that
			\begin{equation}
				\label{4.67}
				u(t,x) \le \phi(x-(c_b-c)(t-T_{1,\varep})+ z_{1,\varep})+\varep e^{-\delta(t-T_{1,\varep})}+\varep e^{-\mu(x-X_{1,\varep})},~~~t\ge T_{1,\varep},~~x\ge X_{1,\varep},
			\end{equation}
			
			\item[(ii)] assume further that $-c_m<c_b$ and $c\in(-c_m,c_b)$, there exist  $X_{2,\varep}>L$,  $T_{2,\varep}>0$ and $z_{2,\varep}\in\mathbb{R}$ such that
			\begin{equation}
				\label{4.68}
				u(t,x)\ge\phi(x-(c_b-c)(t-T_{2,\varep})+z_{2,\varep})-\varep e^{-\delta(t-T_{2,\varep})}-\varep e^{-\mu(x-X_{2,\varep})},~~~t\ge T_{2,\varep},~~x\ge X_{2,\varep},
			\end{equation}
		\end{itemize}		
		with $\mu>0$,  $\delta>0$  as given in Lemma~$\ref{lemma 3.1}$.
	\end{lemma}

\noindent
\textbf{Remark}.
A straightforward consequence of Lemmas \ref{lemma 3.1}--\ref{lemma 3.2} is that \eqref{4.57}--\eqref{4.58} and \eqref{4.67}--\eqref{4.68} hold true under the assumption of Theorem \ref{thm_propagation-2} (i), i.e. when
		$-c_m<c_b$ and $c\in(-c_m,c_b)$.
		\vspace{3mm}
	
	\begin{proof}
		Let $\mu>0$, $\delta>0$, $C>0$, $\kappa>0$ and $\omega>0$ be defined as in~\eqref{4.59}--\eqref{4.62} (notice that these parameters are independent of $\varep$). It is immediate to see from Lemma~\ref{lemma 3.1} that, when $\varep\ge\delta$, the conclusion of Lemma~\ref{lemma 3.2} holds true with $X_{i,\varep}=X_i$, $T_{i,\varep}=T_i$ and $z_{i,\varep}=z_i$, for $i=1,2$. It remains to discuss the case
		$$0<\varep<\delta.$$
		For convenience, let us introduce some further parameters. Pick $C_\varep>C>0$ such that
		$$\phi\ge 1-\frac{\varep}{2}~\text{in}~(-\infty,-C_\varep]\ \hbox{ and }\ \phi\le \varep~\text{in}~[C_\varep,+\infty).$$
		Define
		\be\label{omegaeps}
		\omega_\varep:=\frac{\varep\omega}{\delta}>0.
		\ee
		Finally, let $B_\varep>\omega_\varep$ be large enough such that
		$$\Big(\max_{[-2\delta,1+2\delta]}|f_b'|+ \mu^2-c\mu\Big)e^{-\mu B_\varep}<\Big(\max_{[-2\delta,1+2\delta]}|f_b'|+ \mu^2-c\mu\Big)e^{-\mu(B_\varep-\omega_\varep)}\le \delta.$$
		
		\vskip 2mm
		\noindent\textit{Step 1: proof of~\eqref{4.67}}. By repeating the arguments used in the proof of~\eqref{4.64}--\eqref{4.65}   in  Step 1 of Lemma~\ref{lemma 3.1} and by replacing $\delta$ with $\varep$, there are $T_{1,\varep}$ large enough and $X_{1,\varep}>L$ such that $u(t,X_{1,\varep})\le 1+\varep/2$ for all $t\ge T_{1,\varep}$ and $u(T_{1,\varep},x)\le \phi (x-X_{1,\varep}-A_\varep-C_\varep)+\varep$ for all $x\ge X_{1,\varep}$, for some $A_\varep\ge B_\varep$. Define
		$$\overline u_\varep(t,x)=\phi(\overline\xi_\varep(t,x))+\varep e^{-\delta(t-T_{1,\varep})}+\varep e^{-\mu(x-X_{1,\varep})}~~\text{for}~t\ge T_{1,\varep}~\text{and} ~x\ge X_{1,\varep},$$
		where 
		$$\overline\xi_\varep(t,x)=x-X_{1,\varep}-(c_b-c)(t-T_{1,\varep})+\omega_\varep e^{-\delta(t-T_{1,\varep})}-\omega_\varep-A_\varep-C_\varep.$$
		Following the same lines as in Step 1 of Lemma~\ref{lemma 3.1}, one has $\overline{u}_\varep(X_{1,\varep},x)\ge u(X_{1,\varep},x)$ for all $x\ge X_{1,\varep}$, $\overline{u}_\varep(t,X_{1,\varep})\ge u(t,X_{1,\varep})$ for all $t\ge T_{1,\varep}$, and it can be deduced that $\overline u_\varep(t,x)$ is a supersolution to $u_t=  u_{xx}+cu_x+f_b(u)$ for all $t\ge T_{1,\varep}$ and $x\ge X_{1,\varep}$, by dividing the calculations into three cases:
		
		\begin{itemize}
			\item if $\underline \xi(t,x)\le -C$, then $1-\delta/2\le \phi(\underline \xi(t,x))<1$ by~\eqref{phiC}, hence $1+2\delta \ge1+2 \varep \ge \overline{u}_\varep(t,x)\ge 1-\delta/2 $; therefore, by using~\eqref{4.59}--\eqref{4.60} and the negativity of $\phi'$ and $f_b'(1)$, it follows that
			\begin{align*}
				\mathcal{N}\overline{u}_\varep(t,x)&\ge -  \frac{f_b'(1)}{2}\Big(\varep e^{-\delta(t-T_{1,\varep})}+\varep e^{-\mu(x-X_{1,\varep})}\Big)-\delta\varep e^{-\delta(t-T_{1,\varep})}-   ( \mu^2-c\mu) \varep e^{-\mu(x-X_{1,\varep})}\cr
				&=\Big(-\frac{f_b'(1)}{2}-\delta\Big)\varep e^{-\delta(t-T_{1,\varep})}+\Big(-\frac{f_b'(1)}{2}- \mu^2+c\mu\Big)\varep e^{-\mu(x-X_{1,\varep})}> 0;
			\end{align*}
			\item if $\underline \xi(t,x)\ge C$, then $0<\phi(\overline{\xi}_\varep(t,x))\le \delta$ by~\eqref{phiC} and thus $0<\overline{u}_\varep(t,x)\le 3\delta$; therefore, owing to~\eqref{4.59}--\eqref{4.60} as well as the negativity of $\phi'$  and $f_b'(0)$, it follows that
			\begin{align*}
				\mathcal{N}\overline{u}_\varep(t,x)&\ge  -\frac{f_b'(0)}{2}\Big(\varep e^{-\delta(t-T_{1,\varep})}+\varep e^{-\mu(x-X_{1,\varep})}\Big)-\delta\varep e^{-\delta(t-T_{1,\varep})}-  ( \mu^2-c\mu)\varep e^{-\mu(x-X_{1,\varep})}\cr
				&=\Big(-\frac{f_b'(0)}{2}-\delta\Big)\varep e^{-\delta(t-T_{1,\varep})}+\Big(-\frac{f_b'(0)}{2}- \mu^2+c\mu\Big)\varep e^{-\mu(x-X_{1,\varep})}\ge 0;
			\end{align*}
			\item if $-C\le\overline{\xi}_\varep(t,x)\le C$, one has $x-X_{1,\varep}\ge (c_b-c)(t-T_{1,\varep})-\omega_\varep e^{-\delta(t-T_{1,\varep})}+\omega_\varep+A_\varep+C_\varep-C\ge (c_b-c)(t-T_{1,\varep})+B_{\varep}$, hence $e^{-\mu(x-X_{1,\varep})}\le e^{-\mu((c_b-c)(t-T_{1,\varep})+B_{\varep})}$; since $\omega_\varep=\varep\omega/\delta$, one infers from~\eqref{4.60},~\eqref{4.61}--\eqref{4.62} and~\eqref{4.71}, that
			\begin{align*}
				\mathcal{N}\overline{u}_\varep(t,x)&\ge- \max_{[-2\delta,1+2\delta]}|f_b'|\Big(\varep e^{-\delta(t-T_{1,\varep})}
				+\varep e^{-\mu(x-X_{1,\varep})}\Big)\!+\!\kappa\omega_\varep\delta e^{-\delta(t-T_{1,\varep})}\!\\&\quad-\!\varep\delta e^{-\delta(t-T_{1,\varep})}\!-\!  ( \mu^2-c\mu) \varep e^{-\mu(x-X_{1,\varep})}\cr
				&\ge \Big(-\!\max_{[-2\delta,1+2\delta]}\!|f_b'|\!+\!\kappa\omega\!-\!\delta\!\Big)\varep e^{-\delta(t-T_{1,\varep})}\!-\!\Big(\!\max_{[-2\delta,1+2\delta]}\!|f_b'|\!+\! \mu^2-\!c\mu\!\Big)\varep e^{-\mu((c_b-c)(t-T_{1,\varep})+B_{\varep})}\cr
				&\ge \Big(\max_{[-2\delta,1+2\delta]}|f_b'|-\kappa\omega+2\delta\Big)\varep e^{-\delta(t-T_{1,\varep})}\ge  0.
			\end{align*}
		\end{itemize}
		
		Therefore, the maximum principle implies that
		$$u(t,x)\le \phi\big(x-X_{1,\varep}-(c_b-c)(t-T_{1,\varep})+\omega_\varep e^{-\delta(t-T_{1,\varep})}-\omega_\varep-A_\varep-C_\varep\big)+\varep e^{-\delta(t-T_{1,\varep})}+\varep e^{-\mu(x-X_{1,\varep})}$$
		for all $t\ge T_{1,\varep}$ and $x\ge X_{1,\varep}$. Consequently,~\eqref{4.67} follows by choosing $z_{1,\varep}=-X_{1,\varep}-\omega_\varep-A_\varep-C_\varep$.
		
		\vskip 2mm	
		\noindent\textit{Step 2: proof of~\eqref{4.68}}. Using the same argument as for the proof of~\eqref{4.66} with $\delta$ replaced by $\varep$, one infers that there exist $X_{2,\varep}>L$ and $T_{2,\varep}>0$ sufficiently large such that
		$$u(t,x)\ge  1-\varep~~\text{for all}~t\ge T_{2,\varep}~\text{and}~ x\in[X_{2,\varep},X_{2,\varep}+ B_\varep+ 2C_\varep].$$
		Then we set
		$$\underline u_\varep(t,x)=\phi(\underline\xi_\varep(t,x))-\varep e^{-\delta(t-T_{2,\varep})}-\varep e^{-\mu(x-X_{2,\varep})}~~~\text{for}~t\ge T_{2,\varep}~\text{and}~x\ge X_{2,\varep},$$
		in which
		$$\underline\xi_\varep(t,x)=x-X_{2,\varep}-(c_b-c)(t-T_{2,\varep})-\omega_\varep e^{-\delta(t-T_{2,\varep})}+\omega_\varep-B_\varep-C_\varep.$$
		As in the proof of~\eqref{4.58}, one can show that $\underline{u}_\varep(T_{2,\varep},x)\le u(T_{2,\varep},x)$ for all $x\ge X_{2,\varep}$, that $\underline{u}_\varep(t,X_{2,\varep})\le u(t,X_{2,\varep})$ for all $t\ge T_{2,\varep}$, and that $\underline u_\varep(t,x)$ is a subsolution of $u_t=  u_{xx}+cu_x+f_b(u)$ for all $t\ge T_{2,\varep}$ and $x\ge X_{2,\varep}$. By the maximum principle, one derives that
		$$u(t,x)\ge \phi\big( x-X_{2,\varep}-(c_b-c)(t-T_{2,\varep})-\omega_\varep e^{-\delta(t-T_{2,\varep})}+\omega_\varep-B_\varep-C_\varep\big)-\varep e^{-\delta(t-T_{2,\varep})}-\varep e^{-\mu(x-X_{2,\varep})}$$
		for all $t\ge T_{2,\varep}$ and $x\ge X_{2,\varep}$. Then~\eqref{4.68} follows by taking $z_{2,\varep}=-X_{2,\varep}+\omega_\varep-B_\varep-C_\varep$, since $\phi'<0$. The proof of Lemma~\ref{lemma 3.2} is thereby complete.
	\end{proof}
	
	Based on Lemmas~\ref{lemma 3.1} and~\ref{lemma 3.2}, we now provide the stability result of the bistable traveling front $\phi(x-(c_b-c)t)$ in the bistable region, under the assumption of Theorem \ref{thm_propagation-2} (i), i.e. when
	$-c_m<c_b$ and $c\in(-c_m,c_b)$.
	
	\begin{lemma}
		\label{lemma 3.4}
		Assume   that $-c_m<c_b$ and $c\in(-c_m,c_b)$. Let $\mu>0$, $\delta>0$, $C>0$, $\kappa>0$ and $\omega>0$ be as in~\eqref{4.59}--\eqref{4.62} in the proof of Lemma~$\ref{lemma 3.1}$.   If there exist $\varep\in(0,\delta]$, $t_0>0$, $x_0>L$ and $\xi\in\mathbb{R}$ such that
		\begin{equation}
			\label{4.70}
			\sup_{x\ge x_0}\big|u(t_0,x)-\phi(x-(c_b-c)t_0+\xi)\big|\le \varep,
		\end{equation}
		\begin{equation}\label{4.70'}
			1-\varep\le u(t,x_0)\le 1+\frac{\varep}{2}~~~~\hbox{ for all~ $t\ge t_0$},
			\ee
			$$\phi(x_0-(c_b-c)t_0+\xi)\ge 1-\frac{\varep}{2},$$ 
			and
			\begin{equation}
				\label{4.71}
				\Big(\max_{[-2\delta,1+2\delta]}|f_b'|+ \mu^2-c\mu\Big)e^{-\mu((c_b-c)t_0-x_0-\omega_\varep-\xi-C)}\le\delta
			\end{equation}
			with $\omega_\varep=\varep\omega/\delta$, then  there exists $\widetilde M>0$ such that the following holds true:
			$$\sup_{x\ge x_0}\big|u(t,x)-\phi(x-(c_b-c)t+\xi)\big|\le \widetilde M\varep~~\text{for all}~t\ge t_0.$$
		\end{lemma}
		
		\begin{proof}
			We first claim that 
			$$\overline u(t,x)=\phi(x-(c_b-c)t+\omega_\varep e^{-\delta(t-t_0)}-\omega_\varep+\xi)+\varep e^{-\delta(t-t_0)}+\varep e^{-\mu(x-x_0)}$$
			and 
			$$\underline u(t,x)=\phi(x-(c_b-c)t-\omega_\varep e^{-\delta(t-t_0)}+\omega_\varep+\xi)-\varep e^{-\delta(t-t_0)}-\varep e^{-\mu(x-x_0)}$$
			are, respectively, a super- and a subsolution of $u_t=  u_{xx}+cu_x+f_b(u)$ for $t\ge t_0$ and $x\ge x_0$. We just check that $\underline u(t,x)$ is a subsolution in detail (the supersolution can be handled in a similar way). 
			
			At time $t=t_0$, one has $\underline u(t_0,x)=\phi(x-(c_b-c)t_0+\xi)-\varep-\varep e^{-\mu(x-x_0)}\le u(t_0,x)$ for all $x\ge x_0$ thanks to~\eqref{4.70}. Moreover, $\underline u(t,x_0)=\phi(x_0-(c_b-c) t-\omega_\varep e^{-\delta(t-t_0)}+\omega_\varep+\xi)-\varep e^{-\delta(t-t_0)}-\varep\le 1-\varep\le u(t,x_0)$ for all $t\ge t_0$, owing to~\eqref{4.70'}. It then remains to show that $\mathcal{N}\underline u(t,x):=\underline u_t(t,x)- \underline u_{xx}(t,x)-c\underline u_x(t,x)-f_b(\underline u(t,x))\le 0$ for all $t\ge t_0$ and $x\ge x_0$. For convenience, we set
			$$\underline \xi(t,x):=x-(c_b-c)t-\omega_\varep e^{-\delta(t-t_0)}+\omega_\varep+\xi.$$
			By a straightforward computation, one has
			\begin{equation*}
				\mathcal{N}\underline u(t,x)=f_b(\phi(\underline \xi(t,x)))- f_b(\underline u(t,x))+\phi'(\underline\xi(t,x))\omega_\varep\delta e^{-\delta(t-t_0)}+\varep\delta e^{-\delta(t-t_0)}+   ( \mu^2-c\mu) \varep e^{-\mu(x-x_0)}.
			\end{equation*}
			We divide our analysis into three cases:
			\begin{itemize}
				\item if $\underline \xi(t,x)\le -C$, then $1-\delta/2\le \phi(\underline \xi(t,x))<1$ by~\eqref{phiC}, hence $1>\underline u(t,x)\ge 1-\delta/2-2\varep\ge 1-3\delta$; therefore, by using~\eqref{4.59}--\eqref{4.60} and the negativity of $\phi'$ and $f_b'(1)$, it follows that
				\begin{align*}
					\mathcal{N}\underline u(t,x)&\le  \frac{f_b'(1)}{2}\Big(\varep e^{-\delta(t-t_0)}+\varep e^{-\mu(x-x_0)}\Big)+\varep\delta e^{-\delta(t-t_0)}+   ( \mu^2-c\mu) \varep e^{-\mu(x-x_0)}\cr
					&=\Big(\frac{f_b'(1)}{2}+\delta\Big)\varep e^{-\delta(t-t_0)}+\Big(\frac{f_b'(1)}{2}+ \mu^2-c\mu\Big)\varep e^{-\mu(x-x_0)}\le 0;
				\end{align*}
				\item if $\underline \xi(t,x)\ge C$, then $0<\phi(\underline\xi(t,x))\le \delta$ by~\eqref{phiC} and thus $-2\delta\le-2\varep\le\underline u(t,x)\le \delta$; therefore, owing to~\eqref{4.59}--\eqref{4.60} as well as the negativity of $\phi'$  and $f_b'(0)$, it follows that
				\begin{align*}
					\mathcal{N}\underline u(t,x)&\le  \frac{f_b'(0)}{2}\Big(\varep e^{-\delta(t-t_0)}+\varep e^{-\mu(x-x_0)}\Big)+\varep\delta e^{-\delta(t-t_0)}+  ( \mu^2-c\mu)\varep e^{-\mu(x-x_0)}\cr
					&=\Big(\frac{f_b'(0)}{2}+\delta\Big)\varep e^{-\delta(t-t_0)}+\Big(\frac{f_b'(0)}{2}+ \mu^2-c\mu\Big)\varep e^{-\mu(x-x_0)}\le 0;
				\end{align*}
				\item if $-C\le\underline\xi(t,x)\le C$, one has $x-x_0\ge (c_b-c)(t-t_0)+(c_b-c)t_0-x_0+\omega_\varep e^{-\delta(t-t_0)}-\omega_\varep-\xi-C\ge (c_b-c)(t-t_0)+(c_b-c)t_0-x_0-\omega_\varep-\xi-C$, hence $e^{-\mu(x-x_0)}\le e^{-\mu((c_b-c)(t-t_0)+(c_b-c)t_0-x_0-\omega_\varep-\xi-C)}$; since $\omega_\varep=\varep\omega/\delta$, one infers from~\eqref{4.60},~\eqref{4.61}--\eqref{4.62} and~\eqref{4.71}, that
			\begin{align*}
				\mathcal{N}\underline u(t,x)&\le \max_{[-2\delta,1+2\delta]}|f_b'|\Big(\varep e^{-\delta(t-t_0)}+\varep e^{-\mu(x-x_0)}\Big)\!-\!\kappa\omega_\varep\delta e^{-\delta(t-t_0)}\!+\!\varep\delta e^{-\delta(t-t_0)}\!+\!  ( \mu^2-c\mu) \varep e^{-\mu(x-x_0)}\cr
				&\le \Big(\!\max_{[-2\delta,1+2\delta]}\!|f_b'|\!-\!\kappa\omega\!+\!\delta\!\Big)\varep e^{-\delta(t-t_0)}\!+\!\Big(\!\max_{[-2\delta,1+2\delta]}\!|f_b'|\!+\! \mu^2-\!c\mu\!\Big)\varep e^{-\mu((c_b-c)(t-t_0)+(c_b-c)t_0-x_0-\omega_\varep-\xi-C)}\cr
				&\le \Big(\max_{[-2\delta,1+2\delta]}|f_b'|-\kappa\omega+2\delta\Big)\varep e^{-\delta(t-t_0)}\le  0.
			\end{align*}
			\end{itemize}

			Eventually, one concludes that $\mathcal{N}\underline u(t,x):=\underline u_t(t,x)-c\underline u_x(t,x)- \underline u_{xx}(t,x)-f_b(\underline u(t,x))\le 0$ for all $t\ge t_0$ and~$x\ge x_0$. The maximum principle implies that 
			\begin{align*}
				u(t,x)\ge \underline u(t,x)= \phi\big(x-(c_b-c)t-\omega_\varep e^{-\delta(t-t_0)}+\omega_\varep+\xi\big)-\varep e^{-\delta(t-t_0)}-\varep e^{-\mu(x-x_0)}
			\end{align*}
			for all $t\ge t_0$ and $x\ge x_0$. For these $t$ and $x$, since $\phi'<0$, one derives that
			\begin{align*}
				u(t,x)\ge \phi(x-(c_b-c)t+\omega_\varep+\xi)-2\varep\ge \phi(x-(c_b-c)t+\xi)-\omega_\varep\Vert\phi'\Vert_{L^\infty(\mathbb{R})}-2\varep.
			\end{align*}
			Similarly, using especially that
			$$\Big(\max_{[-2\delta,1+2\delta]}|f_b'|+ \mu^2-c\mu\Big)e^{-\mu((c_b-c)t_0-x_0-\xi-C)}\le\Big(\max_{[-2\delta,1+2\delta]}|f_b'|+ \mu^2-c\mu\Big)e^{-\mu((c_b-c)t_0-x_0-\omega_\varep-\xi-C)}\le\delta$$
			by~\eqref{4.71}, one can also derive that $u(t,x)\le\overline{u}(t,x)=\phi\big(x-(c_b-c)t+\omega_\varep e^{-\delta(t-t_0)}-\omega_\varep+\xi\big)+\varep e^{-\delta(t-t_0)}+\varep e^{-\mu(x-x_0)}$ for all $t\ge t_0$ and $x\ge x_0$, hence
			\begin{align*}
				u(t,x)\le\phi(x-(c_b-c)t-\omega_\varep+\xi)+2\varep\le \phi(x-(c_b-c)t+\xi)+\omega_\varep\Vert\phi'\Vert_{L^\infty(\mathbb{R})}+2\varep.
			\end{align*}
			In conclusion, one has 
			$$\sup_{x\ge x_0}\big|u(t,x)-\phi(x-(c_b-c)t+\xi)\big|\le\omega_\varep\Vert\phi' \Vert_{L^\infty(\mathbb{R})}+2\varep=\widetilde{M}\varep~\text{for all}~t\ge t_0,$$
			where $\widetilde{M}:=\omega_\varep\Vert\phi' \Vert_{L^\infty(\mathbb{R})}/\varep+2=\omega\Vert\phi' \Vert_{L^\infty(\mathbb{R})}/\delta+2$ is independent of $\varep$, $t_0$, $x_0$ and $\xi$.	The proof of Lemma~\ref{lemma 3.4} is thereby complete.
		\end{proof}

		Now we are in a position to complete the proof of Theorem~\ref{thm_propagation-2}.
		
		\noindent
		\textit{Proof of Theorem~$\ref{thm_propagation-2}$ $($continued$)$.}  We first complete the proof of property (i). Remember that $-c_m<c_b$ and $c\in (-c_m,c_b)$.
		Let $X_1>L$, $X_2>L$, $T_1>0$, $T_2>0$, $z_1\in\mathbb{R}$, $z_2\in\mathbb{R}$, $\mu>0$ and $\delta>0$ be as in Lemma~\ref{lemma 3.1}, and let also $C>0$ be as in~\eqref{phiC} in the proof of Lemma~\ref{lemma 3.1}. For $t\ge\max(T_1,T_2)$ and $x\ge\max(X_1,X_2)>L$, Lemma~\ref{lemma 3.1} implies that
		\be\label{4.72}\baa{l}
		\phi(x-(c_b-c)(t-T_2)+z_2)-\delta e^{-\delta(t-T_2)}-\delta e^{-\mu(x-X_2)}\vspace{3pt}\\
		\qquad\qquad\qquad\qquad\qquad\le u(t,x)\le \phi(x-(c_b-c)(t-T_1)+z_1)+\delta e^{-\delta(t-T_1)}+\delta e^{-\mu(x-X_1)}.\eaa
		\ee
		Consider any given sequence $(t_n)_{n\in\mathbb{N}}$ such that $t_n\to+\infty$ as $n\to +\infty$. By standard parabolic estimates, the functions
		$$(t,z)\mapsto u_n(t,z):=u(t+t_n,z+(c_b-c)t_n)$$
		converge as $n\to+\infty$ up to extraction of a subsequence, locally uniformly in $(t,z)\in\mathbb{R}\times\mathbb{R}$, to a classical solution $u_\infty$ of $(u_\infty)_t= (u_\infty)_{zz}+c(u_\infty)_z+f_b(u_\infty)$ in $\mathbb{R}\times\mathbb{R}$. From~\eqref{4.72} applied at $(t+t_n,z+(c_b-c)t_n)$, the passage to the limit as $n\to+\infty$ gives
		\begin{align*}
			\phi(z-(c_b-c)(t-T_2)+z_2)\le u_\infty(t,z)\le\phi(z-(c_b-c)(t-T_1)+z_1)~~\text{for all}~(t,z)\in\mathbb{R}\times\mathbb{R}.
		\end{align*}
		Then, \cite[Theorem 3.1]{BH2007} can be adapted to yield that there exists $\xi\in\mathbb{R}$ such that $u_\infty(t,z)=\phi(z-(c_b-c)t+\xi)$ for all $(t,z)\in\mathbb{R}\times\mathbb{R}$, whence 
		\begin{equation}
			\label{4.73}
			u_n(t,z)\to \phi(z-(c_b-c) t+\xi)~\text{as}~n\to+\infty,~~\text{locally uniformly in}~(t,z)\in\mathbb{R}\times\mathbb{R}.
		\end{equation}
		
		Consider now any $\varep\in(0,\delta/3]$. Let $A_\varep>0$ be such that
		\be\label{phieps2}
		\phi\ge 1-\frac{\varep}{2}\hbox{ in $(-\infty,- A_\varep]\ $ and }\ \phi\le\frac\varep2\hbox{ in $[A_\varep,+\infty)$}.
		\ee
		Set $E_1:=\max\big(A_\varep-(c_b-c)T_1-z_1, A_\varep-\xi\big)$ and $E_2:=\min(-A_\varep-(c_b-c)T_2-z_2,-A_\varep-\xi\big)<E_1$. Then, it can be deduced from~\eqref{4.73} that 
		\begin{equation}
			\label{4.74}
			\sup_{E_2\le z \le E_1}\big|u_n(0,z)-\phi(z+\xi)\big|\le \varep~~\text{for all}~n~\text{large enough}.
		\end{equation}
		Since $t_n\to+\infty$ as $n\to+\infty$,~\eqref{4.72} and~\eqref{phieps2} imply that, for all $n$ large enough,
		\begin{equation}
			\label{4.75}
			\begin{aligned}
				\begin{cases}
					0<u_n(0,z)\le\varep~~~&\text{for all}~z\ge E_1,\cr
					1-\varep\le u_n(0,z)\le 1+\varep&\displaystyle\text{for all}~E_2-\frac{c_b-c}{2}t_n \le z\le E_2.
				\end{cases}
			\end{aligned}
		\end{equation}
		Furthermore, since $E_1\ge A_\varep-\xi$ and $E_2\le -A_\varep-\xi$, one has
		\begin{equation}
			\label{4.76}
			\begin{aligned}
				\begin{cases}
					\displaystyle 0<\phi(z+\xi)\le\frac{\varep}{2}<\varep&\text{for all}~z\ge E_1,\vspace{3pt}\cr
					\displaystyle 1-\varep<1-\frac{\varep}{2}\le \phi(z+\xi)<1&\text{for all}~z\le E_2.
				\end{cases}
			\end{aligned}
		\end{equation}
		Then~\eqref{4.75}--\eqref{4.76} imply that, for all $n$ large enough,
		$$\big|u_n(0,z)-\phi(z+\xi)\big|\le 2\varep~~\text{for all}~z\in\Big[E_2-\frac{c_b-c}{2}t_n,E_2\Big]\cup[E_1,+\infty).$$
		Together with~\eqref{4.74} and the definition of $u_n(t,z)$, one has, for all $n$ large enough,
		\begin{equation}
			\label{4.77}
			\big|u(t_n,x)-\phi(x-(c_b-c)t_n+\xi)\big|\le 2\varep~~\text{for all}~x\ge E_2+\frac{c_b-c}{2}t_n.
		\end{equation}
		
		On the other hand, one infers from Lemma~\ref{lemma 3.2} that, for all $n$ large enough, 
		\begin{align}
			\label{cdn1}
			1-3\varep\le  \phi(x-(c_b-c)(t_n-T_{2,\varep})+z_{2,\varep})-\varep e^{-\delta(t_n-T_{2,\varep})}-\varep e^{-\mu(x-X_{2,\varep})}\le u(t_n,x)~~~~~~~~~~~~\cr
			\le\phi(x-(c_b-c)(t_n-T_{1,\varep})+ z_{1,\varep})+\varep e^{-\delta(t_n-T_{1,\varep})}+\varep e^{-\mu(x-X_{1,\varep})}\le 1+2\varep,
		\end{align}
		for all $\max(X_1,X_2,X_{1,\varep},X_{2,\varep})\le x\le E_2+(c_b-c)t_n/2$, where $X_{1,\varep}>L$, $X_{2,\varep}>L$, $T_{1,\varep}>0$, $T_{2,\varep}>0$, $z_{1,\varep}\in\mathbb{R}$ and $z_{2,\varep}\in\mathbb{R}$ were given in Lemma~\ref{lemma 3.2}. Notice also that, for all $n$ large enough,
		\begin{equation}
			\label{cdn2}
			1-\varep\le\phi(x-(c_b-c)t_n+\xi)<1~~\text{for all}~\max(X_1,X_2,X_{1,\varep},X_{2,\varep})\le x \le E_2+\frac{c_b-c}{2}t_n.
		\end{equation} 
		One deduces from~\eqref{cdn1}--\eqref{cdn2}  that, for all $n$ large enough, 
		$$\big|u(t_n,x)-\phi(x-(c_b-c)t_n+\xi)\big|\le 3\varep~~\text{for all}~\max(X_1,X_2,X_{1,\varep},X_{2,\varep})\le x\le E_2+\frac{c_b-c}{2}t_n.$$
		Together with~\eqref{4.77}, one derives that, for all $n$ large enough,
		$$\big|u(t_n,x)-\phi(x-(c_b-c)t_n+\xi)\big|\le 3\varep~~\text{for all}~x\ge\max(X_1,X_2,X_{1,\varep},X_{2,\varep}).$$
		Choose $x_\varep=\max(X_1,X_2,X_{1,\varep},X_{2,\varep})$, then thanks to \eqref{2.7} and $c_b-c>0$,  it follows that for all $n$ large enough, 
		$$1-3\varep\le u(t,x_\varep)\le 1+\frac{3\varep}{2}~~\text{for all}~t\ge t_n,$$
		and
		$$\phi(x_\varep-(c_b-c) t_n+\xi)\ge 1-\frac{3\varep}{2},~~\Big(\max_{[-2\delta,1+2\delta]}|f_b'|+ \mu^2-c\mu\Big)e^{-\mu((c_b-c)t_n-x_\varep-3\varep\omega/\delta-\xi-C)}\le\delta.$$
		It then follows from Lemma~\ref{lemma 3.4} (applied with $t_0=t_n$, $x_0=x_\varep$ and $3\varep$ instead of $\varep$) that, for all $n$ large enough,
		$$\big|u(t,x)-\phi(x-(c_b-c)t+\xi)\big|\le 3\widetilde{M}\varep~~\text{for all}~t\ge t_n~\text{and}~ x\ge x_\varep,$$
		with $\widetilde{M}$ given in Lemma~\ref{lemma 3.4}. Since $\varep\in(0,\delta/3]$ was arbitrary, one finally infers that 
$$	\lim_{t\to+\infty}\Big(\sup_{x\ge x_\varep}|u(t,x)-\phi(x-(c_b-c)t+\xi)\Big)=0.$$
		Therefore, property (i) is achieved.

		It now remains to prove property (ii). Assume now that $-c_m<c=c_b$. Our goal is to show that $\sup_{x\ge \nu t}u(t,x)\to0$ as $t\to+\infty$ for every $\nu>0$. So let us fix $\nu>0$ in the sequel. For $\varep\in(0,(1-\theta)/2)$, let $f_{b,\varep}$ be a $C^1(\mathbb{R})$ function such that 
		$$\left\{\baa{l}
		f_{b,\varep}(0)=f_{b,\varep}(\theta)=f_{b,\varep}(1+\varep)=0,~~f_{b,\varep}'(0)<0,~~f_{b,\varep}'(1+\varep)<0,\vspace{3pt}\\
		f_{b,\varep}=f_b~\text{in}~(-\infty,1-\varep),~~f_{b,\varep}>0~\text{in}~(\theta,1+\varep),~~f_{b,\varep}<0~\text{in}~(1+\varep,+\infty).\eaa\right.$$
		We can also choose $f_{b,\varep}$ so that $f_{b,\varep}\ge f_b$ in $\R$, so that $f_{b,\varep}$ is decreasing in $[1-\varep,1+\varep]$. For each $\varep\in(0,(1-\theta)/2)$, let $\phi_\varep$ be the unique traveling front profile of $u_t=  u_{xx}+f_{b,\varep}(u)$ such that 
		$$ \phi_\varep''+c_{b,\varep}\phi_\varep'+f_{b,\varep}(\phi_\varep)=0\hbox{ in }\R,~~\phi_\varep'<0~\text{in}~\mathbb{R},~~\phi_\varep(0)=\theta,~~\phi_\varep(-\infty)=1+\varep,~~\phi_\varep(+\infty)=0,$$
		with speed $c_{b,\varep}>c_b$. It is standard to see that $\phi_\varep\to\phi$ in $C^2_{loc}(\R)$ and $c_{b,\varep}\to c_b$ as $\varep\to0$. We then fix $\varep\in(0,(1-\theta)/2)$ small enough such that $0<c_{b,\varep}-c_b<\nu$.

		By Lemma \ref{lem:u:basic_propert},	there are large enough $T>0$ and $X>L$ such that $u(t,x)\le 1+\varep/2$ for all $t\ge T$ and $x\ge X$. Since~$u(t,x)$ has  a Gaussian upper bound as $x\to+\infty$ at each fixed $t>0$ by Lemma \ref{lemma1.3},  whereas~$\phi_\varep(s)$ has an exponential decay (similar to~\eqref{2.6}) as $s\to+\infty$, it follows that there is $A>0$ such that $u(T,x)\le \phi_\varep (x-(c_{b,\varep}-c_b)T-A)$ for all $x\ge X$, and $u(t,X)\le\phi_\varep (X-(c_{b,\varep}-c_b)t-A)$ for all $t\ge T$ (here we also use the fact $c_{b,\varep}>c_b$ and $\phi_\varep(-\infty)=1+\varep$). By setting $\underline u(t,x):=\phi_\varep(x-(c_{b,\varep}-c_b)t-A)$ for $t\ge T$ and $x\ge X$,  a direct computation yields that
		\begin{align*}
			\overline u_t-\overline u_{xx}-c_b\overline u_x-f_b(\overline u)&=-(c_{b,\varep}-c_b)\phi_\varep'(\xi(t,x))-\phi_\varep''(\xi(t,x))-c_b\phi_\varep'(\xi(t,x))-f_b(\phi_\varep(\xi(t,x)))\\
			&=f_{b,\varep}(\phi_\varep(\xi(t,x)))-f_b(\phi_\varep(\xi(t,x)))\ge 0,~~~~~t\ge T,~x\ge X,
		\end{align*}
		with $\xi(t,x):=x-(c_{b,\varep}-c_b)t-A$, since $f_{b,\varep}\ge f_b$ in $\R$.
		The comparison principle implies that~$0<u(t,x)\le\phi_\varep (x-(c_{b,\varep}-c_b)t-A)$ for all $t\ge T$ and $x\ge X$, hence $\sup_{x\ge \nu t}u(t,x)\to0$ as $t\to+\infty$, since $c_{b,\varep}-c<\nu $ and $\phi_\varep(+\infty)=0$. 
		This completes the proof of Theorem~\ref{thm_propagation-2}.
	\end{proof}


	\section[Proofs of Theorems \ref{thm_c le -cm,cb>0}--\ref{thm_extinction}]{Conditional complete propagation vs. extinction when $c\in(-\infty,-c_m]$}

	
	\subsection[Proof of Theorem \ref{thm_c le -cm,cb>0}]{Conditional complete propagation: Proof of Theorem \ref{thm_c le -cm,cb>0}}
	In this subsection, we investigate the leftward propagation properties when $c\in(-\infty,-c_m]$ for ``large enough'' initial data, inspired by \cite{FM1977,HLZ2}. This is notably different from the preceding section. Here, we must consider both the KPP region and the bistable region simultaneously, leading to a more intricate analysis compared to the sections above.
		\begin{lemma}
		\label{lemma:c<-c_m}
		Assume that $c\le  -c_m$.  Let $u$ be the solution of~\eqref{1.1'} with a nonnegative continuous and compactly supported initial datum $u_0\not\equiv 0$. Then there exist  $T_3>0$, $X_3>L$, $z_3\in\mathbb{R}$, $\mu>0$ and $\delta>0$ such that
		\begin{align}\label{6.2-1}
			u(t,x)\le\phi\big(-x-(c_b+c)(t-T_3)+z_3\big)+\delta e^{-\delta(t-T_3)}+\delta e^{-\mu(x-X_3)}~~~\text{for}~t\ge T_3~\text{and}~ x\ge X_3.
		\end{align}
	\end{lemma}
\begin{proof}
We recall that $c_b<c_m$ which is already proved in the introduction. It implies that $-c_m<-c_b$, whence $c<-c_b$, namely $c+c_b<0$.   

We first introduce some parameters. Choose $\mu>0$ such that
\begin{equation}
	\label{6.2-mu}
	0<\mu<\frac{1}{2}\Big(c+\sqrt{c^2+2\min\big(|f_b'(0)|, |f_b'(1)|\big)}\Big).
\end{equation}
 Choose  $\delta>0$ such that 
\begin{equation}	
\label{6.2-delta}\left\{\baa{l}
\displaystyle 0<\delta< \min\Big(-\mu(c_b+c),\frac{1}{5},\frac{|f_b'(0)|}{2}, \frac{|f_b'(1)|}{2}\Big),\vspace{3pt}\\
\displaystyle f_b'\le \frac{f_b'(0)}{2}~\text{in}~[0,3\delta],~~f_b'\le \frac{f_b'(1)}{2}~\text{in}~[1-\delta,1+2\delta].\eaa\right.
\end{equation}
Let $C>0$ be such that
\begin{equation}
\label{6.2-C}
\phi\ge 1-\delta/2~\text{in}~(-\infty,-C]\ \hbox{ and }\ \phi\le \delta~\text{in}~[C,+\infty).
\end{equation}
Since $\phi'$ is negative and continuous in $\mathbb{R}$, there is $\kappa>0$ such that 
\begin{equation}
\label{6.2-kappa}
\phi'\le-\kappa<0~\text{in}~[-C, C].
\end{equation}
Finally, pick $\omega>0$ so large that
\begin{equation}
\label{6.2-omega}
\kappa\omega\ge 2\delta+\max_{[0,1+2\delta]}|f_b'|.
\end{equation} 
	and $A>\omega+C$ such that (recall that $c<0$)
\begin{equation}
	\label{6.2-A}
\Big(\max_{[0,1+2\delta]}|f_b'|+ \mu^2-c\mu\Big)e^{-\mu(A-\omega-C)}\le \delta.
\end{equation}

By Lemma \ref{lem:u:basic_propert}, we derive that there exists $T_3>0$  large enough such that  
\begin{equation}
\label{initial}
u(T_3,x)<1+ \delta/2~~\text{uniformly for}~x\in\R.
\end{equation}
Fix $X_3>L$ and $\overline{X}_3=X_3+A+C$. We infer from Lemma \ref{lemma_c<-c_m} that, up to increasing $T_3$, 
\begin{equation}
	\label{boundary}
	u(t,x) \le \delta~~~~~\text{for}~t\ge T_3, ~x \in [X_3,\overline{X}_3].
\end{equation}
When $x\ge \overline X_3$, we observe that $-x+X_3+A\le -C$ and thus
\begin{equation}
	\label{6.2-initial-phi}
	\phi(-x+X_3+A)\ge 1-\delta/2~~~~\text{for}~x\ge \overline X_3.
\end{equation}

For $t\ge T_3$ and $x\ge X_3$, let us now define
$$\overline u(t,x)=\phi(\overline\xi(t,x))+\delta e^{-\delta(t-T_3)}+\delta e^{-\mu (x-X_3)},$$
where 
$$\overline\xi(t,x)=-x+X_3-(c_b+c)(t-T_3)+\omega e^{-\delta(t-T_3)}-\omega+A.$$
Let us check that $\overline u(t,x)$ is a supersolution to \eqref{1.1'} for $t\ge T_3$ and $x\ge X_3$. At time $t=T_3$, one has $\overline u(T_3,x)\ge \phi(-x+X_3+A)+\delta\ge 1-\delta/2+\delta= 1+\delta/2> u(T_3,x)$ for  $x\ge \overline X_3$, by~\eqref{initial} and \eqref{6.2-initial-phi}. Thanks to \eqref{boundary}, we have $\overline u(T_3,x)\ge \delta \ge  u(T_3,x)$ for $x\in [X_3,\overline{X}_3]$. This implies that $\overline u(T_3,x)\ge u(T_3,x)$ for $x\ge X_3$. Moreover, for $t\ge T_3$, it is obvious that $\overline u(t,X_3)\ge\delta \ge u(t,X_3)$, due to \eqref{boundary}. Therefore, it remains to check that $\mathcal{N}\overline u(t,x)\!:=\!\overline u_t(t,x)\!-\! \overline u_{xx}(t,x)\!-c\overline u_x(t,x)\!-\!f_b(\overline u(t,x))\!\ge\!0$ for all $t\ge T_3$ and $x\ge X_3$. To this end, one derives from a straightforward computation that for $t\ge T_3$ and $x\ge X_3$,
\begin{align*}
\mathcal{N}\overline u(t,x)=f_b(\phi(\overline \xi(t,x)))-f_b(\overline u(t,x))-\phi'(\overline \xi(t,x)))\omega\delta e^{-\delta(t-T_3)}
-\delta^2 e^{-\delta(t-T_3)}-( \mu^2-c\mu) \delta e^{-\mu(x-X_3)}.
\end{align*}
We distinguish three cases: 
\begin{itemize}
\item if $\overline \xi(t,x)\le -C$, one has $1-\delta/2\le \phi(\overline\xi(t,x))<1$ by~\eqref{6.2-C}, hence $1-\delta/2\le \overline u(t,x)\le 1+2\delta$; it follows from~\eqref{6.2-delta} that $f_b(\phi(\overline \xi(t,x)))-f_b(\overline u(t,x))\ge -(f_b'(1)/2)\big(\delta e^{-\delta(t-T_3)}+\delta e^{-\mu (x-X_3)}\big)$ and it then can be deduced from the negativity of $\phi'$ and $f_b'(1)$ as well as \eqref{6.2-mu}--\eqref{6.2-delta} that
\begin{align*}
	\mathcal{N} \overline u(t,x)&\ge \Big(-\frac{f_b'(1)}{2}-\delta\Big)\delta e^{-\delta(t-T_3)}+\Big(-\frac{f_b'(1)}{2}-\mu^2+c\mu\Big)\delta e^{-\mu(x-X_3)}>0;
\end{align*}
\item if $\overline\xi(t,x)\ge C$, one derives $0<\phi(\overline \xi(t,x))\le \delta$ by~\eqref{6.2-C}, and then $0<\overline u(t,x)\le 3\delta$; it follows from~\eqref{6.2-delta} that $f_b(\phi(\overline \xi(t,x)))-f_b(\overline u(t,x))\ge -(f_b'(0)/2)\big(\delta e^{-\delta(t-T_3)}+\delta e^{-\mu (x-X_3)}\big)$; by virtue of the negativity of $\phi'$ and $f_b'(0)$ as well as \eqref{6.2-mu}--\eqref{6.2-delta}, there holds
\begin{align*}
	\mathcal{N} \overline u(t,x)&\ge  \Big(-\frac{f_b'(0)}{2}-\delta\Big)\delta e^{-\delta(t-T_3)}+\Big(-\frac{f_b'(0)}{2}-\mu^2+c\mu\Big)\delta e^{-\mu(x-X_3)}>0;
\end{align*}
\item if $-C\le\overline\xi(t,x)\le C$, then  $0<\overline u(t,x)\le 1+2\delta$. By noticing that $x-X_3\ge -(c_b+c)(t-T_3)-\omega+A-C$, we have $e^{-\mu(x-X_3)}\le e^{-\mu(-(c_b+c)(t-T_3)-\omega+A-C)}$. One infers from  \eqref{6.2-delta} and \eqref{6.2-kappa} -- \eqref{6.2-A} that
\begin{align*}
	\mathcal{N}\overline u(t,x)&\ge -\max_{[0,1+2\delta]}|f_b'|\Big(\delta e^{-\delta(t-T_3)}+\delta e^{-\mu (x-X_3)}\Big)+\kappa\omega\delta e^{-\delta(t-T_3)}-\delta^2 e^{-\delta(t-T_3)}-( \mu^2-c\mu) \delta e^{-\mu(x-X_3)}\cr
	&= \Big(\kappa\omega-\max_{[0,1+2\delta]}|f_b'|-\delta\Big)\delta e^{-\delta(t-T_3)}-\Big(\max_{[0,1+2\delta]}|f_b'|+\mu^2-c\mu\Big)\delta e^{-\mu(x-X_3)}\cr
	&\ge \Big(\kappa\omega-\max_{[0,1+2\delta]}|f_b'|-\delta\Big)\delta e^{-\delta(t-T_3)}-\Big(\max_{[0,1+2\delta]}|f_b'|+\mu^2-c\mu\Big)\delta e^{-\mu\big(
		-(c_b+c)(t-T_3)-\omega+A-C\big)}\cr
	&\ge\Big(\kappa\omega-\max_{[0,1+2\delta]}|f_b'|-2\delta\Big)\delta e^{-\delta(t-T_3)}\ge 0.
\end{align*} 
\end{itemize}
\vskip -2mm

As a consequence, we have proven that $\mathcal{N}\overline u(t,x)=\overline u_t(t,x)- \overline u_{xx}(t,x)-c\overline u_x(t,x)-f_b(\overline u(t,x))\ge 0$ for all $t\ge T_3$ and $x\ge X_3$. The comparison principle yields that 
\begin{align*}
u(t,x)\le \overline u(t,x)=\phi\big(-x+X_3-(c_b+c)(t-T_3)+\omega e^{-\delta(t-T_3)}-\omega+A\big)+\delta e^{-\delta(t-T_3)}+\delta e^{-\mu (x-X_3)}
\end{align*}
for all $t\ge T_3$ and $x\ge X_3$. By picking $z=A-\omega+X_3$, the conclusion follows.
\end{proof}

\noindent
{\it Proof of Theorem \ref{thm_c le -cm,cb>0}.}
		Assume now that $c\le -c_m$ and $c_b>0$. Remember that $u$ is the solution of~\eqref{1.1'} with any nonnegative continuous and compactly supported initial datum $u_0$
		 satisfying $u_0\ge \theta+\eta$ for any $\eta>0$ on an interval of size $L^*$.

		  \noindent
		 {\it Step 1. Propagation.} Let us  look at the Cauchy problem \eqref{1.1} with initial datum $v_0=u_0$ in $\R$. 
		 
	First of all, according to the hypothesis \eqref{hypf} of $f$, the solution $v$ of \eqref{1.1} will be bounded from below by the function $w$ which solves $w_t=w_{yy}+f_b(w)$ (the mass of $f_b$ is positive, since $c_b>0$) for $t>0$ and $x\in\R$ with initial condition $w_0=v_0$ in $\R$, thanks to the comparison principle. By virtue of Theorem 3.2 in Fife and McLeod \cite{FM1977},  there exist some constants $\xi_1\in\R$, $\xi_2\in\R$, $M>0$ and $\eta>0$ such that 
		\begin{equation}\label{2.7-w}
		\sup_{y<0}\big|w(t,y)-\phi(-y-c_bt+\xi_1)\big|+ \sup_{y>0}\big|w(t,y)-\phi(y-c_b t+\xi_2)\big|<M e^{-\eta t}~~~\text{for}~t>0.
		\end{equation}
	 This implies that the function $w$  develops into a pair of diverging fronts,  moving in  opposite directions. Since $w$ plays the role of a lower barrier for the function $v$, one concludes that	  $v$ indeed propagates to the right and to the left,  with propagation speeds no smaller than $c_b>0$.

	 Thanks to the upper barrier of $u$ constructed in Lemma \ref{lemma 3.1} (i) (by noticing $c<c_b$), one uses the transformation $v(t,y)=u(t,y-ct)$ for $(t,y)\in\R_+\times\R$ and immediately deduces that there exist $X_1>L$, $T_1>0$,  $\tau_1\in\mathbb{R}$,  $\bar \mu>0$ and $\bar\delta>0$ such that
	 \begin{equation}\label{6.18-0}
	 	v(t,y)\le\phi(y-c_bt+\tau_1)+\bar\delta e^{-\bar\delta(t-T_1)}+\bar\delta e^{-\bar\mu(y-ct-X_1)}~~\text{for all}~ t\ge T_1~\text{and} ~y\ge X_1+ct,
	 \end{equation}
  The remaining part of the proof will rely on the following several lemmas. 
	 
	 \begin{lemma}\label{lemma6.3_upper bound right}
		Assume that $c\le  -c_m$ and $c_b>0$.  Let $v$ be the solution of~\eqref{1.1} with any nonnegative continuous and compactly supported initial datum $v_0=u_0$ in $\R$, with $u_0$ as given in Theorem \ref{thm_c le -cm,cb>0}.
	 Then there exist $X_2>L$, $X_4>L$, $T_2>0$, $T_4>0$, $\tau_2\in\R$, $z_4\in\R$, $\mu>0$ and $\delta>0$ such that
		\begin{equation}
			\label{6.18}
			v(t,y)\ge\phi(y-c_bt+\tau_2)-\delta e^{-\delta(t-T_2)}-\delta e^{-\mu(y-X_2)}~~\text{for all}~ t\ge T_2~\text{and} ~y\ge X_2,
		\end{equation}
		\begin{equation}
			\label{6.19}
			\qquad\quad	v(t,y)\ge\phi(-y-c_bt+z_4)-\delta e^{-\delta(t-T_4)}-\delta e^{-\mu(X_4-y)}~~\text{for all}~ t\ge T_4~\text{and} ~y\le X_4.
		\end{equation}
	\end{lemma}
	\noindent
	{\bf Remark}.  Let $u$ be as given in Theorem \ref{thm_c le -cm,cb>0}. It is obvious that \eqref{6.19} is equivalent to
	\begin{equation}	\label{u-left-lower bd}
		u(t,x)\ge\phi(-x-(c_b+c)t+z_4)-\delta e^{-\delta(t-T_4)}-\delta e^{-\mu(X_4-x-ct)}~~\text{for all}~ t\ge T_4~\text{and} ~x\le X_4-ct,
	\end{equation}
	which will be more convenient in the investigation of leftward propagation of  Theorem \ref{thm_c le -cm,cb>0} later on.
	\begin{proof}
		We first introduce some parameters.  Let $\mu>0$, $\delta>0$, $C>0$, $\kappa>0$ and $\omega$ satisfy \eqref{4.59}--\eqref{4.62} and \eqref{6.2-mu}--\eqref{6.2-omega}. We further assume that
		\begin{equation}
		\label{2.7-mu}
		0<\mu<\sqrt{\min\left(\frac{|f_b'(0)|}{2},\frac{|f_b'(1)|}{2}\right)},
	\end{equation}
	\begin{equation}	
	\label{2.7-delta}
	\left\{\baa{l}
	\displaystyle 0<\delta< \min\Big(\mu c_b,\frac{1}{5}, \frac{|f_b'(0)|}{2}, \frac{|f_b'(1)|}{2}\Big), \vspace{3pt}\\
	\displaystyle f_b'\le \frac{f_b'(0)}{2}~\text{in}~[-2\delta,\delta],~~f_b'\le \frac{f_b'(1)}{2}~\text{in}~[1-3\delta,1],\eaa\right.
\end{equation}
and
		\begin{equation}
			\label{2.7-omega}
			\kappa\omega\ge 2\delta+\max_{[-2\delta,1]}|f_b'|.
		\end{equation} 
		We also choose some constant $B>\omega$ such that
		\begin{equation}
			\label{2.7-B}
		\Big(\max_{[-2\delta,1]}|f_b'|+ \mu^2\Big)e^{-\mu(B-\omega)}\le \delta.
		\end{equation}

		\vskip0.2cm
		\noindent
		{\it Step 1. Proof of \eqref{6.18}.}
		First,  \eqref{2.7-w} implies that 
		$v(t,y)\ge w(t,y)\ge \phi(y-c_bt+\xi_2)-M e^{-\eta t}$ for $t>0$ and $y\geq0$. In particular, 	
		 there exist $X_2>L$ and $T_2\ge \max( -\ln(\delta/(2M))/\eta,  (X_2+B+3C+ \xi_2)/c_b)>0$ such that for $X_2\le y\le X_2+B+2C$ and $t\ge T_2$, there holds
		\begin{equation}\label{2.7-initial}
			\begin{aligned}
				v(t,y)\ge w(t,y)&\ge \phi(y-c_b t+\xi_2)-M e^{-\eta t}\\
				&\ge \phi(X_2+B+2C-c_b T_2+\xi_2)-M e^{-\eta T_2}\ge \phi(-C)-\delta/2\ge 1-\delta.
			\end{aligned}
		\end{equation}
		
		For $t\ge T_2$ and $y\ge X_2$, we define
		\begin{equation*}
			\underline v(t,y)=\phi(\xi(t,y))-\delta e^{-\delta(t-T_2)}-\delta e^{-\mu(y-X_2)},
		\end{equation*}
		with $\xi(t,y)=y-X_2-c_b(t-T_2)-\omega e^{-\delta(t-T_2)}+\omega-B-C$. We are going to verify that $\underline v$ is a subsolution for \eqref{1.1} for $t\ge T_2$ and $y\ge X_2$. At $t=T_2$, we notice that $\underline v(T_2,y)\le 1-\delta\le v(T_2,y)$ for $X_2\le y\le X_2+B+2C$ by virtue of \eqref{2.7-initial}, while in the region where $y\ge X_2+B+2C$, since $\xi( T_2,y)\ge B+2C-B-C=C$, it follows that $\underline v(T_2,y)\le \delta-\delta-\delta e^{-\mu(y-X_2)}<0<v(T_2,y)$. Therefore, we have $\underline v(T_2,y)\le v(T_2,y)$ for $y\ge X_2$. Moreover, we also observe from \eqref{2.7-initial} that $\underline v(t,X_2)\le 1-\delta e^{-\delta(t-T_2)}-\delta\le v(t, X_2)$ for $t\ge T_2$. It remains to check that $\underline v_t(t,y)-\underline v_{yy}(t,y)-f(y-ct,\underline v(t,y))\le 0$ for $t\ge T_2$ and $y\ge X_2$. By noticing that  $ \underline v_t(t,y)-\underline v_{yy}(t,y)-f(y-ct,\underline v(t,y))\le  \underline v_t(t,y)-\underline v_{yy}(t,y)-f_b(\underline v(t,y))=:\mathcal{L}\underline v(t,y)$ for $t\ge 0$ and $y\in \R$, it suffices to prove that  $\mathcal{L}\underline v(t,y)\le 0$ for $t\ge T_2$ and $y\ge X_2$. After a straightforward computation, one derives 
		\begin{equation*}
			\mathcal{L}\underline v(t,y)=f_b(\phi(\xi(t,y)))-f_b(\underline v(t,y))+\phi'(\xi(t,y))\omega\delta e^{-\delta(t-T_2)}+\delta^2 e^{-\delta(t-T_2)}+\mu^2\delta e^{-\mu(y-X_2)}.
		\end{equation*}
		By doing the analysis according to three cases, namely, $\xi(t,y)\le -C$, $\xi(t,y)\ge C$ and $-C\le \xi(t,y)\le C$  based on the choices of the parameters above, eventually we obtain that $\mathcal{L}\underline v(t,y)=\underline v_t(t,y)-\underline v_{yy}(t,y)-f_b(\underline v(t,y))\le 0$ for $t\ge T_2$ and $y\ge X_2$. The comparison principle then implies that  for $t\ge T_2$ and $y\ge X_2$
		\begin{equation*}
			v(t,y)\ge  \phi(y-X_2-c_b(t-T_2)-\omega e^{-\delta(t-T_2)}+\omega-B-C)-\delta e^{-\delta(t-T_2)}-\delta e^{-\mu(y-X_2)}.
		\end{equation*}
		By taking $\tau_2=-X_2+c_b T_2+\omega-B-C$, we then reach \eqref{6.18}, as stated.	
		
		\vskip0.2cm
		\noindent
		 {\it Step 2. Proof of \eqref{6.19}.}  Fix $X_4>L+B+C$.
			We infer from \eqref{2.7-w} that 
			$v(t,y)\ge w(t,y)\ge \phi(-y-c_bt+\xi_1)-M e^{-\eta t}$ for $t>0$ and $y\leq0$, and that $v(t,y)\ge w(t,y)\ge \phi(y-c_bt+\xi_2)-M e^{-\eta t}$ for $t>0$ and $y\geq0$. Then, there exists $T_4>\max\big((2C+\xi_1)/c_b, -(1/\eta)\ln(\delta/(2M)),(C+X_4+\xi_2)/c_b\big)>0$ such that
			$$
			\begin{aligned}
				v(t,y)\ge w(t,y)&\ge \phi(-y-c_b t+\xi_1)-M e^{-\eta t}\\
				&\ge \phi(C-c_b T_4 +\xi_1)-M e^{-\eta T_4}\ge \phi(-C)-\delta/2\ge 1-\delta,~~~~~ t\ge T_4,~~-C\le y\le 0,
			\end{aligned}
			$$
			and 
			$$\begin{aligned}
				v(t,y)\ge w(t,y)&\ge \phi(y-c_b t+\xi_2)-M e^{-\eta t}\\
				&\ge \phi(X_4-c_b T_4 +\xi_2)-M e^{-\eta T_4}\ge \phi(-C)-\delta/2\ge 1-\delta, ~~~~t\ge T_4,~~0\le y\le X_4.
			\end{aligned}$$
			Combining the above two inequalities, it follows that
			\begin{equation}
				\label{2.7_initial'}v(t,y)\ge 1-\delta~~~\text{for}~~ t\ge T_4~~\text{and}~-C\le y\le X_4.
			\end{equation}
			
			For $t\ge T_4$ and $y\le X_4$, set now
			\begin{equation*}
				\breve v(t,y)=\phi( \zeta(t,y))-\delta e^{-\delta(t-T_4)}-\delta e^{-\mu(X_4-y)},
			\end{equation*}
			with $\zeta(t,y)=-y+X_4-c_b(t-T_4)-\omega e^{-\delta(t-T_4)}+\omega-B-C$. We are going to show that $\breve{v}$ is a subsolution for \eqref{1.1} for $t\ge T_4$ and $y\le X_4$. Indeed, at time $t=T_4$, we notice that $\breve{v}(T_4,y)\le 1-\delta-\delta e ^{-\mu(X_4-y)}\le v(T_4,y)$ for $-C\le y\le X_4$ by virtue of \eqref{2.7_initial'}, and moreover, the observation that $\zeta(T_4,y)\ge C$ for $y\le -C$ implies that $\breve{v}(T_4,y)\le \delta-\delta-\delta e^{-\mu(X_4-y)}\le 0<v(T_4,y)$ for $y\le -C$. Thus, $\breve v(T_4,y)\le v(T_4, y)$ for $y\le X_4.$ On the other hand, \eqref{2.7_initial'} also implies that $v(t,X_4)\ge 1-\delta\ge \breve v(t,X_4)$ for $t\ge T_4$. The remaining parts can be proved similarly to those in Step 1 with $z_4=X_4+c_b T_4+\omega-B-C$.
	\end{proof} 
	
	At this stage, let us combine \eqref{6.18-0} and \eqref{6.18}, by repeating the lines as in the proof of Theorem \ref{thm_propagation-2} (together with parallel results to the technical Lemmas \ref{lemma 3.2}--\ref{lemma 3.4} and with \cite[Theorem 3.1]{BH2007}), eventually we reach the following stability result: there exist $X>L$ and some constant $z_1\in\R$ such that
		\begin{equation*}
		\sup_{y\ge X} |v(t,y)-\phi(y-c_bt+z_1)|\to0~~~\text{as}~t\to+\infty,
	\end{equation*}
		Equivalently,
	\begin{equation}\label{2.7-3}
		\sup_{ x\ge X-ct} |u(t,x)-\phi(x-(c_b-c)t+z_1)|\to0~~~\text{as}~t\to+\infty.
	\end{equation}
	This proves the stability result \eqref{equ:c_le_cm:r} in Theorem \ref{thm_c le -cm,cb>0}, which also demonstrates that the rightward propagation speed is $c_b-c$.
	
	From now on, we will focus on the proof of the ``leftward'' propagation with speed $c_b+c<0$. To start with, we give a stronger estimate than \eqref{6.2-1} in Lemma \ref{lemma:c<-c_m}.

\begin{lemma}
	\label{lemma-6.4}
		Assume that $c\le  -c_m$ and $c_b>0$.  Let $u$ be the solution of~\eqref{1.1'} with a nonnegative continuous and compactly supported initial datum $u_0\not\equiv 0$. Then for any $\varep>0$, there exist $X_{3,\varep}>L$,  $T_{3,\varep}>0$ and  $z_{3,\varep}\in\mathbb{R}$ such that
	\begin{align}\label{6.4-1}
		u(t,x)\le\phi\big(-x-(c_b+c)(t-T_{3,\varep})+z_{3,\varep}\big)+\varep e^{-\delta(t-T_{3,\varep})}+\varep e^{-\mu(x-X_{3,\varep})},~~~t\ge T_{3,\varep},~ x\ge X_{3,\varep},
	\end{align}
with $\mu>0$ and $\delta>0$ as given in Lemma \ref{lemma6.3_upper bound right}. 
\end{lemma}

\begin{proof}

 	Let $\mu>0$, $\delta>0$, $C>0$, $\kappa>0$ and $\omega>0$ be defined as in the proof of Lemma \ref{lemma6.3_upper bound right} (all independent of $\varep$). It follows from Lemma \ref{lemma:c<-c_m} that the conclusion of Lemma \ref{lemma-6.4} immediately holds true when $\varep\ge \delta$, with $z_{3,\varep}=X$ and $T_{3,\varep}=T$. It is left to deal with the case where $$0<\varep<\delta.$$
 	
 	Let us introduce for convenience further parameters. Fix $C_\varep>C>0$ such that
 	\begin{equation}
 		\label{6.4-Cvarep}
 		\phi\ge 1-\frac{\varep}{2}~~~\text{in}~(-\infty,-C_\varep],~~~~\phi\le \frac{\varep}{2}~~~\text{in}~[C_\varep,+\infty).
 	\end{equation}
 Denote
 \begin{equation}
 	\label{6.4-omega_varep}
 	\omega_\varep:=\frac{\varep\omega}{\delta}.
 \end{equation}
Finally, we choose $A_\varep>\omega_\varep+C_\varep$ such that (notice that $\mu^2-c\mu+\max_{[0,1+2\delta]}|f_b'|>0$)
\begin{equation}
	\label{6.4-A}
	\big(\mu^2-c\mu+\max_{[0,1+2\delta]}|f_b'|\big)e^{-\mu(A_\varep-\omega_\varep-C_\varep)}<\delta.
\end{equation}

We simply repeat the arguments used in the proofs of \eqref{initial} and \eqref{boundary} by replacing $\delta$ with $\varep$. Lemma \ref{lem:u:basic_propert} gives the existence of $T_{3,\varep}>0$ large enough such that 
\begin{equation}
	\label{6.4-initial}
	u(T_{3,\varep},x)<1+\frac{\varep}{2}~~~~\text{uniformly for}~x\in\R.
\end{equation} 
Fix $X_{3,\varep}>L$ and $\overline X_{3,\varep}:= X_{3,\varep}+A_\varep+C_\varep$. We deduce from Lemma \ref{lemma_c<-c_m} that, up to increasing $T_{3,\varep}$,
\begin{equation}
	\label{6.4-boundary}
	u(t,x)<\varep~~~~~~~~~~~\text{for}~t\ge T_{3,\varep}~~\text{and}~ ~x \in [X_{3,\varep},\overline X_{3,\varep}].
\end{equation}

Define now for $t\ge T_{3,\varep}$ and $x\ge X_{3,\varep}$,
\begin{equation*}
	\overline u(t,x)=\phi(\overline\xi(t,x))+\varep e^{-\delta(t-T_{3,\varep})}+\varep e^{-\mu(x-X_{3,\varep})},
\end{equation*}
with $$\overline\xi(t,x):=-x+X_{3,\varep}-(c_b+c)(t-T_{3,\varep})+\omega_\varep e^{-\delta(t-T_{3,\varep})}-\omega_\varep+A_\varep.$$
We now check that $\overline u$ is a supersolution to \eqref{1.1'} for $t\ge T_{3,\varep}$ and $x\ge X_{3,\varep}$. In fact, at time $T_{3,\varep}$, we infer from \eqref{6.4-boundary} that
\begin{equation}
	\overline u(T_{3,\varep}, x)\ge \varep> u(T_{3,\varep},x)~~~~~~ \text{ for } X_{3,\varep}\le x\le\overline X_{3,\varep}.
\end{equation}
On the other hand, in the case where $x\ge \overline X_{3,\varep}$, that is, $-x+X_{3,\varep}+A_\varep\le -C_\varep$, \eqref{6.4-Cvarep} and \eqref{6.4-initial} imply
\begin{equation}
		\overline u(T_{3,\varep}, x)\ge \phi(-x+X_{3,\varep}+A_\varep)+\varep\ge \phi(-C_\varep)+\varep\ge 1+\frac{\varep}{2}>u(T_{3,\varep},x).
\end{equation}
Moreover, at $x=X_{3,\varep}$, it is easy to see that $\overline u(t,X_{3,\varep})\ge\varep>u(t,X_{3,\varep})$ for all $t\ge T_{3,\varep}$ due to \eqref{6.4-boundary}. It is then left to check that $\mathcal{N}\overline u(t,x):=\overline u_t(t,x)-\overline u_{xx}(t,x)-c\overline u_x(t,x)-f_b(\overline u(t,x))\ge0$ for $t\ge T_{3,\varep}$ and $x\ge X_{3,\varep}$.

After a straightforward computation, we arrive at
\begin{align*}
	\mathcal{N}\overline u(t,x)=f_b(\phi(\overline\xi(t,x)))-f_b(\overline u(t,x))-\delta\omega_\varep e^{-\delta(t-T_{3,\varep})}\phi'(\overline\xi(t,x)) -\delta \varep e^{-\delta(t-T_{3,\varep})}+(c\mu-\mu^2)\varep e^{-\mu(x-X_{3,\varep})}.
\end{align*}
We divide into three cases.
\begin{itemize}
	\item if $\overline\xi(t,x)\le -C$, one has $1-\varep/2\le\phi(\overline \xi(t,x))<1$, therefore $1-\delta<1-\varep/2\le\overline u(t,x)<1+2\varep<1+2\delta$. Thus, it follows from \eqref{6.2-mu}--\eqref{6.2-delta} that
	\begin{align*}
	\mathcal{N} \overline u(t,x)&\ge \Big(-\frac{f_b'(1)}{2}-\delta\Big)\varep e^{-\delta(t-T_{3,\varep})}+\Big(-\frac{f_b'(1)}{2}-\mu^2+c\mu\Big)\varep e^{-\mu(x-X_{3,\varep})}>0;
\end{align*}
\item if $\overline\xi(t,x)\ge C$, one derives $0<\phi(\overline \xi(t,x))\le \varep/2<\delta$, and then $0<\overline u(t,x)\le 3\varep<3\delta$. Together with \eqref{6.2-mu}--\eqref{6.2-delta}, there then holds
\begin{align*}
\mathcal{N} \overline u(t,x)&\ge  \Big(-\frac{f_b'(0)}{2}-\delta\Big)\varep e^{-\delta(t-T_{3,\varep})}+\Big(-\frac{f_b'(0)}{2}-\mu^2+c\mu\Big)\varep e^{-\mu(x-X_{3,\varep})}>0;
\end{align*}
\item if $-C\le\overline\xi(t,x)\le C$, then  $0<\overline u(t,x)\le 1+2\varep<1+2\delta$. By noticing that $x-X_{3,\varep}\ge-(c_b+c)(t-T_{3,\varep})-\omega_\varep+A_\varep-C\ge -(c_b+c)(t-T_{3,\varep})-\omega_\varep+A_\varep-C_\varep$, we have $e^{-\mu(x-X_{3,\varep})}\le e^{-\mu\big(-(c_b+c)(t-T_{3,\varep})-\omega_\varep+A_\varep-C_\varep\big)}$. One then infers from  \eqref{6.2-delta}, \eqref{6.2-kappa}, \eqref{6.2-omega} and \eqref{6.4-A} that
\begin{align*}
\mathcal{N}\overline u(t,x)\ge& -\max_{[0,1+2\delta]}|f_b'|\Big(\varep e^{-\delta(t-T_{3,\varep})}+\varep e^{-\mu (x-X_{3,\varep})}\Big)+\kappa\omega_\varep\delta e^{-\delta(t-T_{3,\varep})}\\
&\qquad\qquad\qquad\qquad\qquad\qquad-\delta \varep e^{-\delta(t-T_{3,\varep})}-( \mu^2-c\mu) \varep e^{-\mu(x-X_{3,\varep})}\cr
=& \Big(\kappa\omega-\max_{[0,1+2\delta]}|f_b'|-\delta\Big)\varep e^{-\delta(t-T_{3,\varep})}-\Big(\max_{[0,1+2\delta]}|f_b'|+\mu^2-c\mu\Big)\varep e^{-\mu(x-X_{3,\varep})}\cr
\ge& \Big(\kappa\omega-\max_{[0,1+2\delta]}|f_b'|-\delta\Big)\varep e^{-\delta(t-T_{3,\varep})}-\Big(\max_{[0,1+2\delta]}|f_b'|+\mu^2-c\mu\Big)\varep e^{-\mu\big(
	-(c_b+c)(t-T_{3,\varep})-\omega_\varep+A_\varep-C_\varep\big)}\cr
\ge&\Big(\kappa\omega-\max_{[0,1+2\delta]}|f_b'|-2\delta\Big)\varep e^{-\delta(t-T_{3,\varep})}\ge 0.
\end{align*} 
\end{itemize}
Consequently, we have shown that $\mathcal{N}\overline u(t,x)=\overline u_t(t,x)-\overline u_{xx}(t,x)-c\overline u_x(t,x)-f_b(\overline u(t,x))\ge0$ for $t\ge T_{3,\varep}$ and $x\ge X_{3,\varep}$. The maximum principle implies that 
\begin{equation*}
	u(t,x)\le \overline u(t,x)= \phi\big(-x+X_{3,\varep}-(c_b+c)(t-T_{3,\varep})+\omega_\varep e^{-\delta(t-T_{3,\varep})}-\omega_\varep+A_\varep\big)+\varep e^{-\delta(t-T_{3,\varep})}+\varep e^{-\mu(x-X_{3,\varep})}
\end{equation*}
for  $t\ge T_{3,\varep}$ and $x\ge X_{3,\varep}$. Thus, \eqref{6.4-1} is achieved by taking $z_{3,\varep}=X_{3,\varep}-\omega_\varep+A_\varep$. This completes the proof of Lemma \ref{lemma-6.4}.
\end{proof}

In the same spirit of Lemma \ref{lemma-6.4}, we  give below a strong version of \eqref{6.19} in Lemma \ref{lemma6.3_upper bound right}, which concerns the lower estimate of the leftward propagation front.

\begin{lemma}
	\label{lemma-6.5}
	Assume that $c\le  -c_m$ and $c_b>0$.  Let $v$ be as given in Lemma \ref{lemma6.3_upper bound right}.  Then for any $\varep>0$, there exist  $X_{4,\varep}>L$, $T_{4,\varep}>0$ and $z_{4,\varep}\in\R$ such that 
	\begin{equation}
		\label{6.5}
		v(t,y)\ge\phi(-y-c_b(t-T_{4,\varep})+z_{4,\varep})-\varep e^{-\delta(t-T_{4,\varep})}-\varep e^{-\mu(X_{4,\varep}-y)}~~~\text{for}~ t\ge T_{4,\varep}~\text{and} ~y\le X_{4,\varep},
	\end{equation}
	with $\mu>0$ and $\delta>0$ as given in  Lemma \ref{lemma6.3_upper bound right}.
\end{lemma}
Let $u$ be as given in Theorem \ref{thm_c le -cm,cb>0}. It turns out that \eqref{6.5} is parallel to the following estimate of $u$ for $t\ge T_{4,\varep}$ and  $x\le X_{4,\varep}-ct$,
\begin{equation}
	\label{6.5-u}
	u(t,x)\ge\phi(-x-(c+c_b)(t-T_{4,\varep})-cT_{4,\varep}+z_{4,\varep})-\varep e^{-\delta(t-T_{4,\varep})}-\varep e^{-\mu(X_{4,\varep}-x-ct)},
\end{equation}
with $\mu>0$ and $\delta>0$ as given in  Lemma \ref{lemma6.3_upper bound right}. 

\begin{proof}The proof  is essentially very similar to that of Lemma \ref{lemma-6.4}. We sketch the details for the sake of completeness.
	
	Let $\mu>0$, $\delta>0$, $C>0$, $\kappa>0$ and $\omega>0$ be defined as in the proof of Lemma \ref{lemma6.3_upper bound right} (independent of $\varep$). We first notice that the case of $\varep\ge\delta$ is done in Lemma \ref{lemma6.3_upper bound right}. Therefore, it is sufficient to consider the case of $0<\varep<\delta$. As before, we introduce parameters $C_\varep>0$ and $\omega_\varep>0$ as in \eqref{6.4-Cvarep}--\eqref{6.4-omega_varep} for convenience.

Fix $X_{4,\varep}>C_\varep+\omega_\varep+L$ such that 
\begin{equation}
	\label{6.5-X}
	\Big(\max_{[-2\delta,1]}|f_b'|+ \mu^2\Big)e^{-\mu(X_{4,\varep}-C_\varep-\omega_\varep)}<\delta.
\end{equation}
We infer from \eqref{2.7-w} that 
$v(t,y)\ge w(t,y)\ge \phi(-y-c_bt+\xi_1)-M e^{-\eta t}$  for $t>0$ and $y\le 0$, 
and that
$v(t,y)\ge w(t,y)\ge \phi(y-c_bt+\xi_2)-M e^{-\eta t}$  for  $t>0$ and $y\ge 0$.  We can then choose $T_{4,\varep}>\max\big((2C_\varep+\xi_1)/c_b,(C_\varep+X_{4,\varep}+\xi_2)/c_b, -(1/\eta)\ln(\varep/(2M))\big)>0$ such that 
\begin{equation*}
	\begin{aligned}
		v(t,y)\ge w(t,y)&\ge \phi(-y-c_b t+\xi_1)\!-\!M e^{-\eta t}\\
		&\ge \phi(C_\varep-c_b T_{4,\varep} +\xi_1)\!-\!M e^{-\eta T_{4,\varep}}\ge \phi(-C_\varep)\!-\!\frac{\varep}{2}\ge 1\!-\!\varep,~~ t\ge T_{4,\varep},~-C_\varep\le y\le 0,
	\end{aligned}
\end{equation*}
 and
\begin{equation*}
	\begin{aligned}
		v(t,y)\ge w(t,y)&\ge \phi(y-c_b t+\xi_2)\!-\!M e^{-\eta t}\\
		&\ge \phi(X_{4,\varep}-c_b T_{4,\varep} +\xi_2)\!-\!M e^{-\eta T_{4,\varep}}\ge \phi(-C_\varep)-\frac{\varep}{2}\ge 1\!-\!\varep,~~~  t\ge T_{4,\varep},~~ 0 \le y \le X_{4,\varep}.
	\end{aligned}
\end{equation*}
Combining the above two inequalities, we obtain that 
\begin{equation}
	\label{6.5_ini}
	v(t,y)\ge 1-\varep~~~~~\text{for}~~t\ge T_{4,\varep}~~\text{and}~~-C_\varep\le y\le X_{4,\varep}.
\end{equation}

For $t\ge T_{4,\varep}$ and $y\le X_{4,\varep}$, let us define
\begin{equation*}
	\underline v(t,y)=\phi( \zeta(t,y))-\varep e^{-\delta(t-T_{4,\varep})}-\varep e^{-\mu(X_{4,\varep}-y)},
\end{equation*}
with $\zeta(t,y)=-y-c_b(t-T_{4,\varep})-\omega_\varep e^{-\delta(t-T_{4,\varep})}+\omega_\varep$.
We shall prove that $\underline{v}$ is a subsolution for \eqref{1.1} for $t\ge T_{4,\varep}$ and $y\le X_{4,\varep}$. Indeed, at time $t=T_{4,\varep}$, we notice that $\underline{v}(T_{4,\varep},y)\le 1-\varep-\varep e ^{-\mu(X_{4,\varep}-y)}\le v(T_{4,\varep},y)$ for $-C_\varep\le y\le X_{4,\varep}$ by virtue of \eqref{6.5_ini}. For $y\le -C_\varep$, the observation that $\zeta(T_{4,\varep},y)\ge C_\varep$  implies that $\underline{v}(T_{4,\varep},y)\le 0<v(T_{4,\varep},y)$ for $y\le -C_\varep$. Consequently, we derive that $\underline v(T_{4,\varep},y)\le v(T_{4,\varep}, y)$ for $y\le X_{4,\varep}.$ At $x=X_{4,\varep}$, \eqref{6.5_ini} implies that $v(t,X_{4,\varep})\ge 1-\varep\ge \underline v(t,X_{4,\varep})$ for all $t\ge T_{4,\varep}$. Therefore, it remains to verify that $\underline{v}_t(t,y)-\underline{v}_{yy}(t,y)-f(y-ct,\underline{v}(t,y))\le 0$ for $t\ge T_{4,\varep}$ and $y\le X_{4,\varep}$. Since $\underline{v}_t(t,y)-\underline{v}_{yy}(t,y)-f(y-ct,\underline{v}(t,y))\le \underline{v}_t(t,y)-\underline{v}_{yy}(t,y)-f_b(\underline{v}(t,y))=:\mathcal{L}\underline v(t,y)$ for $t\ge 0$ and $y\in\R$, it is therefore sufficient to check that $\mathcal{L}\underline v(t,y)\le 0$ for $t\ge T_{4,\varep}$ and $y\le X_{4,\varep}$. By a direct computation, we have that for $t\ge T_{4,\varep}$ and $y\le X_{4,\varep}$,
\begin{equation*}
	\mathcal{L}\underline v(t,y)=f_b(\phi(\zeta(t,y)))-f_b(\underline v(t,y))+\phi'(\zeta(t,y))\omega_\varep\delta e^{-\delta(t-T_{4,\varep})}+\delta\varep e^{-\delta(t-T_{4,\varep})} +\mu^2 \varep e^{-\mu(X_{4,\varep}-y)}.
\end{equation*}
Again, we distinguish between three cases:
\begin{itemize}
	\item 
	if $\zeta(t,y)\le -C$,  $1>\phi(\zeta(t,y))\ge 1-\varep/2$ and $1>\underline v\ge 1-3\varep$. One infers from \eqref{2.7-mu}--\eqref{2.7-delta} that
	\begin{align*}
		\mathcal{L}\underline v\le \Big(\frac{f_b'(1)}{2}+\delta\Big)\varep e^{-\delta(t-T_{4,\varep})}+\Big(\frac{f_b'(1)}{2}+\mu^2\Big)\varep e^{-\mu(X_{4,\varep}-y)}<0;
	\end{align*}
	
	\item 
	 if $\zeta(t,y)\ge C$, we have  $\varep>\phi(\zeta(t,y))>0$ and $\varep>\underline v\ge -2\varep$. Thanks to \eqref{2.7-mu}--\eqref{2.7-delta}, one has 
	 \begin{align*}
	 	\mathcal{L}\underline v\le \Big(\frac{f_b'(0)}{2}+\delta\Big)\varep e^{-\delta(t-T_{4,\varep})}+\Big(\frac{f_b'(0)}{2}+\mu^2\Big)\varep e^{-\mu(X_{4,\varep}-y)}<0;
	 \end{align*}
	
	\item 
	 if $-C\le \zeta(t,y)\le C$, it follows that $-y+X_{4,\varep}\ge c_b(t-T_{4,\varep})+X_{4,\varep}-\omega_\varep-C\ge c_b(t-T_{4,\varep})+X_{4,\varep}-\omega_\varep-C_\varep$.  One deduces from \eqref{2.7-delta}--\eqref{2.7-omega} and \eqref{6.5-X} that 
		 \begin{align*}
		\mathcal{L}\underline v\le& \max_{[-2\delta,1]}|f_b'|\Big(
		\varep e^{-\delta(t-T_{4,\varep})}+\varep e^{-\mu(X_{4,\varep}-y)}
		\Big)-\kappa\omega_\varep\delta e^{-\delta(t-T_{4,\varep})} + \delta\varep e^{-\delta(t-T_{4,\varep})}+\mu^2\varep e^{-\mu(X_{4,\varep}-y)}\\
		=&		 \Big(\max_{[-2\delta,1]}|f_b'|+\delta-\kappa\omega\Big)\varep e^{-\delta(t-T_{4,\varep})}+\Big(\max_{[-2\delta,1]}|f_b'|+\mu^2\Big)\varep e^{-\mu(X_{4,\varep}-y)}\\
		\le &\Big(\max_{[-2\delta,1]}|f_b'|+\delta-\kappa\omega\Big)\varep e^{-\delta(t-T_{4,\varep})}+\Big(\max_{[-2\delta,1]}|f_b'|+\mu^2\Big)\varep e^{-\mu(c_b(t-T_{4,\varep})+X_{4,\varep}-\omega_\varep-C_\varep)}\\
		\le&\Big(\max_{[-2\delta,1]}|f_b'|+2\delta-\kappa\omega\Big)\varep e^{-\delta(t-T_{4,\varep})}\le 0.
	\end{align*}
\end{itemize}
We conclude that $\mathcal{L}\underline{v}(t,y)\le 0$ for $t\ge T_{4,\varep}$ and $y\le X_{4,\varep}$, which implies that $\underline{v}$ is a subsolution of \eqref{1.1} for $t\ge T_{4,\varep}$ and $y\le X_{4,\varep}$. The comparison principle then gives that  for $t\ge T_{4,\varep}$ and $y\le X_{4,\varep}$,
\begin{equation*}
	v(t,y)\ge \underline v(t,y)=\phi(-y-c_b(t-T_{4,\varep})-\omega_\varep e^{-\delta(t-T_{4,\varep})}+\omega_\varep)-\varep e^{-\delta(t-T_{4,\varep})}-\varep e^{-\mu(X_{4,\varep}-y)}.
\end{equation*}
Therefore, \eqref{6.5} is achieved by taking $z_{4,\varep}=\omega_\varep$ and  by noticing that  $\phi'<0$ in $\R$. This completes the proof.	
\end{proof}

\begin{lemma}
	\label{lemma:left:stability:sup}
	Assume that $c\le  -c_m$ and $c_b>0$. Let $\mu>0$, $\delta>0$, $C>0$, $\kappa>0$ and $\omega>0$ satisfy \eqref{6.2-mu}--\eqref{6.2-omega} and \eqref{2.7-mu}--\eqref{2.7-omega}.   If there exist $\varep\in(0,\delta]$, $ t_1>0$, $x_1>L$ and $\xi_1\in\mathbb{R}$  such that
	\begin{equation}
		\label{6.38-initial}
		\sup_{x\le x_1-ct_1}\Big(u(t_1,x)-\phi (-x-(c_b+c) t_1+\xi_1)\Big)\le \varep,
	\end{equation}
	\begin{equation}\label{6.39-bc-u}
	\sup_{x\le x_1}	u(t,x) \leq \varep~~\text{and} ~~~ u(t,x_1-ct)\le 1+\varep/2~~~~~~\text{for}~  t \ge   t_1,
	\end{equation}
\begin{equation}
	\label{6.39-bc-phi}
\phi(-x_1-c_bt_1+\xi_1)\ge 1-\varep/2,
\end{equation}
	and
	\begin{equation}
		\label{equ:left:sup:f_mu}
	\Big(\max_{[0,1+2\delta]}|f_b'|+ \mu^2 -c \mu \Big)e^{-\mu(-(c_b+c) t_1-x_1-\omega_\varep+\xi_1-C)}+	\Big(\max_{[0,1+2\delta]}|f_b'|+ \mu^2 \Big)e^{-\mu(c_b t_1+x_1-\xi_1-C)}\le\delta
	\end{equation}
	with $\omega_\varep=\varep\omega/\delta$, then  there exists $M_1>0$ such that the following holds true:
	$$	\sup_{x\le x_1-ct}\Big(u(t,x)-\phi (-x-(c_b+c) t+\xi_1)\Big)\le  M_1\varep~~~~\text{for all}~t\ge  t_1.$$
\end{lemma}
	
	\begin{proof}
		The proof is based on a comparison argument. First of all, we observe that $\sup_{x\le x_1}\big(u(t,x)-\phi (-x-(c_b+c) t+\xi_1)\big)\le \varep$ for all $t\ge t_1$, thanks to \eqref{6.39-bc-u}. Let us now consider $t\ge t_1$ and $x_1\le x\le x_1-ct$. Define 
		$$\overline{u}(t,x)=\phi(-x-(c_b+c)t+\omega_\varep e^{-\delta(t- t_1)}-\omega_\varep+\xi_1)+\varep e^{-\delta(t- t_1)}+\varep e^{-\mu(x-x_1)}+\varep e^{-\mu(x_1-ct-x)}$$
		for $t\ge t_1$ and $x_1\le x\le x_1-ct$. Let us prove that $\overline u$ is a supersolution of \eqref{1.1'} for $t\ge t_1$ and $x_1\le x\le x_1-ct$. 
		
		In fact, at time $t= t_1$, one infers from \eqref{6.38-initial} that $\overline{u}(t_1,x)\ge \phi(-x-(c_b+c) t_1+\xi_1)+\varep\ge u( t_1,x)$ for all  $x_1\le x\le x_1-ct_1$. On the other hand, at $x=x_1$, it follows from \eqref{6.39-bc-u}--\eqref{6.39-bc-phi} that $\overline{u}(t,x_1)\ge\varep\ge  u(t,x_1)$ and $\overline u(t,x_1-ct)\ge \phi(-x_1-c_b t_1+\xi_1)+\varep\ge 1+\varep/2\ge u(t,x_1-ct)$ for all $t\ge  t_1$. It then remains to show that $\mathcal{N}\overline u(t,x)\!:=\!\overline u_t(t,x)\!-\! \overline u_{xx}(t,x)\!-c\overline u_x(t,x)\!-\!f_b(\overline u(t,x))\!\ge\!0$ for all $t\ge  t_1$ and $x_1\le x\le x_1-ct$. For convenience, we set
		$$\xi(t,x):=-x-(c_b+c)t+\omega_\varep e^{-\delta(t- t_1)}-\omega_\varep+\xi_1.$$
		A straightforward computation implies that
		\begin{align*}
			\mathcal{N}\overline{u}(t,x)=f_b(\phi(\xi(t,x)))- f_b(\overline{u}(t,x))-\phi'&(\xi(t,x))\omega_\varep\delta e^{-\delta(t- t_1)}-\delta\varep e^{-\delta(t- t_1)}\\
		&-(\mu^2- c\mu)\varep e^{-\mu(x-x_1)}-\mu^2\varep e^{-\mu(x_1-ct-x)}.
		\end{align*}
		We distinguish three cases:  
		\begin{itemize}
			\item if $\xi(t,x)\le -C$, then $1-\delta/2\le \phi(\xi(t,x))<1$ by~\eqref{6.2-C}, hence $1-\delta/2\le \overline{u}(t,x)\le 1+2\delta$; therefore, by using~\eqref{6.2-mu}--\eqref{6.2-delta} and \eqref{2.7-mu}, along with the negativity of $\phi'$ and $f_b'(1)$, it follows that
			\begin{align*}
				\mathcal{N} \overline u(t,x)\ge\! \Big(\!-\!\frac{f_b'(1)}{2}\!-\!\delta\Big)\varep e^{-\delta(t- t_1)}+\!\Big(-\frac{f_b'(1)}{2}-&\mu^2+c\mu\Big)\varep e^{-\mu(x-x_1)}\\
				&+\Big(-\frac{f_b'(1)}{2}-\mu^2\Big)\varep e^{-\mu(x_1-ct-x)}>0;
			\end{align*}
			\item if $\xi(t,x)\ge C$, then $0<\phi(\xi(t,x))\le \delta$ by~\eqref{6.2-C} and thus $0\le\overline{u}(t,x)\le 3 \delta$; therefore, owing to~\eqref{6.2-mu}--\eqref{6.2-delta} and \eqref{2.7-mu} as well as the negativity of $\phi'$  and $f_b'(0)$, it follows that
			\begin{align*}
				\mathcal{N} \overline u(t,x)\ge  \Big(-\frac{f_b'(0)}{2}-\delta\Big)\varep e^{-\delta(t- t_1)}+\Big(-\frac{f_b'(0)}{2}-&\mu^2+c\mu\Big)\varep e^{-\mu(x-x_1)}\\
				&+\Big(-\frac{f_b'(0)}{2}-\mu^2\Big)\varep e^{-\mu(x_1-ct-x)}>0;
			\end{align*}
			\item if $-C\le\xi(t,x)\le C$, one has $x-x_1\ge -(c+c_b)(t- t_1)-(c+c_b) t_1-x_1+\omega_\varep e^{-\delta(t- t_1)}-\omega_\varep+\xi_1-C\ge -(c_b+c)(t- t_1)-(c_b+c) t_1-x_1-\omega_\varep+\xi_1-C$, hence 
			$e^{-\mu(x-x_1)}\le e^{-\mu(-(c_b+c)(t- t_1)-(c_b+c) t_1-x_1-\omega_\varep+\xi_1-C)}$.
			On the other hand, $x_1-ct-x\ge x_1-ct+(c_b+c)t-\omega_\varep e^{-\delta(t- t_1)}+\omega_\varep-\xi_1-C\ge x_1+c_b t-\xi_1-C$, which implies that $e^{-\mu(x_1-ct-x)}\le e^{-\mu(c_b (t-t_1)+c_b t_1+x_1-\xi_1-C)}$ In view of $\omega_\varep=\varep\omega/\delta$, one infers from~\eqref{6.2-delta}, \eqref{6.2-kappa}, \eqref{6.2-omega}, \eqref{2.7-delta} and \eqref{equ:left:sup:f_mu}, that
			\begin{align*}
				\mathcal{N}\overline u(t,x)&
				\ge -\max_{[0,1+2\delta]}|f_b'|\Big(\varep e^{-\delta(t- t_1)}+\varep e^{-\mu(x-x_1)}+\varep e^{-\mu(x_1-ct-x)}\Big)+\kappa\omega\varep e^{-\delta(t- t_1)}\\
				&~~~~~~~~~~~~~~~~~~~~~~~~~~~~~~~~~~~~~~~~~~-\delta \varep e^{-\delta(t- t_1)}-( \mu^2-c\mu) \varep e^{-\mu(x-x_1)}-\mu^2\varep e^{-\mu(x_1-ct-x)}\cr
				&= \Big(\kappa\omega-\max_{[0,1+2\delta]}|f_b'|-\delta\Big)\varep e^{-\delta(t- t_1)}-\Big(\max_{[0,1+2\delta]}|f_b'|+\mu^2-c\mu\Big)\varep e^{-\mu(x-x_1)}\\
				&~~~~~~~~~~~~~~~~~~~~~~~~~~~~~~~~~~~~~~~~~~~~~~~~~~~~~~~~~~~~~~~- 			\Big(\max_{[0,1+2\delta]}|f_b'|+\mu^2\Big)\varep e^{-\mu(x_1-ct-x)}
				  \cr
				&\ge \Big(\kappa\omega-\max_{[0,1+2\delta]}|f_b'|-\delta\Big)\varep e^{-\delta(t- t_1)}- 			\Big(\max_{[0,1+2\delta]}|f_b'|+\mu^2\Big)\varep e^{-\mu(c_b(t-t_1)+c_b t_1+x_1-\xi_1-C)}\\
				&~~~~~~~~~~~~~~~~~~~~~~~~~~\quad-\Big(\max_{[0,1+2\delta]}|f_b'|+\mu^2-c\mu\Big)\varep e^{-\mu\big(
					-(c_b+c)(t- t_1)-(c_b+c) t_1-x_1-\omega_\varep+\xi_1-C\big)}\cr
				&\ge\Big(\kappa\omega-\max_{[0,1+2\delta]}|f_b'|-2\delta\Big)\varep e^{-\delta(t- t_1)}\ge 0.
			\end{align*} 
		\end{itemize}
		We then obtain that $\mathcal{N}\overline u(t,x)\!=\!\overline u_t(t,x)\!-\! \overline u_{xx}(t,x)\!-c\overline u_x(t,x)\!-\!f_b(\overline u(t,x))\!\ge\!0$ for all $t\ge  t_1$ and $x_1\le x\le x_1-ct$. The comparison principle implies that
		$$u(x,t)\le\overline u(t,x)= \phi(-x-(c_b+c)t+\omega_\varep e^{-\delta(t- t_1)}-\omega_\varep+\xi_1)+\varep e^{-\delta(t- t_1)}+\varep e^{-\mu(x-x_1)}+\varep e^{-\mu(x_1-ct-x)}$$
		for $t\ge  t_1$ and $x_1\le x\le x_1-ct$.
		For these $t$ and $x$, since $\phi'<0$, one derives that
		$$u(x,t)\le \phi(-x-(c_b+c)t-\omega_\varep+\xi_1)+3\varep \le
		\phi(-x-(c_b+c)t+\xi_1)+ \omega_\varep\Vert\phi'\Vert_{L^\infty(\mathbb{R})} +3\varep. $$
		In conclusion, one has
		\begin{equation*}
			 	\sup_{x_1\le x\le x_1-ct}\Big(u(t,x)-\phi(-x-(c_b+c) t_1+\xi_1)\Big)\le {M}_1\varep~~~\text{for all}~t\ge t_1,
		\end{equation*}
		where ${M}_1:=\omega_\varep\Vert\phi' \Vert_{L^\infty(\mathbb{R})}/\varep+3=\omega\Vert\phi' \Vert_{L^\infty(\mathbb{R})}/\delta+3$ is independent of $\varep$, $ t_1$, $x_1$ and $\xi_1$. Lemma \ref{lemma:left:stability:sup} is therefore achieved, with $M_1$ chosen above.
	\end{proof}
\begin{lemma}
	\label{lemma:left:stability:sub}
	Assume that $c\le  -c_m$ and $c_b>0$. Let $\mu>0$, $\delta>0$, $C>0$, $\kappa>0$ and $\omega>0$ satisfy \eqref{6.2-mu}--\eqref{6.2-omega} and \eqref{2.7-mu}--\eqref{2.7-omega}.   If there exist $\varep\in(0,\delta]$, $ t_2>0$, $x_2\in\R$ and $\xi_2\in\mathbb{R}$ such that
	\begin{equation}
		\label{equ:left:init}
		-\varep\le \sup_{y\le x_2}\Big(v(t_2,y)-\phi(-y-c_bt_2+\xi_2)\Big),
	\end{equation}
	\begin{equation}\label{equ:left:init2}
		1-\varep\le v(t,x_2)~~~~\text{for all}  ~~t\ge  t_2,
		\ee
		and
		\begin{equation}
			\label{equ:left:f_mu}
			\Big(\max_{[-2\delta,1]}|f_b'|+ \mu^2\Big)e^{-\mu(c_b t_2+x_2-\omega_\varep-\xi_2-C)}\le\delta
		\end{equation}
		with $\omega_\varep=\varep\omega/\delta$, then  there exists $M_2>0$ such that the following holds true:
		$$-M_2\varep\le \sup_{y\le x_2}\Big(v(t,y)-\phi(-y-c_b t+\xi_2)\Big)~~~~\text{for all}~t\ge  t_2.$$
	\end{lemma}
	
	\begin{proof}
		We first claim that 
		$$\underline{v}(t,y)=\phi(-y-c_bt-\omega_\varep e^{-\delta(t- t_2)}+\omega_\varep+\xi_2)-\varep e^{-\delta(t- t_2)}-\varep e^{-\mu(x_2-y)}$$
		is a subsolution of \eqref{1.1} for $t\ge  t_2$ and $y\le x_2$, for which it is sufficient to show that $\overline v$ is a subsolution to $v_t=  v_{yy}+f_b(v)$ for $t\ge  t_2$ and $y\le x_2$. 
		
		At time $t= t_2$, one has $\underline{v}( t_2,y)\le \phi(-y-c_b t_2+\xi_2)-\varep\le v( t_2,y)$ for all $y\le x_2$, thanks to~\eqref{equ:left:init}. Moreover, $\underline{v}(t,x_2)\le 1-\varep\le v(t,x_2)$ for all $t\ge  t_2$, owing to~\eqref{equ:left:init2}. It remains to show that $\mathcal{N}\underline{v}(t,y):=\underline{v}_t(t,y)- \underline{v}_{yy}(t,y)-f_b(\underline{v}(t,y))\le 0$ for all $t\ge  t_2$ and $y\le x_2$. For convenience, we set
		$$\eta(t,y):=-y-c_bt-\omega_\varep e^{-\delta(t- t_2)}+\omega_\varep+\xi_2.$$
		By a straightforward computation, one has
		\begin{equation*}
			\mathcal{N}\underline{v}(t,y)=f_b(\phi(\eta(t,y)))- f_b(\underline{v}(t,y))+\phi'(\eta(t,y))\omega_\varep\delta e^{-\delta(t- t_2)}+\varep\delta e^{-\delta(t- t_2)}+   \mu^2 \varep e^{-\mu(x_2-y)}.
		\end{equation*}
		We divide our analysis into three cases:
		\begin{itemize}
			\item if $\eta(t,y)\le -C$, then $1-\delta/2\le \phi(\eta(t,y))<1$ by~\eqref{6.2-C}, hence $1>\underline{v}(t,y)\ge 1-\delta/2-2\varep\ge 1-3\delta$; therefore it follows from~\eqref{2.7-mu}--\eqref{2.7-delta}, together with the negativity of $\phi'$ and $f_b'(1)$,  that
			\begin{align*}
				\mathcal{N}\underline{v}(t,y)&\le  \frac{f_b'(1)}{2}\Big(\varep e^{-\delta(t- t_2)}+\varep e^{-\mu(x_2-y)}\Big)+\varep\delta e^{-\delta(t- t_2)}+   \mu^2\varep e^{-\mu(x_2-y)}\cr
				&=\Big(\frac{f_b'(1)}{2}+\delta\Big)\varep e^{-\delta(t- t_2)}+\Big(\frac{f_b'(1)}{2}+ \mu^2\Big)\varep e^{-\mu(x_2-y)}\le 0;
			\end{align*}
			\item if $\eta(t,y)\ge C$, then $0<\phi(\eta(t,y))\le \delta$ by~\eqref{6.2-C} and thus $-2\delta\le-2\varep\le\underline{v}(t,y)\le \delta$; therefore, owing to~\eqref{2.7-mu}--\eqref{2.7-delta} as well as the negativity of $\phi'$  and $f_b'(0)$, it follows that
			\begin{align*}
				\mathcal{N}\underline{v}(t,y)&\le  \frac{f_b'(0)}{2}\Big(\varep e^{-\delta(t- t_2)}+\varep e^{-\mu(x_2-y)}\Big)+\varep\delta e^{-\delta(t- t_2)}+  \mu^2\varep e^{-\mu(x_2-y)}\cr
				&=\Big(\frac{f_b'(0)}{2}+\delta\Big)\varep e^{-\delta(t- t_2)}+\Big(\frac{f_b'(0)}{2}+ \mu^2\Big)\varep e^{-\mu(x_2-y)}\le 0;
			\end{align*}
			\item if $-C\le\eta(t,y)\le C$, one has $x_2-y\ge c_b(t- t_2)+c_b t_2+x_2+\omega_\varep e^{-\delta(t- t_2)}-\omega_\varep-\xi_2-C\ge c_b(t- t_2)+c_b t_2+x_2-\omega_\varep-\xi_2-C$, hence $e^{-\mu(x_2-y)}\le e^{-\mu(c_b(t- t_2)+c_b t_2+x_2-\omega_\varep-\xi_2-C)}$; since $\omega_\varep=\varep\omega/\delta$, one infers from~\eqref{6.2-kappa}--\eqref{6.2-omega}, \eqref{2.7-delta}--\eqref{2.7-omega} and \eqref{equ:left:f_mu} that
			\begin{align*}
				\mathcal{N}\underline{v}(t,y)&\le \max_{[-2\delta,1]}|f_b'|\Big(\varep e^{-\delta(t- t_2)}+\varep e^{-\mu(x_2-y)}\Big)\!-\!\kappa\omega_\varep\delta e^{-\delta(t- t_2)}\!+\!\varep\delta e^{-\delta(t- t_2)}\!+\!   \mu^2\varep e^{-\mu(x_2-y)}\cr
				&\le \Big(\!\max_{[-2\delta,1]}\!|f_b'|\!-\!\kappa\omega\!+\!\delta\!\Big)\varep e^{-\delta(t- t_2)}\!+\!\Big(\!\max_{[-2\delta,1]}\!|f_b'|\!+\! \mu^2\Big)\varep e^{-\mu(c_b(t- t_2)+c_b t_2+x_2-\omega_\varep-\xi_2-C)}\cr
				&\le \Big(\max_{[-2\delta,1]}|f_b'|-\kappa\omega+2\delta\Big)\varep e^{-\delta(t- t_2)}\le  0.
			\end{align*}
		\end{itemize}

		Eventually, one concludes that $\mathcal{N}\underline{v}(t,y)=\underline{v}_t(t,y)- \underline{v}_{yy}(t,y)-f_b(\underline{v}(t,y))\le 0$ for all $t\ge  t_2$ and~$y\le x_2$. The maximum principle implies that 
		\begin{align*}
			v(t,y)\ge \underline v(t,y)= \phi\big(-y-c_bt-\omega_\varep e^{-\delta(t- t_2)}+\omega_\varep+\xi_2\big)-\varep e^{-\delta(t- t_2)}-\varep e^{-\mu(x_2-y)}
		\end{align*}
		for all $t\ge  t_2$ and $y\le x_2$. For these $t$ and $y$, since $\phi'<0$, one derives that
		\begin{align*}
			v(t,y)\ge \phi(-y-c_bt+\omega_\varep+\xi_2)-2\varep\ge \phi(-y-c_bt+\xi_2)-\omega_\varep\Vert\phi'\Vert_{L^\infty(\mathbb{R})}-2\varep.
		\end{align*}
	In conclusion, one has
	\begin{equation*}
	-M_2\varep\le 	\sup_{y\le x_2}\Big(v(t,y)-\phi(-y-c_b t_2+\xi_2)\Big)~~~\text{for all}~t\ge t_2,
	\end{equation*}
			where $M_2:=\omega_\varep\Vert\phi' \Vert_{L^\infty(\mathbb{R})}/\varep+2=\omega\Vert\phi' \Vert_{L^\infty(\mathbb{R})}/\delta+2$ is independent of $\varep$, $ t_2$, $x_2$ and $\xi_2$.		The proof of Lemma~\ref{lemma:left:stability:sub} is thereby complete. 
	\end{proof}

Lemma \ref{lemma:left:stability:sub} can be equivalently written as:
\begin{lemma}
	\label{lemma:left:stability:sub:equiv}
	Assume that $c\le  -c_m$ and $c_b>0$. Let $\mu>0$, $\delta>0$, $C>0$, $\kappa>0$ and $\omega>0$ satisfy \eqref{6.2-mu}--\eqref{6.2-omega} and \eqref{2.7-mu}--\eqref{2.7-omega}.   If there exist $\varep\in(0,\delta]$, $ t_2>0$, $x_2\in\R$ and $\xi_2\in\mathbb{R}$ such that
	\begin{equation}
		\label{equ:left:init:equiv}
		-\varep\le \sup_{x\le x_2-ct_2}\Big(u(t_2,x)-\phi(-x-(c_b+c)t_2+\xi_2)\Big),
	\end{equation}
	\begin{equation}\label{equ:left:init2:equiv}
		1-\varep\le u(t,x_2-ct)~~~~\text{for all}~~ t\ge t_2,
		\ee
		and
		\begin{equation*}
			\Big(\max_{[-2\delta,1]}|f_b'|+ \mu^2\Big)e^{-\mu(c_bt_2+x_2-\omega_\varep-\xi_2-C)}\le\delta
		\end{equation*}
		with $\omega_\varep=\varep\omega/\delta$, then  there exists $M_2>0$ such that the following holds true:
		$$	-M_2 \varep\le \sup_{x\le x_2-ct}\Big(u(t,x)-\phi(-x-(c_b+c)t+\xi_2)\Big)~~~~\text{for all}~t\ge t_2.$$
	\end{lemma}

We are now in a position to finish the proof of Theorem \ref{thm_c le -cm,cb>0}, by Lemma \ref{lemma:left:stability:sup} and Lemma \ref{lemma:left:stability:sub:equiv}.

\vspace{0.2cm}

\noindent
\begin{proof}[End of proof of Theorem \ref{thm_c le -cm,cb>0}.]
	Let  $X_3>L$,  $T_3>0$, $z_3\in\mathbb{R}$ be given as in Lemma \ref{lemma:c<-c_m}, and let $X_4>L$,  $T_4>0$,  $z_4\in\R$, $\mu>0$ and $\delta>0$ be given as in the proof of Lemma \ref{lemma6.3_upper bound right}. It follows from \eqref{6.2-1} and \eqref{u-left-lower bd} that, for $t\ge T:=\max(T_3,T_4)$ such that $X_3<X_4-cT$ and for $X_3\le x\le X_4-ct$,
	\begin{equation}
		\begin{aligned}\label{6.41}
			\phi(-x-(c_b+c)t&+z_4)-\delta e^{-\delta(t-T_4)}-\delta e^{-\mu(X_4-x-ct)}\\
			&~~~\le 	u(t,x)\le\phi\big(-x-(c_b+c)(t-T_3)+z_3\big)+\delta e^{-\delta(t-T_3)}+\delta e^{-\mu(x-X_3)}.
		\end{aligned}
	\end{equation}
	Consider any sequence $(t_n)_{n\in\mathbb{N}}$ in $\R$ such that $t_n\to+\infty$ as $n\to+\infty$. From parabolic estimates, the functions
	\begin{equation*}
		(t,z)\in\R^2\mapsto u_n(t,z):=u(t+t_n, z-(c_b+c)t_n)
	\end{equation*}
	converge as $n\to+\infty$, locally uniformly in $\R^2$, to a classical solution $u_\infty$ of $(u_\infty)_t=(u_\infty)_{zz}+c(u_\infty)_z+f_b(u_\infty)$ in $\R^2$. By applying \eqref{6.41} at $(t+t_n, z-(c_b+c)t_n)$ and then passing to the limit as $n\to+\infty$, it follows that
	\begin{equation}
		\label{6.42}
		\phi(-z-(c_b+c)t+z_4)\le u_\infty(t,z)\le 	\phi(-z-(c_b+c)(t-T_3)+z_3)~~~~\text{in}~\R^2.
	\end{equation}
	Then, \cite[Theorem 3.1]{BH2007} implies that there exists $\eta\in\R$ such that
	\begin{equation}
		\label{6.43}
		u_\infty(t,z)=\phi(-z-(c_b+c)t+\eta)~~~~\text{for}~(t,z)\in\R^2.
	\end{equation}
	whence 
	\begin{equation}
		\label{equ:converge:local}
		u_n(t,z)\to \phi(-z-(c_b+c) t+\eta)~\text{as}~n\to+\infty~~\text{locally uniformly in}~(t,z)\in\mathbb{R}\times\mathbb{R}.
	\end{equation}
	
	Consider now any $\varep \in (0,\delta/3]$. Let $B_{\varep}>0$ be such that
	\begin{equation}\label{equ:phi:A_vareps}
		\phi \ge 1 -\frac{\varep}{2} \text{ in } (-\infty,-B_{\varep}]  \text{ and } \phi \le \frac{\varep}{2} \text{ in } [B_{\varep},+\infty).
	\end{equation}
	Set
	$E_1:= \min\big( -B_{\varep}+\eta,-B_{\varep}+(c_b+c)T_3 +z_3\big)$ and $E_2:=\max\big(B_{\varep}+\eta, B_{\varep}+z_4\big)>E_1$.
	It then follows from \eqref{equ:converge:local} that
	\begin{equation}\label{equ:u_n:phi:int}
		\sup_{E_1 \leq z \leq E_2} \vert u_n(0,z)- \phi(-z+\eta) \vert \leq \varep
	\end{equation}
	for all $n$ large enough. In view of $t_n \rightarrow +\infty$ as $n \rightarrow +\infty$, \eqref{6.41}, \eqref{equ:phi:A_vareps} and the definition of $E_1$ and $E_2$, we then have 
	\begin{equation}
		\label{6.51}
		\begin{aligned}
			\begin{cases}\displaystyle
				0 < u_n(0,z) \le \varep &\text{ for all } z \le E_1,\\
				 1-\varep \le u_n(0,z) \le 1+\varep  &\text{ for all } \displaystyle E_2 \le z \le E_2 + \frac{c_b}{2} t_n, 
			\end{cases}
		\end{aligned}
	\end{equation}
	for all $n$ large enough. Furthermore, since $E_2 \ge B_{\varep} +\eta$ and  $E_1 \le -B_{\varep} +\eta$,
	\begin{equation}
		\label{6.52}
		\begin{aligned}
			\begin{cases} \displaystyle
				0< \phi(-z+\eta) \le \frac{\varep}{2} <\varep&\text{ for all } z \le E_1,\\
			\displaystyle	1-\varep < 1 - \frac{\varep}{2} \le \phi(-z+\eta )<1 &\text{ for all } z \ge E_2.
			\end{cases}
		\end{aligned}
	\end{equation}
	It then follows from \eqref{equ:u_n:phi:int}-\eqref{6.52} that, for all $n$ large enough
	$$
	\vert u_n(0,z) -\phi(-z+\eta) \vert \le 2 \varep~~~~
	\text{ for all }~~ z\le E_2+\frac{c_b}{2}t_n.
	$$
	Due to the definition of $u_n(t,z)$, one has
	\begin{equation}
		\label{6.53}
			\vert u(t_n,x) -\phi(-x-(c_b+c)t_n+\eta) \vert \le 2\varep ~~~~\text{ for all }~ x\le E_2-\frac{c_b}{2}t_n-ct_n.
	\end{equation}
	On the other hand, one infers from \eqref{6.4-1} and \eqref{6.5-u}  that for all $n$ large enough,
	$$	
	\begin{aligned}
		1-3\varep&\le\phi(-x-(c_b+c)(t_n-T_{4,\varep})-cT_{4,\varep}+z_{4,\varep})-\varep e^{-\delta(t_n-T_{4,\varep})}-\varep e^{-\mu(X_{4,\varep}-x-ct_n)}\\
		& \le u(t_n,x)\le\phi\big(-x-(c_b+c)(t_n-T_{3,\varep})+z_{3,\varep}\big)+\varep e^{-\delta(t_n-T_{3,\varep})}+\varep e^{-\mu(x-X_{3,\varep})}\le 1+2\varep
	\end{aligned}
	$$
	for $E_2-\frac{c_b}{2}t_n-ct_n \le x \le X_{4,\varep} -ct_n$, where $X_{3,\varep}>L$, $X_{4,\varep}>L$, $T_{3,\varep}>0$, $T_{4,\varep}>0$, $z_{3,\varep}\in\R$, $z_{4,\varep}\in \mathbb{R}$ were given in Lemmas \ref{lemma-6.4}--\ref{lemma-6.5}.
	Notice also that for all $n$ large enough
	$$
	1-\varep \le \phi(-x - (c_b+c) t_n +\eta) <1 ~~~~\text{ for all } ~~ E_2-\frac{c_b}{2}t_n-ct_n\le x \le X_{4,\varep}-ct_n.
	$$
	 Combining the above two inequalities yields that for all $n$ large enough,
	$$
	\big| u(t_n,x) -\phi(-x-(c_b+c)t_n +\eta) \big| \le 3 \varep ~~~~\text{ for all }~ x \in \left[E_2-\frac{c_b}{2}t_n-ct_n,X_{4,\varep} -ct_n\right].
	$$
	Together with \eqref{6.53}, one derives that for all $n$ large enough,
	$$
\sup_{x \le X_{4,\varep}-ct_n}	\big| u(t_n,x)-\phi(-x-(c_b+c)t_n+\eta) \big| \le 3 \varep.
	$$

	We are now in a position to prove  \eqref{equ:c_le_cm:l} by applying Lemmas \ref{lemma:left:stability:sup} and  \ref{lemma:left:stability:sub:equiv}.
	Let $x_1=x_2=X_{4,\varep}$ and $\xi_1=\xi_2=\eta$. We infer from Lemma \ref{lem:u:basic_propert} and \eqref{2.7-w} that for all $n$ large enough,
	$$1-\varep\le u(t, X_{4,\varep}-ct)\le 1+\frac{\varep}{2}~~~~\text{for}~~t\ge t_n.$$
	Moreover, one has for all $n$ large enough that $\sup_{x\le X_{4,\varep}} u(t,x)\le \varep$ for all $t\ge t_n$ by Lemma \ref{lemma_c<-c_m}, and $\phi(-X_{4,\varep}-c_b t_n+\eta)\ge 1-\frac{\varep}{2}$, thanks to $c_b>0$, and that
	\begin{equation*}
		\label{equ:left:sup:f_mu}
		\Big(\max_{[-2\delta,1+2\delta]}|f_b'|+ \mu^2 -c \mu \Big)e^{-\mu(-(c_b+c) t_n-X_{4,\varep}-\omega_\varep+\eta-C)}+	\Big(\max_{[-2\delta,1+2\delta]}|f_b'|+ \mu^2 \Big)e^{-\mu(c_b t_n+X_{4,\varep}-\omega_\varep-\eta-C)}\le\delta
\end{equation*}

It then follows from Lemmas \ref{lemma:left:stability:sup} and  \ref{lemma:left:stability:sub:equiv} (applied with $t_1=t_2=t_n$, $x_1=x_2=X_{4,\varep}$ and $\xi_1=\xi_2=\eta$  and $3\varep$ instead of $\varep$) that for $n$ large enough,
\begin{equation*}
	 \big|u(t,x)-\phi(-x-(c_b+c)t+\xi_2)\big|\le 3M\varep~~~~\text{for all}~t\ge t_n ~\text{and}~x\le X_{4,\varep}-ct,
\end{equation*}
for $M:=\max(M_1,M_2)$, where $M_1$ and $M_2$ were  given in Lemmas \ref{lemma:left:stability:sup} and  \ref{lemma:left:stability:sub:equiv} respectively. Since $\varep\in(0,\delta/3]$ was arbitrarily chosen, one eventually infers that
\begin{equation*}
	\sup_{x\le X_{4,\varep}-ct}     \big|u(t,x)-\phi(-x-(c_b+c)t+\xi_2)\big|\to 0~~~~\text{as}~t\to+\infty.
\end{equation*}
Thus, \eqref{equ:c_le_cm:l}  is achieved,  with $z_2=\xi_2$. This completes the proof of Theorem \ref{thm_c le -cm,cb>0}.
\end{proof}

	\subsection[Proof of Theorem \ref{thm_extinction}]{Extinction: Proof of Theorem \ref{thm_extinction}}

	\begin{proof}[Proof of Theorem \ref{thm_extinction}]
		Remember that $c\le-c_m$.
		\vspace{0.2cm}
		
		\noindent
		\textit{Step 1: proof under condition (i).}	We first consider the case where $c_b<0$. Our proof includes two parts:  $c_b< -c_m$ and $c\in(c_b,-c_m]$;  $c_b<0$ and $c\le \min(-c_m,c_b)$.
		\vspace{0.2cm}
		
			\noindent
		\textit{Step 1.1}.	 Assume that $c_b< -c_m$ and $c\in(c_b,-c_m]$.	Since $c_b<c$, we can construct $\overline u$ as in \eqref{super_thm2.5} in the proof of Theorem \ref{thm_blocking_c>c_b} (which relies only on $c_b<c$). We readily check that $\overline u$ is a supersolution such that the solution $u$ is blocked in the right direction, i.e., $u(t,x)\to 0$ as $x\to+\infty$ uniformly in $t\ge 0$. 	Then, for any $\varep\in(0,\theta)$, there exists $x_0>L$ sufficiently large such that
		\begin{equation*}
			0<\sup_{x\ge x_0}u(t,x)<\varep \quad\text{for}~ t \geq 0.
		\end{equation*}
		On the other hand, since $c\le -c_m$, we infer from Lemma \ref{lemma_c<-c_m} that there exists $t_0>0$ large enough such that 
		$$0<\sup_{x \le x_0}u(t,x)< \varep \quad\text{for}~ t\ge t_0.$$ 
		Therefore, one concludes that $u(t,x)\to 0$ as $t\to+\infty$, uniformly for $x\in\R$.
		\vspace{0.2cm}
	
		\noindent
		\textit{Step 1.2}. Assume that $c_b<0$ and $c\le \min(-c_m,c_b)$. We divide into  two cases.

		\noindent
		\textit{Case 1. Assume that $c\le -c_m$ and $c<c_b$.} Choose $\varep_0\in (0,-c_b)$.
	Lemma \ref{lemma:c<-c_m} immediately implies that there exists $X>L$ such that
		\begin{equation}\label{region-1}
			\limsup_{t\to+\infty}\sup_{X\le x<-(c_b+c+\varep_0)t} u(t,x)=0.
		\end{equation}
		On the other hand, it follows from \eqref{4.57} in Lemma \ref{lemma 3.1} (i) that
		\begin{equation}\label{region-2}
			\limsup_{t\to+\infty}\sup_{ x>(c_b-c+\varep_0)t} u(t,x)=0.
		\end{equation}
		Combining \eqref{region-1}--\eqref{region-2}, along with the fact that $c_b-c+\varep_0<-c_b-c-\varep_0$ as well as Lemma \ref{lemma_c<-c_m}, one reaches $\lim_{t\to+\infty}u(t,x)=0$ uniformly in $x\in\R$.
	
		\vspace{0.2cm}
		
			\noindent
		\textit{Case 2. Assume that $c=c_b\le -c_m$.} For any $\varep\in(0,\theta/2)$, let $f_{b,\varep}$ be a $C^1(\mathbb{R})$ function such that $f_{b,\varep}\ge f_b~\text{in}~\R$, and
		$$\left\{\baa{l}
		f_{b,\varep}(0)=f_{b,\varep}(\theta-\varep)=f_{b,\varep}(1)=0,~~f_{b,\varep}'(0)<0,~~f_{b,\varep}'(\theta)>0,~~f_{b,\varep}'(1)<0,\vspace{3pt}\\
		f_{b,\varep}<0~\text{in}~(0,\theta-\varep),~~f_{b,\varep}>0~\text{in}~(\theta-\varep,1),~~\int_0^1f_{b,\varep}(s)ds<0.\eaa\right.$$
		 For each $\varep\in(0,\theta/2)$ small enough, let $\phi_\varep$ be the unique traveling front profile of $u_t=  u_{xx}+f_{b,\varep}(u)$ such that 
		$$ \phi_\varep''+c_{b,\varep}\phi_\varep'+f_{b,\varep}(\phi_\varep)=0\hbox{ in }\R,~~\phi_\varep(0)=\theta,~~\phi_\varep(-\infty)=1,~~\phi_\varep(+\infty)=0,$$
		with speed $c_{b,\varep}>c_b$. Then, $\phi_\varep\to\phi$ in $C^2_{loc}(\R)$ and $c_{b,\varep}\to c_b$ as $\varep\to0$. We then fix $\varep\in(0,\theta/2)$ small enough such that $c_b<c_{b,\varep}<0$. 
		
		Let now $u_\varep$ be the solution to the Cauchy problem \eqref{1.2}  with $f_\varep(x,s)$ instead of $f(x,s)$ for $(x,s)\in\R\times\R_+$ starting from the initial condition $u_0$, where we assume that $f_\varep(x,s)=f(x,s)$ for $(x,s)\in(-\infty,L]\times \R_+$, while $f(x,s)=f_{b,\varep}(s)$ for $x\in[L,+\infty)$. By the comparison principle, we have that $u_\varep(t,x)>u(t,x)$ for $t>0$ and $x\in\R$. We now claim that $u_\varep$ extincts, therefore so does $u$.
		
		 In fact, we notice that $c\le -c_m$ and $c<c_{b,\varep}$.  On the one hand, since $c\le -c_m$, we apply Lemma \ref{lemma:c<-c_m} and derive that there exist  $T>0$, $X>L$, $z\in\mathbb{R}$ and $\delta>0$ such that
		\begin{align*}
			u_\varep(t,x)\le\phi_\varep\big(-x+X-(c_{b,\varep}+c)(t-T)+z\big)+\delta e^{-\delta(t-T)}~~~\text{for}~t\ge T~\text{and}~ x\ge X.
		\end{align*}
	This implies that $	\lim_{t\to+\infty}\sup_{X\le x<-(c_{b,\varep}+c)t} u_\varep(t,x)=0$. Together with Lemma \ref{lemma_c<-c_m}, we then obtain 
	\begin{equation}
		\label{2.8-1}
		\lim_{t\to+\infty}\sup_{ x<-(c_{b,\varep}+c)t} u_\varep(t,x)=0
	\end{equation} 
	On the other hand, due to $c<c_{b,\varep}$, we adapt Lemma \ref{lemma 3.1} (i) and deduce that  there exist $X_1>L$, $T_1>0$, $z_1\in\mathbb{R}$, $\mu_1>0$ and $\delta_1>0$ such that
	\begin{equation*}
		u_\varep(t,x)\le\phi_\varep(x-(c_{b,\varep}-c)(t-T_1)+z_1)+\delta_1 e^{-\delta_1(t-T_1)}+\delta_1 e^{-\mu_1(x-X_1)}~~\text{for all}~ t\ge T_1~\text{and} ~x\ge X_1,
	\end{equation*}
		which yields that
		\begin{equation}
			\label{2.8-2}
			\lim_{t\to+\infty}\sup_{ x>(c_{b,\varep}-c)t} u_\varep(t,x)=0.
		\end{equation} 
		Combining \eqref{2.8-1}--\eqref{2.8-2}, along with the fact that $c_{b,\varep}<0$, we immediately get $\sup_{x\in\R}u_\varep(t,x)\to 0$ as $t\to+\infty$. That is, $u_\varep$ goes extinction. 		
	
		\vspace{0.2cm} 
		 
		 \noindent
			\textit{Step 2: proof under condition (ii).}		Assume now that \textnormal{spt}$(u_0)$ is and $\Vert u_0\Vert_{L^\infty(\mathbb{R})} < \theta$.	For any $\varep\in(0,\theta)$,  we choose $f_{b,\varep} \in C^{1}(\mathbb{R})$ such that $f_{b,\varep} \geq f_b$ and 
		 \begin{equation}
		 	\label{tilde f_b}
		 	\begin{aligned}
		 		\left\{\baa{l}
		 		f_{b,\varep}(\varep)=f_{b,\varep}(\theta)=f_{b,\varep}(1)=0,~~f_{b,\varep}'(\varep)<0,~~f_{b,\varep}'(\theta)>0,~~f_{b,\varep}'(1)<0,\vspace{3pt}\\
		 		f_{b,\varep}>0~\text{in}~(-\infty,\varep)\cup(\theta,1),~~f_{b,\varep}<0~\text{in}~(\varep,\theta)\cup(1,+\infty).\eaa\right.
		 	\end{aligned}
		 \end{equation}
		 We now prove that there exist monostable decreasing traveling waves of
		 \begin{equation}
		 	\label{6.57}
		 	w_t=w_{xx}+f_{b,\varep}(w), ~~~~(t,x)\in\R^2,
		 \end{equation}
		 connecting $\theta$ and $\varep$, and moving to the left, one of which will serve as an upper barrier for the function $u$. To do so, set $z(t,x)=\theta-\omega(t,x)$ for $(t,x)\in\R^2$, then $z$ satisfies
		 \begin{equation}
		 	\label{6.58}
		 	z_t=z_{xx}+g(z),~~~~(t,x)\in\R^2,
		 \end{equation} 
		 with $g(z)=-f_{b,\varep}(\theta-z)$ being positive in $(0,\theta-\varep)$. Then it is known that  \eqref{6.58} admits traveling front solutions $\psi(x-\nu t)$ with $\psi: \mathbb{R} \rightarrow (0,\theta-\varep)$ such that $\psi''+\nu\psi'+g(\psi)=0$ in $\R$,  $0<\psi<\theta-\varep$, $\psi(-\infty)=\theta$ and $\psi(+\infty)=\varep$ if and only if $\nu\ge\nu_{\text{min}}$ for some $\nu_{\text{min}}\ge 2\sqrt{f'_{b,\varep}(\theta)}$. 
		 Fix $\nu_0 > \max(-c, \nu_{\text{min}})$. It is straightforward to see that  $\theta-\psi(-x-\nu_0 t)$ is a decreasing traveling wave for \eqref{6.57} connecting $\theta$ and $\varep$, and moving to the left with speed $\nu_0>0$.
		 
		 On the other hand, one can easily verify that $\varphi_{c_m+c}(-x-(c_m+c)t)$ is a supersolution to \eqref{1.1'} for $(t,x)\in\R^2$ which has an increasing profile and moves to the left with speed $c_m+c\le 0$.
		 
		  Since $\text{spt}(u_0)$ is included in the bistable region and $\Vert u_0\Vert_{L^\infty(\mathbb{R})}<\theta$, one can fix  $B\in\R$ such that  $\theta-\psi(-x-B)>u_0(x)$ for $x\ge L$. Then, choose  $A\in\R$ such that $\varphi_{c_m+c}(-L+A)=\theta-\psi(-L-B)$.	   Then, one can check that $$\overline u(t,x):=\min\big(\varphi_{c_m+c}(-x-(c_m+c)t+A),\theta-\psi(-x-(\nu_0+c) t-B)\big),~~~t\ge 0,~x\in\R,$$ is a supersolution to \eqref{1.1'} such that $\overline u(0,x)>u_0(x)$ for $x\in\R$. The comparison principle gives that $\overline u(t,x)>u(t,x)$ for $t>0$ and $x\in\R$. Noticing that $\theta-\psi(-x-(\nu_0+c) t-B)$
		   indeed moves to the left with speed $\nu_0+c>0$, it  implies that $\sup_{x\ge X} u(t,x)\le \varep$ as $t\to+\infty$ for some $X>L$. Since $\varep>0$ was arbitrarily chosen, together with  $\sup_{x\le X} u(t,x)\le \varep$ as $t\to+\infty$ by Lemma \ref{lemma_c<-c_m}, it follows that $u(t,x)\to 0$ as $t\to+\infty$, uniformly for $x\in\R$. 	 This completes the proof of Theorem \ref{thm_extinction}.
	\end{proof}

	\section[Proof of Theorem \ref{thm_log delay}]{Sharp estimate of the level sets in the left direction when $c\in(-c_m,+\infty)$: Proof of Theorem \ref{thm_log delay}}
	
The proof of Theorem \ref{thm_log delay} is highly motivated by \cite{HNRR2013}. In fact, the  estimate of the level sets from exterior (upper bound) comes from the supersolution and  is a straightforward application of \cite{HNRR2013}. Regarding the estimate from interior (lower bound), the idea in \cite{HNRR2013} can be adapted with a slight modification, so that it works for every $c>-c_m$.    The details will be sketched below for the sake of completeness.
	
	\begin{proof}[Proof of Theorem \ref{thm_log delay}]
	Remember that  the level set is defined as $E_\varrho^-(t)=\inf\{x\in\R|u(t,x)=\varrho\}$ for any $t>0$ and for any given $\varrho\in(0,1)$.
	
	First of all,  consider \eqref{1.1} with $f_m(v)$ instead of $f(y-ct,v)$, for which let us define $S_\varrho(t)=\inf\{y\in\R|v(t,y)=\varrho\}$ for $t>0$ and for any $\varrho\in(0,1)$. Then it follows from \cite{HNRR2013} that $S_\varrho(t)=-c_m t+3/(2\lambda^*)\ln t+O_{t\to+\infty}(1)$, with $2\lambda^*=c_m$. Correspondingly, the level sets of equation \eqref{1.1'} with $f_m(u)$ instead of $f(x,u)$ in the left direction behave asymptotically like  $-(c_m+c) t+3/(2\lambda^*)\ln t+O_{t\to+\infty}(1)$. This implies that  $E_\varrho^-(t)\ge -(c_m+c) t+3/(2\lambda^*)\ln t+C_1 $ for some $C_1\in\R$. In what follows, it remains to prove that $E_\varrho^-(t)\le -(c_m+c) t+3/(2\lambda^*)\ln t+C_2 $ for some $C_2\in\R$.

		\noindent
		{\it Step 1.} The linearized problem with Dirichlet boundary condition at $-(c+c_m)t$. It is easy to verify that the function $w(t,\hat{x})=u(t,-(c+c_m)t-\hat{x}-L)$ solves 
		$$
		w_t-w_{\hat{x}\hat{x}}-c_m w_{\hat{x}}-f(-(c+c_m)t-\hat{x}-L,w)=0,~~~~t>0,~\hat{x}\in\R.
		$$
		In particular, we have $
		w_t-w_{\hat{x}\hat{x}}-c_m w_{\hat{x}}-f_m(w)=0$ for $t>0$ and $\hat{x}>0$.
		Consider the following linear equation with Dirichlet boundary condition at $\hat{x}=0$:
		\begin{equation}
			\begin{aligned}
				\begin{cases}
						z_t-z_{\hat{x}\hat{x}}-c_mz_{\hat{x}}-f_m'(0)z=0,~~~~&t>0,~\hat{x}>0,\\
						z(t,0)=0,~~~~  &t >0.
				\end{cases}
			\end{aligned}
		\end{equation} 
		In view of $c_m=2\sqrt{f_m'(0)}=2\lambda^*$, the function $p(t,\hat{x})=e^{\lambda^* \hat{x}} z(t,\hat{x})$ solves
		$$
		\begin{cases}
			p_t=p_{\hat{x}\hat{x}},&t>0,\hat{x}>0,\\
			p(t,0)=0,&t>0,
		\end{cases}
		$$
		 hence
		$$
		z(t,\hat{x})=\frac{e^{-\lambda^*\hat{x}}}{\sqrt{4\pi t}}
		\int_{0}^{+\infty}
		\left( e^{-\frac{(\hat{x}-\upsilon)^2}{4t}}-e^{-\frac{(\hat{x}+\upsilon)^2}{4t}} \right) p(0,\upsilon) d\upsilon~~~~ \text{ for all } t>0 \text{ and } \hat{x} \ge 0,
		$$
		which implies that
		$$
		z(t,\hat{x}) \sim C \hat{x} e^{-\lambda^* \hat{x} -\frac{\hat{x}^2}{4t}}t^{-\frac{3}{2}} \text{ as } t \rightarrow +\infty
		$$
		in the interval $\hat{x} \in[0,\sqrt{t}]$, where $C$  depends only on $p(0,\cdot)$.
		
			\noindent
		{\it Step 2.} Lower bound at $x=-(c+c_m)t-O(\sqrt{t})$.
		In view of $f_m\in C^2([0,1])$, there exists $M>0$ so that 
		$$
		f(s) -f_m'(0)s \geq -Ms^2 \text{ for } s \in[0,s_0)~~~~~ \text{ for some } s_0>0.
		$$
		Notice that $z(t,\hat{x})\leq C(t+1)^{-\frac{3}{2}}$ for $t >0$. Let  $a(t)$ solve
		$
		a'(t)=-CM(t+1)^{-\frac{3}{2}}a(t)^2,~t>0.$ Such $a(t)$ can be chosen uniformly bounded from above and below: $0<a_0<a(t)<a_1<+\infty$. 
		Then, the function $\underline{w}(t,\hat{x})=a(t) z(t,\hat{x})$ satisfies
		\begin{align*}
			\underline{w}_t- \underline{w}_{\hat{x}\hat{x}}-c_m \underline{w}_{\hat{x}}  -f_m(\underline{w})
			=&a'(t)z +a(t)\big(z_t-z_{\hat{x}\hat{x}}-c_m z_{\hat{x}}-f_m'(0) z\big)+f'(0)a z-f(az)\\
			\le &a'(t)z+M(az)^2=\big(a'(t)+Ma(t)^2z\big)z=0.
		\end{align*}		
		 Therefore, for any $\sigma>0$, there exists $\hat{\delta}>0$ such that 
		\begin{equation}\label{equ:u:lambda_m}
			u(t,-(c+c_m)t-\hat{x}-L)=w(t,\hat x)\geq \underline{w}(t,\hat{x})\geq \hat{\delta} \hat{x} e^{-\lambda^* \hat{x}} t^{-\frac{3}{2}}
		~~~~	\text{ for all } t\ge 2 
			\text{ and } \hat{x} \in [0,\sigma\sqrt{t}].
		\end{equation}
		
			\noindent
		{\it Step 3.} The approximate traveling fronts are subsolutions of \eqref{1.1'} for $-(c+c_m)t - O(\sqrt{t})-L\le x \le \mu_0 t-L$ with any fixed $-(c+c_m)<\mu_0<  \min(0,c_m-c)$. Fix $\sigma >0$ and let $\xi(t) = \sigma \sqrt{t}$. By the estimate \eqref{equ:u:lambda_m}, we will construct an explicit subsolution of \eqref{1.1'} on the interval $-(c+c_m)t -\xi(t)-L\le x \le \mu_0 t-L$ for $t$ large enough. This subsolution will be an approximate traveling front, moving with leftward speed $c_m+c>0$.
		
		According to \eqref{equ:u:lambda_m}, there exist $\tilde{\delta}>0$ and $T_1\ge 0$ such that
		\begin{equation}\label{equ:u:xi}
			u(t,-(c+c_m)t -\xi(t)-L) \ge \underline{w}(t,\xi(t)) \ge \tilde{\delta} \xi(t) e^{-\lambda^* \xi (t)} t^{-\frac{3}{2}}
		\end{equation}
		for all $t \ge T_1$. It follows from Theorems \ref{thm_c>cm}--\ref{thm:-cm_cm:left} that
		$$
		\inf_{\mu_1 t \le x \le \mu_2 t} u(t,x)  \rightarrow 1~~~~ \text{ for any }~ -(c_m+c) <\mu_1 <\mu_2<  \min(0,c_m-c).
		$$
		Given any $\varrho\in(0,1)$, fix $\overline\varrho\in(\rho,1)$. Therefore, there is $T_2(\ge T_1)$ such that $u(t,\mu_0t-L) \ge \overline\varrho$ for all $t \ge T_2$.
		
		Let $f_1$ be a $C^1$ function such that $f_1 \le f_m$ for $u \in [0,\overline\varrho]$, $f_1(0)=f_1(\overline\varrho)=0$, $f_1'(0)=f_m'(0)$ and $f_1(s)>0 \text{ on } (0,\overline\varrho)$. The function $f_1$ then satisfies
		$$
		f_1(s)\le f_m(s) \le f_m'(0) s =f_1'(0)s~~~~~\text{for all}~s\in [0,\overline\varrho].
		$$
		 Then there exists a traveling front $U_{c_m}(x-c_mt)$ of $u_t=u_{xx}+f_1(u)$ such that $0<U_{c_m}<\overline\varrho$ in $\mathbb{R}$, $U_{c_m}(-\infty)=\overline\varrho$, $U_{c_m}(+\infty)=0$ with speed $c_m=2 \sqrt{f_m'(0)}$. The profile $U_{c_m}$ is decreasing in $\mathbb{R}$ and is such that 
		$$
		U_{c_m}(s) \sim \tilde{B} s e^{-\lambda^* s} \text{ as } s \rightarrow+ \infty,~~~~~\text{for some}~\tilde{B}>0.
		$$
		Let now $\gamma>0$ and fix $x_1 \in \mathbb{R}$ large enough so that $\tilde{B}(\gamma+1) e^{-\lambda^* x_1} \le \tilde{\delta}$. Since there exists $T_3 \ge T_2$  such that 
		$$
		\frac{3}{2 \lambda^*} \ln t +x_1 < \gamma  \xi (t) ~~~\text{ for } t \ge T_3,
		$$
		we have 
		\begin{equation}\label{equ:U:c_m}
			U_{c_m}\Big(\frac{3}{2\lambda^*} \ln t + \xi(t)+x_1\Big) \le \tilde{\delta} \xi (t) e^{-\lambda^* \xi(t)}t^{-\frac{3}{2}},~~~~~t\ge T_3.
		\end{equation}

		On the other hand, in view of $U_{c_m}(+\infty)=0$ and $\min_{x \in [-(c+c_m)T_3-\xi(T_3)-L,\mu_0T_3-L]}u(T_3,x)>0$, there exists $x_2 (\ge x_1)$ such that
		$$
		U_{c_m}\Big(-x -(c+c_m)T_3+\frac{3}{2 \lambda^*} \ln T_3 +x_2-L\Big) \le u(T_3,x) 
		$$
		for all $x \in [-(c+c_m)T_3-\xi(T_3)-L,\mu_0T_3-L]$.
		Define the subsolution $\underline{u}$ as follows
		$$
		\underline{u}(t,x)=U_{c_m}\Big(-x -(c+c_m)t +\frac{3}{2 \lambda^*} \ln t +x_2-L\Big)
		$$
		for $t \ge T_3$ and $x \in [-(c+c_m)t-\xi(t)-L,\mu_0t-L]$.
		
		It is easy to verify that $\underline{u}(T_3,x) \le u(T_3,x)$ for all $x \in [-(c+c_m)T_3-\xi(T_3)-L,\mu_0T_3-L]$. Since $x_2 \geq x_1$ and $U_{c_m}$ is decreasing, along with  \eqref{equ:u:xi}--\eqref{equ:U:c_m}, we have 
		$$
		\underline{u}(t,-(c+c_m)t-\xi(t)-L)=U_{c_m}\Big(\xi(t) +\frac{3}{2 \lambda^*} \ln t +x_2\Big) \le \tilde{\delta} \xi(t) e^{-\lambda^* \xi(t)} t^{-\frac{3}{2}}  \le u(t,-(c+c_m)t-\xi(t)-L)
		$$
		for all $t \ge T_3 \ge T_1$. Besides, $\underline{u}(t,\mu_0t-L) < \overline\varrho \le u(t,\mu_0t-L)$ for all $t \ge T_3 \ge T_2$.
		
		Lastly, since $f_1 \le f_m$ in $[0,\overline\varrho]$ and since $U_{c_m}$ is decreasing and satisfies $U_{c_m}^{''}+c_mU_{c_m}^{'}+f_1(U_{c_m})=0$ for all $ t\ge T_3$ and $x \in[-(c+c_m)t-\xi(t)-L,\mu_0t-L]$, we get
		$$
		\underline{u}_t -\underline{u}_{xx}-c\underline{u}_x -f_m(\underline{u})
		=\Big(-(c+c_m)+\frac{3}{2\lambda^*t}\Big)U_{c_m}'(\eta) -U_{c_m}^{''}(\eta)+cU_{c_m}^{'}(\eta)-f_1(U_{c_m}(\eta))=\frac{3}{2\lambda^*t}U_{c_m}'(\eta) \le0
		$$
		where $\eta=-x -(c+c_m)t+ 3/(2\lambda^*) \ln t +x_2-L$.
		Therefore, the function $\underline{u}$ is a subsolution of \eqref{1.1'} for all $t\ge T_3$ and $x \in [-(c+c_m)t-\xi(t)-L,\mu_0t-L]$. The comparison principle yields that
		\begin{equation}\label{equ:lower:u}
			u(t,x) \geq \underline{u}(t,x)=U_{c_m}\Big(-x -(c+c_m)t +\frac{3}{2 \lambda^*} \ln t +x_2-L\Big)
		\end{equation}
		for all $t \ge T_3$ and $x \in [-(c+c_m)t-\xi(t)-L,\mu_0t-L]$.
		
		\noindent
		{\it Step4.} Conclusion of the proof. The inequality \eqref{equ:lower:u} implies that for any given $x' \in \mathbb{R}$
		$$
		u\Big(t,-(c+c_m)t + \frac{3}{2 \lambda^*} \ln t +x'\Big) \ge \underline{u}\Big(t,-(c+c_m)t + \frac{3}{2 \lambda^*} \ln t +x'\Big)=U_{c_m}(-x'+x_2-L)>0
		$$ 
		for $t$ large enough. This proves that $E_\varrho^-(t)\le -(c_m+c) t+3/(2\lambda^*)\ln t+C_2 $ for some $C_2\in\R$. This completes the proof of Theorem \ref{thm_log delay}.
	\end{proof}

	In addition, assume that $c>-c_m$, it follows from the lines of \cite[Theorem 1.2]{HNRR2013}, one can further show that the solution $u$ of \eqref{1.1'} approaches the family of shifted traveling waves $\varphi_{c_m}(-x-(c_m+c)t+3/(2\lambda^*)\ln t+\xi(t))$ uniformly for $x<[\min(0, c_m-c)-\varep ]t$ for $\varep>0$ small enough, where $\xi:(0,+\infty)\to\R$ is such that $|\xi(t)|\le C$ with some $C>0$, and that the solution $u$ converges along its level sets to the profile of the minimal traveling wave, for which we refer readers to \cite{HNRR2013} and will not give further details.

\end{document}